\documentclass[12pt]{article}
\usepackage{amsfonts}
\usepackage{hyperref}        % I activated this in ymh33h on April 7, 2010.
\usepackage{graphicx}
\usepackage{amsmath}
\usepackage{amssymb}

\newtheorem{theorem}{Theorem}[section]

\newtheorem{corollary}[theorem]{Corollary}

\newtheorem{definition}[theorem]{Definition}
\newtheorem{example}[theorem]{Example}

\newtheorem{lemma}[theorem]{Lemma}

\newtheorem{notation}[theorem]{Notation}

\newtheorem{proposition}[theorem]{Proposition}
\newtheorem{remark}[theorem]{Remark}

\newenvironment{proof}[1][Proof]{\textbf{#1.} }{\ \rule{0.5em}{0.5em}}

\def \L{\Lambda}

\def \<{\langle}
\def \>{\rangle}
\def \l{\lambda}
\def \R{\mathbb R}

\def \A{{\mathcal A}}
\def \D{{\mathcal D}}
\def \H{{\cal H}}

\def \H^0{{\cal H}^0 or}

\def \G{{\mathcal G}}

\def \P{{\mathcal P}}
\def \V{{\mathcal V}}

\def \w{\omega}

\def \kf{\frak k}
\def \p{\partial}
\def \beq{\begin{equation}}
\def \eeq{\end{equation}}

\def \en{{\bf n}}

\def \n{\nabla}

\def \eref{\eqref}

\def \lrc{\lrcorner}
\numberwithin{equation}{section}

\begin{document}

%-------------------   following is the topmatter from JEMS DGS1.
%      I inserted it on 2/13/10. It works with \documentclass[12pt]{article}.
%    This document class produces a nice table of contents. The amsart
%   document class doesn't.

\title{The Yang-Mills heat semigroup on three-manifolds with boundary.
 }
\author{
Nelia Charalambous
\thanks{Research was supported in part by NSF Grant DMS-0072164,
 NSF Grant DMS-0223098 and by CONACYT of Mexico.}\\
Department of Mathematics, \\
Instituto Technol\'ogico Aut\'onomo de M\'exico, \\
M\'exico \\
{\tt nelia.charalambous@itam.mx}\\ \\
Leonard Gross\\
Department of Mathematics\\
Cornell University\\
Ithaca, NY 14853-4201\\
{\tt gross@math.cornell.edu}\\ \\
}

%-------------------------------------------------------------------
%2010 AMS Subject Classification for YMH
% Primary:
%       35K58       Semilinear parabolic equations
%       35K65       Degenerate parabolic equations

% Secondary:
%      70S15   Yang-Mills and other gauge theories
%     35K51   Initial-boundary value problems for second-order parabolic systems
%       58J35       Heat and other parabolic equation methods
%-------------------------------
% Other possible categories:
% X       35K40   Second-order parabolic systems
%                             No. Too general sounding.
% X        35K55  Nonlinear parabolic equations.
%                          No. Ours is just semi-linear.
%------------------------------------------------------------------

%\date{\today \ \emph{File:\jobname{.tex}}}
\date{\today}

\maketitle

\newpage
\begin{abstract} Long time existence and uniqueness of solutions
to the Yang-Mills heat equation is proven over a compact 3-manifold
with smooth boundary. The initial data is taken to be a
Lie algebra valued connection form in the Sobolev space
$H_1$. Three kinds of boundary conditions are explored, Dirichlet type,
Neumann type and Marini boundary conditions. The last is a
nonlinear boundary condition, specified by setting the normal
component of the curvature to zero on the boundary. The Yang-Mills heat equation  is a weakly parabolic non-linear equation. We use a technique of
Donaldson and Sadun to convert it to a parabolic equation and then
gauge transform the solution of the parabolic equation
back to a solution of the original equation.
Apriori estimates are developed  by first establishing a gauge invariant
version of the Gaffney-Friedrichs inequality. A gauge invariant
regularization procedure for solutions is also established.
  Uniqueness holds upon imposition of boundary conditions on
  only two of the three components of the connection form
   because of weak parabolicity.
 This work is motivated by possible applications to quantum field theory.

\end{abstract}
\tableofcontents

%-----------------End topmatter

\newpage

\section{Introduction} \label{secIntro}

\subsection{Nonlinear distribution spaces.} \label{secIntro1}

Heat equations have been  used to characterize
 various function spaces by identifying these
 function spaces with  the
 initial data space for a parabolic equation. This method of characterizing
  function spaces  goes  back at least to the 1961 paper of
   Lions \cite{Lio}[Section 5],
  the 1960s papers, \cite{Tai1,Tai2,Tai3}, of Taibleson and
 to  the 1980s papers, \cite{Mat1,Mat2,Mat3}, of Matsuzawa.
 The  papers  of Matsuzawa  characterize an
 ultradistribution $u$ on a compact subset of $\R^n$ by properties of the
 solution to the heat equation with initial data $u$. See the classic book
 \cite{BB} for early work and the paper  \cite{AEO} for some recent history.

        By way of a simple example, consider a non-negative unbounded
   self-adjoint operator
  $A$ acting on a Hilbert space $H$. Assume for simplicity that $A \ge I$.
   Let $\alpha > 0$.
  The easily verified identity,
  \beq
  \| A^{-\alpha} u_0\|^2
  = C_\alpha  \int_0^\infty s^{2\alpha-1} \|e^{-sA} u_0\|^2 ds, \ C_\alpha = \text{constant}  \label{0.1}
    \eeq
  shows that the norm $ u_0 \rightarrow \|A^{-\alpha} u_0 \|$ on $H$
  can be characterized in terms of solutions to the initial value problem
  \beq
  u'(s) = -A u(s),\ \text{for}\ s >0, \ \  u(0) = u_0,     \label{0.2}
  \eeq
  since the solution is just $u(s) = e^{-sA} u_0$.
  In fact it is clear that the initial value problem \eref{0.2}
  sets up a one-to-one correspondence between  the space of those solutions
  of the equation  $u'(s) = -A u(s)$ for which the right side of \eref{0.1} is finite,
  and  the large initial data space consisting of  the  completion
   of $H$ in the norm   $\|A^{-\alpha} u_0 \|$.
    If $H$ is an $L^2$ space over some
    Riemannian manifold
    and $-A$ is a second order elliptic operator then these completed spaces
    are just negative Sobolev spaces and the correspondence
    $u_0 \leftrightarrow u(\cdot)$, set up by \eref{0.2}, identifies these
    Sobolev spaces with certain spaces of solutions of the heat equation for $-A$.
 In general $H$ may be some other kind of Banach space or Frechet space,
 and the   completion spaces  need not be Sobolev spaces, \cite{AEO}.

              Some quantum field theories seem to require use of
   large completions    of spaces which  are  not  linear spaces.
  Most important is the example     in which the space to be completed
   is a space $\A$ of connections on $\R^3$ modulo a gauge group $\G$.
    Whatever smoothness one imposes on $\A$ and $\G$, the space $\A/\G$
  is not a linear space in typical  cases of interest.
   See e.g. \cite{Si1,Si2,NR} for  discussions
   of the geometry of this space in case $\R^3$ is replaced by a
    compact manifold.

    The reason for the need to complete such a quotient space
   is that the quantum theory requires a space large enough to support
   certain measures of physical interest. Typically, the measures arising in quantum field theory need some negative Sobolev space to live on.
   In the preceding  example some kind of {\it nonlinear} negative
   Sobolev space seems to be required.
   We are going to explore the  (nonlinear) Yang-Mills heat equation
   as a replacement for the linear equation \eref{0.2}.
         The measure theoretic difficulties increase with spatial dimension
          as does the difficulty in proving existence of solutions to the
          Yang-Mills heat equation.
          For example no  completion is necessary
          for addressing the measure theory in one spatial dimension,
          even though study of the associated stochastic process
           presents severe problems of its own.
             See, e.g., A. Sengupta, \cite{Sen1,Sen2}.
           We are going to address the Yang-Mills heat equation in
            three space dimensions only.
            The corresponding existence and uniqueness theorems are
            simpler in two space dimensions and follow easily from our techniques.

         This paper is intended as a first step in constructing
      non-linear distribution spaces for Yang-Mills fields over
       three dimensional space.
 In contrast with the simple example of \eref{0.2}, the flow  equation
    associated  to such a nonlinear distribution space will be  itself nonlinear.
    In the case of a Yang-Mills field the natural equation is the gradient flow
     equation of the magnetic energy
     (which is the square of the $L^2$ norm  of the curvature).
     Elsewhere, the nonlinear sigma model will be investigated from
      this same point
     of view and the nonlinear equation will again be a gradient flow equation
     of a non-quadratic energy.
     Thus in each of the examples of interest  the flow  equation is a
   geometric flow given as the gradient flow of some natural energy
   functional on some non-linear manifold.
     It is the intention  of this program to realize the required non-linear
    distribution spaces     as   complete ``Riemannian'' infinite dimensional
     manifolds whose elements are geometric flows and which support
     genuine functions, such as gauge invariantly regularized Wilson loop
     variables.

                In order to understand the spaces of flows for which there is no
    identifiable initial data it is first necessary to understand those
   flows for which there is an identifiable initial value. Unlike the linear case
   a proper  understanding of the space of initial data  for some class of flows
   requires treating both the space of flows and the initial data space
   as infinite dimensional Riemannian manifolds: one needs to know
   not only which initial data propagates to a flow but also which
    variations of the     initial data propagate to a solution of the
     variational equation along the flow. In the linear
   case there is no distinction between the flow equation and its variational equation.  In the nonlinear case, when the initial data
   is singular,  the variational equation will have singular coefficients at
   time zero,
   and a variation of the initial data may not propagate past the singularity.
   This issue will be treated in a separate work.
       In the present paper we are going
   to prove existence and uniqueness of solutions to the Yang-Mills
    heat equation, \eref{0.5}, with initial data in Sobolev class 1,
    and also establish apriori bounds useful for extending the class
   of initial data to connection forms of Sobolev class 1/2.
   The latter appears to be the largest class of initial data for which
    variations propagate, and is also the natural  initial data  space
        from the point of  view of relativity theory.
           The extension, however, will  be made elsewhere.

%subsection
\subsection{Manifolds with boundary and local observables} \label{secIntro2}

            We are going to consider the Yang-Mills heat equation in
             a product bundle over a
 compact Riemannian 3-manifold $M$ with smooth boundary.
  The case of interest for quantum field theory is that in which $M$ is the closure
  of a bounded open set $O$ in $\R^3$ with smooth boundary.
    Roughly, our main theorem asserts
   that if $K$ is a compact, connected Lie group with Lie algebra $\kf$ and if
   $A_0$ is a  $\kf$ valued connection form over $M$,
    lying in the first order Sobolev space $W_1(M)$, then
    there exists a unique  solution  to the Yang-Mills heat equation
   \beq
   \p A(t)/\p t = - d_{A(t)}^* B(t),\ \ \ t >0\ \ \text{with}\ \    A(0) = A_0,     \label{0.5}
   \eeq
   satisfying Dirichlet type or Neumann type boundary conditions.
  Here $B(t)$ is the curvature 2-form, $B(t)= dA(t) + A(t)\wedge A(t)$,
   of the connection form $A(t)$ and
  $d_{A(t)}^*$ is the gauge covariant coderivative. Equation \eref{0.5}
  is the gradient flow equation for the magnetic energy $\|B\|_{L^2(M)}^2$.
  We are also going to examine a purely nonlinear boundary condition
  suggested by work of A. Marini in the context of nonlinear
  elliptic  boundary value problems for Yang-Mills connections over
   four dimensional manifolds, \cite{Ma3, Ma4,Ma7}.

        There is a fundamental conceptual reason for considering the  Yang-Mills
 heat equation over  a bounded open set  $O$ in $\R^3$ rather than
  over all of $\R^3$ or over a closed 3-manifold such as $T^3$:
 Suppose that $\gamma$ is a piecewise smooth closed curve in $\R^3$.
      Denote by $W_\gamma(A)$  the composition of a character of $K$
  with  the parallel transport  around $\gamma$  by a connection
   form $A$ defined in a neighborhood of $\gamma$.
      That is,  $W_\gamma(A) \equiv trace\ ( //_\gamma^A)$, where the trace
 is computed in some finite dimensional unitary representation of $K$ and
 $//_\gamma^A$ denotes parallel transport.

          Then the  holonomy function $A \mapsto W_\gamma(A)$
  (the Wilson loop variable)   is gauge invariant and descends
  to the quotient manifold $\A/\G$  discussed above.
          In the sought for space of connection
   forms, on whose moduli space the desired measure lives, a typical
   connection form $A$ is not even an almost everywhere defined form,
   let alone  continuous,
   and the function $W_\gamma(A)$ is therefore not well defined.
          This is known from the electromagnetic case, $K = U(1)$, for which
   the measure theory is explicitly solvable.
         Nevertheless similar holonomy functions on $\A/\G$ have been used extensively
  both  for formulation of a mathematical theory  \cite{Si2}, \cite{Sei}[Chapter 8],
  and for computational comparisons with experiment \cite{LeP}.
         But if $A(\cdot)$ solves the Yang-Mills heat equation \eref{0.5}, with
  initial data $A_0$, which we take to be some kind of generalized
  connection form on $\R^3$,
   then, for any $t > 0$, $A(t)$ will be (essentially) a $C^\infty$
   1-form and the map  $A_0 \mapsto W_\gamma(A(t))$ will be
    well defined and gauge invariant.
          Thus the Yang-Mills heat equation provides a gauge invariant
   regularization  procedure for a connection form $A_0$,
   which is applicable even when $A_0$ is in some distribution  space.
              However, since the (weakly) parabolic equation \eref{0.5}
   propagates information with infinite speed, the map
   $A_0 \mapsto W_\gamma(A(t))$ depends on $A_0$ over all of $\R^3$.
          This is unsatisfactory from the point of view of local quantum field
  theory, which requires use of ``local observables'', \cite{GJ}, \cite{SW},
  that is, functions of $A_0$ which depend only on the  behavior
  of $A_0$ in some specified (say bounded) open set $O \subset \R^3$.
           Now solving equation \eref{0.5}
      over $O$ with initial data $A_0|O$ produces a function
      $W_\gamma(A(t))$ depending only on $A_0|O$,
      when $\gamma \subset O$.
            In this way we expect to construct useful ``local observables''.

                     We anticipate that the conventional lattice regularization
    of Yang-Mills quantum field theory, \cite{Wil,KS,Sei,G7}, will mesh well
     with the present     continuum regularization.

%subsection
\subsection{Technical description and history} \label{sec1.3}

The Yang-Mills heat equation has a long history
\cite{AB, ChSh, Do1, Sa, Do2,Ra,Has,Ho, HT2,HT1}.
While most of these works were aimed at immediate
 application in mathematics, some, e.g. \cite{Sa},
    were aimed primarily at application to physics.

 Standard methods for proving existence and uniqueness for
  nonlinear parabolic equations do not seem applicable to
  equation \eref{0.5} because the
 equation is only weakly parabolic and  the functional
 $ A \mapsto \| d A +A\wedge A \|_{L^2(M)}^2$, whose flow we
  are following, is not  (even weakly) convex.
       We are going to adapt a method that seems to have its origin
  in papers of   Zwanziger, \cite{Z},  Donaldson, \cite{Do1}, and
  Sadun, \cite{Sa}.  This consists in adding a term $-d_A d^*A$ to the right side
  of \eref{0.5}, which  makes the equation parabolic.
   A time dependent gauge transformation can then be constructed
 which  changes the solution of the modified equation into a solution
  of the original equation,  \eref{0.5}.
      Zwanziger first added  such a term into the stochastic evolution equation
  for a  quantum field theory, \cite{Z}.
   Donaldson, \cite{Do1}, independently
  added such a term to the evolution equation of a classical Yang-Mills
   heat equation and similarly ``gauged it away''. L. Sadun, motivated by
   Zwanziger's work,  used this technique in proving existence of solutions
  to \eref{0.5} over $\R^3$ in his Ph. D. thesis, \cite{Sa},  as a step in
  carrying out stochastic quantization for Yang-Mills fields.
  See also   the book \cite [Section 6.3]{DK}  for further exposition
   of Donaldson's method.

   Our proofs depend on establishing  apriori estimates for solutions
   of \eref{0.5}. There are two kinds of apriori estimates,
   both  based on energy estimates for various gauge
 invariant derivatives  of the curvature.
        One type of estimate is  based on the
   assumption that the initial data has finite action
   (loosely equivalent to $A_0 \in H_{1/2}$) and the other
   on the assumption that the initial data has finite energy
   (loosely equivalent to $A_0 \in H_1$).

       The proof of the energy estimates is based on a re-expression of
       Sobolev's inequality for $H_1$ functions in terms of  the
        gauge invariant exterior    derivatives $d_A$ and $d_A^*$ instead
        of the  gauge invariant Riemannian gradient $\n^A$.
   For real valued forms the key inequality relating these two kinds of estimates
   is the Gaffney-Friedrichs inequality \cite{Ga1,Fr,ME56,Mt01, Tay1}. In our case we need
   gauge covariant derivatives and for this purpose we will establish
   a gauge invariant version of the Gaffney-Friedrichs inequality.
   Not surprisingly, the curvature of the connection form $A$ enters
   in a substantial way and contributes to some of the technical
    problems to be resolved.

       J. R{\aa}de,  \cite{Ra}, has proven existence of solutions for the
  Yang-Mills heat equation on a closed 3-manifold and investigated
   the longtime behavior of the solutions.
  The method used by R{\aa}de to solve the problem of lack of
  parabolicity is quite different from the method of Donaldson and Sadun.
   The curvature, $F_A$,  of the 1-form $A$ is taken as
  an unknown, $L$,  independent  of $A$, and a joint system of equations
  for $A$ and $L$  is solved. The joint system is parabolic.
  R{\aa}de    proved
  that the solution $L(t)$ agrees with $F_{A(t)}$  for all time
   if they agree at time zero.
  This method seems to go back to
   Ginibre and Velo, \cite{GV1,GV2}, in the context of the hyperbolic
      Yang-Mills equations and  to  De Turck, \cite{DeT}, in the
  context of the parabolic Ricci flow problem.   This method might offer some advantages in our circumstance. But
  the presence of boundary conditions seems to add considerable
   difficulty.

   The transition from short time existence to long time existence is carried out
   in different ways in the various works \cite{Do1, Do2,Ra,Sa} and in
    the present paper. In addition, semi-probabilistic methods have
    also been used:  See, e.g.,  Arnoudon et al, \cite{ABT},
     and  Pulemotov, \cite{Pu},
    for a very different approach to long time existence.

%Section
\section{Statement of results}   \label{secstate}

\begin{notation} \label{not2.1}{\rm $M$ will denote a compact
Riemannian 3-manifold with smooth boundary.  $K$ will denote
a compact connected Lie group. Without loss of generality we may
 and will identify $K$ with a subgroup of the orthogonal group, respectively unitary group,  of some finite dimensional real, respectivley complex, inner product space $\V$. Thereby the Lie algebra of $K$, denoted $\frak k$, is a
 real  subspace of $End\ \V$. We will be concerned only with  a product bundle $M\times \V \rightarrow M$ over $M$. We assume given
 an  $Ad\ K$ invariant inner product $\<\cdot, \cdot\>$ on $\frak k$ with norm
  denoted by $|\xi|_{\frak k}$  for $\xi \in \frak k$.

 If $\w$ and $\phi$ are $\frak k$ valued p-forms  define
 $(\w, \phi) = \int_M\<\w(x), \phi(x)\>_{\L^p\otimes \frak k} dx$
 and $\|\w \|_2^2 = (\w, \w)$.
 Define also
 $\|w\|_\infty = \sup_{x \in M}|\w(x)|_{\L^p \otimes \frak k}$ and
 \beq
 \| \w \|_{W_1(M)}^2
 = \int_M |\n \w|_{\Lambda^p\otimes\frak k}^2 d\, \text{Vol}\ \ + \| \w \|_2^2   \label{ymh2}
 \eeq
 where $\n$ is the Riemannian gradient on forms. Define
 $W_1 = W_1(M) = \{ \w: \|\w\|_{W_1(M)} <\infty\}$.
            The notation $H_1$ will be used later for forms in $W_1$ which
     satisfy specified boundary conditions.
 Since we are concerned only with a product bundle, a connection form
  can be identified with a $\frak k$ valued 1-form.
  For a connection form $A$, given in local coordinates by
  $A = \sum_{j=1}^3 A_j(x) dx^j$, its curvature (magnetic field) is given by
   \beq
 B = dA +(1/2) [A\wedge A]                                \label{ymh3}
 \eeq
 where $[A\wedge A] = \sum_{i,j} [A_i, A_j] dx^i\wedge dx^j$ and
 $[A_i(x), A_j(x)]$
 is the commutator in $\frak k$. $B$ is a   $\frak k$ valued 2-form.
 For $\w \in W_1$
  we define  $d_A \w = d \w + (ad\ A) \wedge \w$ and
  $d_A^* \w = d^*\w + (ad\ A\wedge)^* \w$. No boundary conditions
   are implied on these operators in this section.
  The domains of these operators will be discussed further in Section \ref{secDN}.
 }
 \end{notation}

\begin{definition}\label{defstrsol} {\rm Let $0 < T \le \infty$.
By a {\it strong solution} to the Yang-Mills heat
equation over $[0, T)$
we mean a continuous function
\beq
A(\cdot): [0,T) \rightarrow W_1 \subset \frak k\text{-valued 1-forms}
\eeq
such that
\begin{align}
&a) \ B(t) \in W_1 \  \text{for each}\ \  t\in (0,T),                    \label{ymh8}\\
&b) \ \ \text{the strong $L^2(M)$ derivative $A'(t) \equiv (d/dt) A(t)$}\notag \\
& \qquad   \qquad    \text{exists for each} \    t\in (0, T),        \label{ymh9}\\
&c) \  A'(t) = - d_{A(t)}^* B(t)\ \ \text{for each}\ t \in(0, T),     \label{ymh10}\\
&d) \  \| B(t)\|_\infty\ \text{is bounded on each
                             bounded interval $ [a,b) \subset (0, T)$,}  \label{ymh11}\\
&e) \  t^{3/4} \| B(t)\|_\infty\
           \text{ is bounded on some interval $(0, b) \subset (0, T)$.} \label{ymh12}
\end{align}
}
\end{definition}

\begin{remark} {\rm The condition e) allows the degree of singular behavior near
$ t =0$ that is to be expected in three dimensions.
It will be shown in a separate work
that, when $M$ is convex,  conditions d) and e) follow
 from a), b) and c). Usually $A'(t)$ will signify $\p A(t)/ \p t$. But in b)
 we are regarding $A(\cdot)$ as a function into $L^2(M; \L^1\otimes \frak k)$.
}
\end{remark}

\subsection{Dirichlet, Neumann and Marini boundary conditions}

\begin{notation}\label{notnorm}{\rm (Tangential and normal components.)
At a point $x\in \p M$ denote by $\en$
 the outward drawn unit normal and by $\nu$ the dual unit conormal. Any p-form $\gamma$ over $T_x(M)$ can be written uniquely as
 $\gamma = \alpha\wedge \nu + \beta$ where $\beta(\en, X_1, \dots, X_{p-1}) = \alpha(\en, X_1,\dots, X_{p-2}) =0$ for all $X_j \in T_x(M)$.
     As is customary, we will write
 $\gamma_{norm} = \alpha \wedge \nu$ and $\gamma_{tan} = \beta$.
 The restriction maps $\alpha\rightarrow i^* \alpha$ and
 $\beta \rightarrow i^*\beta$ are clearly isomorphisms on
  these classes of forms when
 $i:T_x(\p M)\rightarrow T_x(M)$ is the inclusion map. Clearly
 $\gamma_{tan} =0$ if and only if $\gamma\wedge \nu =0$.
  A coordinate based description of these two components of
  $\gamma$ will be given in Section \ref{secGFS}.
 }
 \end{notation}

\begin{theorem} \label{thm1N} $($Neumann boundary conditions.$)$
Suppose that $A_0 \in W_1$ and
$(A_0)_{norm} =0$. Then there is a strong solution $A(\cdot)$ over $[0, \infty)$ such that $A(0) = A_0$ and that satisfies the boundary conditions
\begin{align}
&i)\ \ \  A(t)_{norm} =0\ \ \text{for}\ t \ge 0 \ \text{and}    \label{N1}\\
&ii)\ \ B(t)_{norm}=0\ \   \text{for}\  t >0.         \label{N2}
\end{align}
Uniqueness: If $A_1$ and $A_2$ are two strong
solutions  which agree at time zero and satisfy \eref{N2}
then $A_1 = A_2$ on $[0, \infty)$.
\end{theorem}

\begin{remark}{\rm  Notice that for uniqueness the condition \eref{N1}
is  not required, even for $t =0$.
For an explanation of the terminology ``Neumann boundary conditions''
for the pair of conditions \eref{N1} and \eref{N2}
see Remark \ref{remU3}.
}
\end{remark}

\begin{theorem}\label{thm1D} $($Dirichlet boundary conditions.$)$
Suppose that $A_0 \in W_1$ and $(A_0)_{tan} =0$.
 Then there is a strong solution over $[0, \infty)$ such that
 $A(0) = A_0$ and that satisfies the boundary conditions
  \begin{align}
 &i)\ \ \  A(t)_{tan} =0\ \ \ \text{for all}\ t \ge 0 \ \text{and}    \label{D1}\\
  &ii)\ \  B(t)_{tan} =0 \ \ \ \text{for all}\ t >0.                 \label{D2}
 \end{align}
  Uniqueness: If $A_1$ and $A_2$ are two strong
  solutions which agree at time zero and satisfy \eref{D1}
 then $A_1 = A_2$ on $[0,\infty)$.
 \end{theorem}

 \begin{remark}{\rm Notice that for uniqueness the conditions $ B_j(t)_{tan} =0 ,
  t >0$ are not required. In fact  $A(t)_{tan}=0$ implies $B(t)_{tan}=0$.
   (See, e.g., \eref{DN72}). So the latter is not an independent condition.
   }
   \end{remark}

   \begin{remark} \label{remreg}{\rm  (Weak parabolicity and regularization.)
  Suppose that $g \in C^2(M;K)$ and is the identity in a neighborhood
  of $\p M$. Let $A_0 = g^{-1} dg$.
  Then $A_0 \in C^1(M:\L^1\otimes\frak k) \subset W_1$ and is zero in a neighborhood of $\p M$.
  Define $A(t) = A_0$ for all $t\ge 0$.
  $A(t)$ has curvature zero and
   satisfies   all of the Neumann and Dirichlet boundary conditions,
   \eref{N1}, \eref{N2}, \eref{D1} and  \eref{D2}, including the initial conditions.
    It  is the unique strong solution  specified in Theorems
     \ref{thm1N} and \ref{thm1D}.
    Thus the Yang-Mills heat equation does not regularize all initial data,
    reflecting the well known fact that it is only weakly parabolic.
    The weak parabolicity  will be   particularly visible in
    equation \eref{ST13} and the discussion following it.
  There is a gain of regularity for the curvature, however, and this will
    allow  the strong sense of solution specified  in Definition \ref{defstrsol}.
    Nevertheless, for $t >0$, the curvature $B(t)$ itself will not be smooth
    under our initial conditions. For example if $g$ is as above and $A_0$
    is any initial condition in $W_1(M)$ then the gauge transform $A_0^g$
    is also in $W_1$ while $B^g(t)(x) = g(x)^{-1}B(t) g(x)$,
    which will not be smooth even if $B(t)$ is smooth.
}
\end{remark}

 % Remark U1
 \begin{remark}\label{remU1} {\rm (Weak parabolicity and uniqueness.)
 Theorems \ref{thm1N} and \ref{thm1D}
 show that, for both Dirichlet and Neumann type
  boundary conditions, uniqueness follows from the imposition of only
  two boundary conditions on the three component connection form $A(t)$.
    This effect can   be attributed to the fact that the Yang-Mills heat equation
  is only weakly parabolic. It is well known that degeneracy of
  an elliptic operator $L$ on a manifold with boundary  can force
   uniqueness on solutions of the  weakly parabolic equation
  $\p u/\p t = Lu$ under fewer boundary conditions on $u$ than usual.
  See \cite[Section 7.2]{MeMu}  for a recent work discussing
   this issue for scalar functions.
  }
  \end{remark}

 %Remark U3
 \begin{remark}\label{remU3} {\rm (Neumann and Marini boundary conditions.)
 In Theorem \ref{thm1D} the boundary condition
  $A(t)_{tan}=0$, $t \ge 0$, appears in both the existence and uniqueness
  portion of the theorem, whereas in Theorem \ref{thm1N} the
  initial boundary condition $(A_0)_{norm} =0$ is needed for the existence
  proof while  $A(t)_{norm}=0, t >0$ is not needed  for  uniqueness.
  If $A(t)_{norm} =0$ for $t >0$  then $[A(t)\wedge A(t)]_{norm}=0$ and consequently
  $B(t)_{norm} = (dA(t))_{norm}$. Thus in  the presence of \eref{N1}
   the nonlinear boundary condition
  $B(t)_{norm} =0$ in \eref{N2} is equivalent to the pure
     Neumann boundary condition $(dA(t))_{norm} =0$.

          A. Marini, \cite{Ma3,Ma4,Ma7}, has explored the nonlinear
      boundary condition $F_{norm} =0$ in the context of the weakly elliptic
     boundary value problem $d_A^* F =0$, where $F= F_A$
  is the curvature of a connection $A$ over a 4-manifold with boundary.
    In the context of Theorem \ref{thm1N}, the corresponding
  Marini boundary condition,
  $B(t)_{norm} =0$, is fully gauge invariant and does not depend on
   the choice of a fiducial gauge, unlike  the pair of conditions
   $ A_{norm} = 0, \ (dA)_{norm}=0$,
  to which the pair of equations \eref{N1} and \eref{N2} is equivalent.

   The Marini boundary condition  will ultimately be the case
    of interest for the intended application to quantum field theory.
   Theorem \ref{thm1N} easily yields the following existence and uniqueness
   theorem with the pure nonlinear boundary condition
   $B(t)_{norm} =0$ by itself.
   The restrictive regularity of the initial data will be removed in a later work.
   }
   \end{remark}

      \begin{theorem} \label{thm1M} $($Marini boundary conditions.$)$
 Suppose that $A_0\in C^2(M; \Lambda^1\otimes \kf)$. Then
 there is a unique strong solution over $[0,\infty)$ such that
 $A(0) = A_0$ and
 \beq
 B(t)_{norm}=0\  \text{for}\ \ t >0.         \label{M1}
   \eeq
   \end{theorem}

   Theorems \ref{thm1N}, \ref{thm1D} and \ref{thm1M}
    will be proven in Section \ref{secLTE}.

\noindent
\subsection{The method of  Donaldson and Sadun}

            In order to prove the existence of solutions to \eref{ymh10} we
 are first going  to add a  gauge symmetry breaking
term to the equation \eref{ymh10} and prove short time existence
 and uniqueness of solutions for the modified equation.
 Let us write $C(t)$ for a time dependent  $\frak k$ valued 1-form on $M$
 which satisfies the initial value problem
\beq
(\p/\p t)C = -(d_C^* B_C + d_C d^*C), t >0, \ \
         C(0) = A_0 .                                             \label{ST11}
\eeq
along with one of the following two kinds of boundary conditions, (N) or (D).
\begin{align}
(N)\ \  C(t)_{norm}=0\
    &\text{for}\ t \ge 0, \ \ (B_{C(t)})_{norm}\  =0 \ \ \text{for}\ t >0  \label{ST11N} \\
(D)\ \    C(s)_{tan} =0\ \ \
    &\text{for}\ t \ge 0,\ \ (d^*C(t))|_{\p M} = 0\ \ \text{for}\ t >0. \label{ST11D}
\end{align}

The equation \eref{ST11}  is a strictly parabolic differential equation,
    unlike \eref{ymh10}.
 The boundary conditions (D) are relative boundary conditions
 in the sense of  Ray and Singer, \cite{RaS}, while,
 in view of Remark    \ref{remU3},
 the boundary conditions (N) are equivalent to absolute boundary conditions.
                 For  recent systematic discussions of absolute and relative
   boundary conditions for real valued forms see the book
   \cite[Chapter 5, Section 9]{Tay1}
      and \cite{MMT01}, especially Chapter 5.

    In Section \ref{secST} we are going to use a quadratic form version
    of these boundary conditions.
    We will prove existence of solutions to these initial-boundary value
    problems as in Theorem \ref{thmpara} below, and then,
     roughly speaking,  we will  construct  a function
$g: [0,T) \rightarrow C^1(M ; K)$
 for which the gauge transform $A \equiv C^g$ satisfies the
 Yang-Mills heat equation, \eref{ymh10},
 together with either the Neumann type  boundary conditions
 \eref{N1}, \eref{N2} or the Dirichlet type boundary
 conditions \eref{D1}, \eref{D2}.
 For the resulting solution $A(\cdot)$,  the relative and absolute boundary
  conditions partly disappear.
 For Marini boundary conditions they disappear completely.

     Here is an informal description of the gauge transform procedure
     of Donaldson, \cite{Do1},  and Sadun, \cite{Sa}.
           A precise version will be given in Theorem \ref{thmSTE3}.

  \begin{lemma}\label{lemDS}  Let $C(t)$ be a solution to \eref{ST11} with
  boundary conditions \eref{ST11N}, respectively \eref{ST11D}.
   Define a function
  $ g: [0,T) \rightarrow C^\infty(M;K)\subset C^\infty(M; End\ \V)$
   as the solution to the initial value problem
  \beq
  g'(t, x) g(t,x)^{-1} = d^* C(t,x),\ \ g(0,x) = I_\V         \label{ST30}
  \eeq
   for each $x \in M$. Let $A = C^g$. That is,
   $ A(t,x) = g(t,x)^{-1} C(t,x) g(t,x) + g^{-1} dg $.
   Then $A$ solves \eref{ymh10} with the boundary conditions
    \eref{N1}, \eref{N2}, respectively \eref{D1}, \eref{D2}.
    \end{lemma}

    Actually, because of the singular behavior of $d^*C(t,x)$ as
     $t \downarrow 0$ it is difficult to establish the regularity
      of $g(t, x)$ needed to ensure that $A(t) \in W_1(M)$ for $t \ge 0$.
      We will instead define $g_\epsilon(t)$ for $t\ge \epsilon$
       using the same differential equation, \eref{ST30},  but with initial condition
       $g_\epsilon(\epsilon) = I_\V$.
       Defining $A_\epsilon(t) = C(t)^{g_\epsilon(t)}$ for $t \ge \epsilon$,
       we will then show that the connection forms $A_\epsilon(\cdot)$
       define smooth solutions which converge
        in a strong sense to the desired solution to \eref{ymh10}
         as $\epsilon \downarrow 0$.
         See Section \ref{secST3}   for precise statements and proof.

%Theorem  Para
\begin{theorem}\label{thmpara} Let $A_0 \in W_1$. Assume that
$(A_0)_{norm} =0$, respectively $(A_0)_{tan} =0$. Then there exists
$T >0$ and a continuous function $C:[0, T) \rightarrow W_1$
 such that   $C(0) = A_0$ and

a$)$ \  $B_{C(t)} \in W_1$ and $ d^* C(t) \in W_1$ for each $t \in (0,T)$,

b$)$ \  the strong $L^2(M)$ derivative $(d/dt) C(t)$ exists for each $t >0$,

c$)$ \ the equation \eref{ST11} holds for each $t >0$ along with
      the boundary conditions  \eref{ST11N}, respectively \eref{ST11D},

f$)$ \  \  $t^{3/4} \| B_{C(t)}\|_\infty$ is bounded on $(0, T)$.

      The solution is unique under the preceding conditions.
  Moreover, $C(\cdot)$ lies in $C^\infty((0,T)\times M; \L^1 \otimes \frak k)$.
   \end{theorem}

   This will be proved in Section \ref{secST1}.
   The proof  proceeds by a standard reduction
    to an integral equation and a contraction mapping argument,
     followed then by  a regularity theorem.
        However we  will also deduce in Section \ref{secST} some
    regularity directly from our  form of the contraction argument.

\begin{remark}\label{remU4} {\rm Our proof of uniqueness for the equation \eref{ST11}
requires use of all of the boundary conditions \eref{ST11N}, respectively
\eref{ST11D}. Each of these imposes three conditions at each point of the boundary. This is to be expected for a strictly parabolic equation.
This should be contrasted with the discussion in Remark \ref{remU1}.
}
\end{remark}

\subsection{Gauge invariant Gaffney-Friedrichs inequalities}

   %Gaffney-Friedrichs

Most of the estimates in this paper will depend on the use of
Sobolev inequalities in which the energy form
$\| \n^A \w \| _{L^2(M)}^2 + \| w \|_2^2 $ is replaced by the Hodge version,
$\| d_A \w\|_2^2 + \| d_A^* \w \|_2^2 + C \|\w\|_2^2$.
Here we have written
              $(\n^A)_j \w = \n_j \w +[A_j,\w]$
     for the gauge covariant gradient of a $\frak k$ valued form $\w$.
It will be necessary to establish equivalences between these two
energy forms because it is the former that controls $L^p$ norms via
Sobolev inequalities  while it is the latter that
relates well to the Yang-Mills heat equation.
The constant $C$ depends on the curvature of the connection form
$A$, and the nature of this
dependence is crucial for dealing with singular initial data.
It is this gauge-invariant Gaffney-Friedrichs inequality that
 will allow us to prove regularity  in the only directions in which
 regularity can occur. See Remark \ref{remreg}.

      In the classical case, i.e., real valued forms, such equivalences go back
   to Gaffney, \cite{Ga1}, and Friedrichs, \cite{Fr}. See also Eells and Morrey,
   \cite{ME56},  for a very  early work in this direction.
   The constants in these classical inequalities depend on the Riemannian
    curvature of $M$ and the curvature of its boundary. M. Mitrea, \cite{Mt01},
    has shown that such inequalities can be established with no
    dependence on the Riemannian curvature of $M$ and only mild dependence
    (convexity) on the curvature of the boundary in these classical cases.
    The benefit of using a convex domain for real valued forms
    was  observed  early on by   Saranen, \cite{Sar},
     for a convex domain  in $\R^3$.
    A reader may consult the book by Taylor, \cite[pages 361-364]{Tay1}
  for a recent derivation
    of the Gaffney-Friedrichs inequality in the classical case and \cite{MMT01}
    for extensions to nonsmooth Riemannian manifolds in the classical case.

           Our concern here is primarily with the dependence of the
     constant $C$  on the curvature of  the connection form  $A$.

% Theorem GF
    \begin{theorem}\label{thmGF}
    $($Gauge invariant Gaffney-Friedrichs inequality.$)$
          Suppose that $M$ is a compact Riemannian 3-manifold with smooth boundary and that $A$ is a $\frak k$ valued 1-form in $W_1(M)$
  with curvature $B$. Assume that $\|B\|_{L^3(M)} < \infty$.
 Let $\w$ be a    $\frak k$ valued p-form in $W_1(M)$ for which either
 \beq
 \w_{tan} =0\ \ \ \text{or}\ \ \ \w_{norm} =0.                    \label{gaf49}
 \eeq
   Then
        \begin{align}
(1/2)\{\|\n^A \w \|_{L^2(M)}^2 + \|\w \|_2^2\}
   \le  \|d_A \w \|_{L^2(M)}^2  + \|d_A^* \w \|_{L^2(M)}^2
   + \lambda_3 \| \w\|_2^2                                                   \label{gaf50}
   \end{align}
   holds with
\beq
\lambda_3=  \lambda_M + (\kappa c)^2 \| B \|_{L^3(M)}^2 .      \label{gaf51}
\eeq

  Here the constant $\lambda_M$  and the Sobolev constant $\kappa$
   depend only
   on the geometry of $M$ and not on $A$.
$d_A$ is the covariant exterior derivative with domain matching the boundary condition on $\w$
and $d_A^*$ is its adjoint. $\lambda_M = 1$ if $M$ is a convex
subset of $\R^3$.
The constant
 $c\equiv  \sup \{ \| ad\ x\|_{\frak k \rightarrow \frak k} : |x|_{\frak k} \le 1\}$
 measures the non-commutativity of $K$ and is zero if $K$ is commutative.
     \end{theorem}
     This will be proven in Section \ref{secGFS}.

 \begin{remark}{\rm A Gaffney-Friedrichs estimate of the above form
 will also be derived in which $\|B\|_{L^2(M)}$ enters (but quartically)  instead
 of $\| B\|_{L^3(M)}^2$. Both forms of the Gaffney-Friedrichs inequality
 will be useful for initial conditions of finite energy.
  }
  \end{remark}

  \begin{remark}\label{remcontdep}
  {\rm (Continuous dependence on initial data.)
  The solution to \eref{ST11} described in Theorem \ref{thmpara}
  is easily shown to depend continuously on the initial data $A_0$ in $W_1$ norm. So do the solutions in Theorems \ref{thm1N} and \ref{thm1D}.
  But the proofs for these two cases will be postponed to a later work in which the initial data space will be enlarged to include $H_{1/2}$ data.
  }
  \end{remark}

  \noindent
{\bf Apriori estimates.}

In order to carry out the transition from the parabolic equation \eref{ST11}
to the Yang-Mills equation \eref{ymh10} and to develop further properties
of the solution to \eref{ymh10} we will derive a number of apriori estimates
which have the form
\beq
t^\alpha \| D^n B(t) \|_p < \infty,\ \ 0 < t < T                   \label{ap20}
\eeq
for various gauge covariant derivatives $D^n$, with $ p \in [2,6]$
 and with corresponding    values of $\alpha \ge 0$.
  In this paper we will need to use only $n =0,1$.
   For most of these
apriori estimates we will assume that $\| B_0\|_{L^2(M)} < \infty$.
 This is the assumption of finite (magnetic) energy. It holds if $A_0 \in W_1(M)$.
 But for some estimates we will assume only the weaker condition
 \beq
 \int_0^1 s^{-1/2} \| B(s) \|_2^2 ds < \infty,                    \label{ap21}
 \eeq
 which should be interpreted as defining an initial condition $A_0$
 of finite action.
 This is a gauge invariant version of the condition that $A_0$
 is in the Sobolev space $H_{1/2}$ (modulo gauge transformations).
 Both the assumptions and conclusions of these
  apriori estimates are gauge invariant.
 The key tool in proving these apriori estimates will be  repeated use
 of  gauge invariant Gaffney-Friedrichs-Sobolev inequalities.

 Our derivation of  apriori estimates requires more differentiability
  of the   solution $A(\cdot)$ than is available in the definition of
   strong solution.
   To resolve this we are going to introduce a regularization method
    that produces smooth solutions which approximate strong
    solutions for a short time. It is based on the use of the approximation
    implicit in the discussion following Lemma \ref{lemDS}.
     This will be done  in Section \ref{secLTE}, where it will be needed
      for proving long time      existence.

%%%%%%%%   Dirichlet and Neumann boundary conditions.

\section{Dirichlet and Neumann boundary conditions} \label{secDN}

 In this section we will extend some of the machinery  developed by
 Conner,  \cite{Co}, for real valued differential forms to forms in a
  product bundle over $M$ with a connection.
  See \cite[Chapter 5, Section 9]{Tay1}  for a recent
  exposition of the real valued case.
    Our objective is to develop the mechanisms  needed to make effective
  use of the gauge invariant Gaffney-Friedrichs inequality of
  Section \ref{secGFS}.

\subsection{The minimal and maximal  exterior derivatives}\label{secDN1}

 \begin{notation} {\rm
  $M^{int}$ will denote the interior of the compact Riemannian 3-manifold $M$.
     Denote by $\delta$ the coderivative on
     $C^\infty(M; \L^{p+1} \otimes \kf)$. Thus if $\phi$ is a $\kf$ valued
     p - form  in $C_c^\infty(M^{int})$  and
       $\w \in  C^\infty(M; \L^{p+1} \otimes \kf)$
     then
     \beq
   (\phi, \delta \w )_{ L^2(M ;\L^p \otimes \kf)} =
               (d\phi, \w)_{L^2(M; \L^{p+1} \otimes \kf)}.           \label{C7}
     \eeq
 If $u= \sum_{|I|=r} u_I dx^I$ and $ v =\sum_{|J|=p} v_J dx^J$ are $End \ \V$
 valued forms then their wedge product,
      $u\wedge v = \sum_{I,J} u_I v_J dx^I\wedge dx^J$, is another
       $End\ \V$ valued form.
         But when the appropriate action of $u$ on $v$ is via $ad\ u$
      then we will write
  $[u\wedge v] = \sum_{I,J} [u_I, v_J]dx^I\wedge dx^J$.
 This will be the case when $u$ is an $End\ \V$ valued connection form or its time derivative. If $u$ and $v$ take their values in $\frak k$ then so does $[u\wedge v]$.

    The interior product,  $[u\lrc v]$, of an element $u \in \L^p \otimes \kf$ with
         an element $v \in \L^{p+r} \otimes \kf$ is defined by
         \beq
         \< w, [u\lrc v]\> _{\L^r \otimes \kf}
         = \< [ u\wedge w ], v\>_{\L^{p+r} \otimes \kf} \ \ \text{ for all}\ \
                    w \in \L^r\otimes\kf.     \label{C9}
         \eeq
          If $u$ or $v$ or both are real valued then we will write simply $u\lrc v$ since the commutator bracket should be omitted.
 If     $u$ and $v$ are both in $\Lambda^1 \otimes \frak k$ then \eref{C9} gives
    $\frak k \ni [u\lrc v] = - [u\cdot v] = -\sum_j [u_j, v_j]$ in an orthonormal frame for $\L^1$.  And if $w\in \L^2\otimes \frak k$ then $[w\lrc w] =0$.
}
\end{notation}

  We wish to consider a connection on the product bundle
   $M\otimes \V \rightarrow M$.
    We may     and will identify the connection with a
    $ \kf$ valued 1-form $A$ on $M$.
   The corresponding  gauge covariant exterior derivative is then given by
   $ d_A \w = d \w + [ A\wedge \w]$ on smooth $\kf$ valued forms.
    However we are going to use the symbols $d$ and $d_A$ for the closed
   versions of these differential operators as follows.

\begin{notation}{\rm Denote by $D$ the closure of the exterior derivative
 operator
defined initially on $\frak k$ valued p-forms in $C^\infty(M)$.
Denote by $d$ the closure of $D|C_c^\infty(M^{int})$. Then $d \subset D$.
$D$ and $d$ are the maximal and minimal  exterior derivative
operators respectively.

 For $A\in L^\infty (M; \L^1\otimes \frak k)$ define
 \begin{align}
 D_A \w &= D \w +[ A\wedge \w]\ \ \
         \text{for}\ \ \ \w \in \D(D)    \label{C13}\\
 d_A\w &= d\w +[ A\wedge \w]\ \ \ \
        \text{for}\ \ \ \w\in  \D(d)   \label{C14}
\end{align}
}
\end{notation}

The Hodge star operator $*$ on forms defines a unitary map
from $L^2$ forms to itself with the following properties.
\begin{align}
D_A^*  &= *^{-1} (d_A) *                                        \label{C16}\\
d_A^*  &=*^{-1} (D_A )*                                          \label{C17}\\
W_1 &\subset \D(D_A) \cap \D(d_A^*)                  \label{C18}\\
d_A &= (\delta_A)^*                                              \label{C19d}\\
D_A & = (\delta_A|C_c^\infty(M^{int}))^*                   \label{C19}
\end{align}
where
\beq
  \delta_A \w = \delta \w + [A\lrc \w] \ \ \
 \text{for}\ \  \w \in C^\infty( M; \L^p \otimes \kf), \ \  p \ge 1  \label{C11}
  \eeq

\begin{remark}{\rm $ D_A$ and $d_A^*$ are maximal operators
in the sense that their domains are restricted only by size and regularity
and not by boundary conditions.
 However the domains of their adjoints,
 $D_A^*$ and $d_A$  are restricted also by
  boundary conditions as follows.

         The symbol $(D)$ in front of an equation will
  signify that the equation is relevant for Dirichlet boundary conditions.
  An $(N)$ signifies that the equation is relevant for Neumann boundary conditions.
  }
  \end{remark}

           \begin{lemma}\label{lemDN1}
  Suppose that $\w \in W_1(M; \L^p\otimes \frak k)$ and $A \in L^\infty(M)$.
   Then
  \begin{align}
  (D)\qquad \qquad & \w \in \D(d_A)\ \ \text{if and only if}\ \ \w_{tan}=0 \qquad\qquad                                         \label{C19D}\\
  (N)\qquad \qquad &  \w \in \D(D_A^*)\ \ \text{if and only if}\ \ \w_{norm}=0\qquad\qquad                      \label{C19N}
\end{align}
\end{lemma}

             \begin{proof}
    These boundary conditions are already known for the minimal and maximal
 operators when $A=0$. See \cite{Co}. Since $A$ is bounded the domains are the same as for $A=0$
 \end{proof}

% Proposition DN4
\begin{proposition}\label{propDN4} Assume that $\w$
is a $\frak k$ valued form  and that $A \in W_1\cap L^\infty$.
 Denote the curvature  of $A$ by $B$, as in \eref{ymh3}.

\noindent
If $ [B\wedge \w] \in L^2$ then
\begin{align}
(N)\ \ \ \w \in \D(D_A)\ &\text{implies}\  \w \in \D((D_A)^2)\
        \text{and}\  D_A^2 \w = [ B\wedge \w]       \label{DN50}  \\
 \text{and}    \ (D)\  \ \     \w \in \D(d_A)\ &\text{implies}\
            \w \in \D((d_A)^2)\
       \text{and}\   d_A^2 \w = [ B\wedge \w].              \label{DN51}
        \end{align}
  \noindent
 If $[B\lrc \w] \in L^2$ then
 \begin{align}
 (D)\ \ \  \w \in \D(d_A^*)\  &\text{implies}\     \w \in \D((d_A^*)^2)\
      \text{and}\  (d_A^*)^2 \w = [ B\lrc \w]                 \label{DN52} \\
\text{and}\ (N)\ \   \w \in \D(D_A^*)\  & \text{implies}\    \w \in \D((D_A^*)^2)
\ \text{and}\ \ \ (D_A^*)^2 \w = [ B\lrc \w].                              \label{DN53}
\end{align}
\end{proposition}

        \begin{proof} It will be clarifying to distinguish the closed
  operators $d_A$ and $D_A$ from the pointwise defined
          differential operator $\{ d_A\}$ acting
on smooth forms, and which ignores boundary conditions. If $A$ and
$\w$ are in $C^\infty(M)$ then the Bianchi identity
$ \{ d_A \}^2 \w = [B\wedge \w]$ holds and we need only
address domain issues in the four assertions of the proposition. To this end observe that if $\w$ and  $u$ are both in $C^\infty (M)$ and one
 has compact support in $M^{int}$ then we may integrate by parts to find
\beq
( \{d_A\} \w, \{\delta_A\} u) = ( \{d_A\}^2 \w, u) = ( [B\wedge \w], u)  \label{DN57}
 = (\w, [B\lrc u])\eeq
Since the first, third and fourth terms are continuous in $A$ in $W_1$ norm
the equality of these terms persists for $A\in W_1$.

        Now since $C^\infty(M)$ is a core for $D_A$ and the far right side
 is continuous in $\w$ in $L^2$ norm it follows that
\beq
(D_A \w, \{\delta_A\} u ) = (\w, [B\lrc u] ) = ([B\wedge \w], u)   \label{DN58}
\eeq
for all $\w \in \D(D_A)$ and $u \in C_c^{\infty}(M^{int})$.
 Since $[B\wedge \w] \in L^2$ the right side
 is continuous in $ u$ in $L^2$ norm and therefore so is
 $(D_A \w, \{\delta_A\} u )$. Hence
  $D_A\w \in \D(D_A)$, by \eref{C19} and  $(D_A^2 \w , u) = ([B\wedge \w], u)$.
       This proves \eref{DN50}.

To prove \eref{DN51} take $\w \in C_c^\infty(M^{int})$ and
   $u \in C^\infty(M)$ in \eref{DN57}.
   Then $\w \in \D(d_A)$ and, since $C_c^\infty(M^{int})$ is a core for $d_A$
   and $[B\lrc u] \in L^2$, equality of the first and fourth terms in \eref{DN57}
 implies that
$(d_A\w, \{\delta_A\} u) = (\w, [B\lrc u]) = ([B\wedge \w], u)$
for all $\w \in \D(d_A)$ and $u \in C^\infty(M)$. Since $[B\wedge \w] \in L^2(M)$
the equality of the first and third terms now shows that $(d_A\w, \{\delta_A\} u)$
is  continuous in $u$ in $L^2$ norm and therefore $d_A\w \in \D(d_A)$.
Thus $(( d_A)^2\w, u) = ( [B\wedge \w], u)$ for all $u \in C^\infty(M)$.
   This proves \eref{DN51}.

   The assertions \eref{DN52} and \eref{DN53} could be derived in the same
 way as \eref{DN50} and \eref{DN51}. But they also follow directly from these
 by use of \eref{C16} and \eref{C17}. Thus if $\w \in \D(d_A^*)$ then
 \eref{C17} shows that $*\w \in \D(D_A)$. By \eref{DN50} $*\w$  is therefore
 in $\D(D_A^2)$. It now follows from \eref{C17} again that $\w \in \D((d_A^*)^2)$.
 Of course $(d_A^*)^2 \w = [B\lrc \w]$ since the adjoint of
 $[B\wedge \cdot]$ is $[B\lrc \cdot]$  by \eref{C9}.
  The proof of \eref{DN52} is similar.
  \end{proof}

 %%%%   Corollary DN5
\begin{corollary}\label{corDN5} Suppose that $\w$ is a $\frak k$
 valued p-form in $W_1$, that $A \in W_1\cap  L^\infty$
and that  the function
$x \mapsto |B(x)| |\w(x)|$ is in $L^2(M)$.
      \begin{align}
(D)&   \qquad  \text{If}\ \  \w_{tan}=0\ \ \text{and}\ \ d_A \w \in W_1
             \ \ \text{then}\ \  (d_A \w)_{tan} = 0.  \qquad           \label{DN60} \\
 (N)& \qquad  \text{If}\ \   \w_{norm}=0\ \ \text{and}\ \  D_A^* \w \in W_1
             \ \ \text{then}\ \  (D_A^* \w)_{norm} = 0.                   \label{DN61}
        \end{align}
      \end{corollary}

                 \begin{proof} If $\w \in W_1$ and $\w_{tan} =0$ then
    $\w \in \D(d_A)$ by  \eref{C19D}.
  From \eref{DN51} we see that $d_A \w \in \D(d_A)$. Therefore
 if $d_A \w \in W_1$ then $(d_A\w)_{tan} = 0$ by \eref{C19D}. This proves
 \eref{DN60}. The proof of \eref{DN61}  follows from
  \eref{C19N} and \eref{DN53}  similarly.
      \end{proof}

       \begin{corollary}\label{corfbi} (Functional Bianchi identity.)
 Assume that $A \in W_1 \cap L^\infty$. Then
 \begin{align}
 &B \in \D(D_A)\ \text{and}\ \ D_A B =0. \qquad \qquad \label{DN70}
   \end{align}
 If, moreover, $A_{tan} =0$ and $ B \in W_1$ then
 \beq
(D)\qquad    B_{tan}=0,\ \ B\in \D(d_A) \  \
            \text{and}\ \ d_A B =0.  \qquad\qquad     \ \ \                  \label{DN72}
 \eeq
    \end{corollary}

              \begin{proof}
  For $A \in C^\infty(M)$ and $u \in C_c^\infty(M^{int})$ an integration by parts
  and Bianchi's identity gives
  $(B, \delta_A u) = (\{ d_A\} B, u ) =0$. The left side of this identity is continuous
  in $A \in W_1$ and therefore
   \beq
  (B, \delta_A u ) =0
  \eeq
  for all $A \in W_1$ and all $u \in C_c^\infty(M^{int})$.
    Since the right side of this identity, being zero,
   is continuous in    $u $ in  $L^2$ norm, it follows that
   $B \in \D((\delta_A|C_c^\infty(M^{int}))^*)$\linebreak
   $ =\D(D_A)$ and that $D_A B =0$,
   proving  \eref{DN70}.

  Now suppose that $A_{tan} =0$ and that $B \in W_1$. Take $\w = A$ in \eref{C19D} and take
  the form $A$ of that lemma to be our present $A/2$. It follows that
  $A \in \D(d_{A/2})$. But $d_{A/2} A = B$.
  Further, we see that $[B\wedge \w] = [B\wedge A] \in L^2$ because
  $B \in L^2$ and $A \in L^\infty$.
   Hence $ B \in \D(d_{A/2})$ by \eref{DN51}.
  Reapplying \eref{C19D} again we find that $B_{tan} =0$. Equation
  \eref{C19D} now shows that $B \in \D(d_A)$. But $d_A B = D_A B=0$,
   by \eref{DN70}.
 \end{proof}

\begin{corollary} \label{corDN9}
Assume that $A\in W_1\cap L^\infty$ and $B\in W_1$.
 Then $B \in \D((d_A^*)^2)$ and
\beq
(d_A^*)^2 B =0.                                    \label{DN80}
\eeq
If, in addition, $B_{norm} =0$ then  $ B \in \D((D_A^*)^2)$ and
\beq
(D_A^*)^2 B =0.                                    \label{DN81}
\eeq
\end{corollary}

            \begin{proof}
 We see that $B \in W_1 \subset \D(d_A^*)$, by  \eref{C18},
   and, since $[B\lrc B] =0$, we may choose $\w = B$ in \eref{DN52},
   from which it follows that $B \in \D( (d_A^*)^2)$ and
   that \eref{DN80} holds.
Suppose, further, that $B_{norm} =0$. Then \eref{C19N} implies that
$B \in \D(D_A^*)$. Therefore \eref{DN53} now shows that
 $D_A^* B \in \D(D_A^*)$ and $(D_A^*)^2 B = [B\lrc B] = 0$,
 which proves \eref{DN81}.
 \end{proof}

              \begin{remark}\label{remCon}
  {\rm  For real valued forms the use of the maximal and minimal
   operators $D$ and $d$ goes back to Conner,  \cite{Co}.
   In particular, Proposition \ref{propDN4} in the real valued case,
   which simply reads
$d^2 = 0$, $D^2 =0$, $(d^*)^2 =0$ and $(D^*)^2=0$,
 along with proper statements about the
domains, was proved by Conner \cite{Co}[ninth page].
}
\end{remark}

                 \begin{remark}\label{remdom}{\rm  A 1-form $\w$ need not be
      even weakly differentiable in
      order to be in the domain of the minimal operators $d$ or $d_A$.
      For example if $M \subset \R^3$
      and $f$ is  a $\kf$ valued smooth function with compact support in $M^{int}$
      of the form $f(x_1, x_2, x_3) = h(x_1) g(x_2, x_3)$ then, defining
      $\w = f dx_1$, one has
       $d\w = h(x_1)\{(\p_2 g) dx_2\wedge dx_1  + (\p_3 g) dx_3\wedge dx_1\}$
       so that $h$ is not differentiated. Thus  if we now  allow
       $h \in L^2(\R^1;\kf)$ and
       $g \in C_c^\infty(\R^2)$, still insisting that support $ f \subset M^{int}$, then  the resulting form $\w$ can
        easily be approximated
       in $d$ graph norm by functions in $C_c^\infty(M^{int})$ of the same form.
       So $\w \in \D(d)$.
       This example also shows that if $A$ is unbounded then $d$ and $d_A$
       will not have the same domain. One need only take $A= A_2(x_1)dx_2$.
       Then $d_A\w - d\w = [A_2(x_1), h(x_1)] g(x_2, x_3) dx_2\wedge dx_1$,
       which need not be in $L^2(M)$   if $A$ is unbounded
        and $ h \in L^2(\R; \kf)$.
               We conjecture, however that all results in this section
      will remain valid if the condition $A\in L^\infty$ is replaced by $A \in L^3$.
       }
 \end{remark}

%%%%%%%%%%% Gaffney-Friedrichs-Sobolev

\section{Gauge invariant Gaffney-Friedrichs-Sobolev inequalities} \label{secGFS}

In this section we will prove Theorem \ref{thmGF} and derive from it
 Sobolev inequalities  in a form that will be needed for establishing
 apriori estimates. We will also prove a version in which the constant
 $\lambda_3$ is replaced by a constant $\lambda_2$ depending  (quartically)
 on $\|B\|_2$. The former will be most useful for initial data of finite action.
 Both will be useful for initial data of finite energy.

     Throughout this section the exterior derivative operators $d$ and $d_A$
     and their adjoints are to be interpreted as acting on smooth forms
     or on $W_1$ forms, as indicated, without boundary conditions built in.

 %%%%%%%%%%%   Gaffney identity
\subsection{A gauge invariant Gaffney identity}     \label{secGF1}

% Begin Theorem :   thmgaf
         \begin{theorem} $($A gauge invariant Gaffney identity.$)$ \label{thmgaf}
         Let $M$ be a compact Riemannian n-manifold with smooth boundary.
Suppose that $A \in C^\infty(M; \L^1\otimes \frak k)$.
Let $\alpha$ and $\beta$ be smooth $\frak k$ valued p-forms on
$ M$ with either
\beq
\alpha_{tan} = \beta_{tan} =0 \ \text{on}\ \p M \ \ \ \
\text{or}\ \ \ \  \alpha_{norm} = \beta_{norm} =0\  \text{on}\ \p M .      \label{gaf4}
\eeq
  Then
\begin{align}
(d_A \alpha, d_A \beta ) &+ (d_A^* \alpha, d_A^* \beta )
                   - (\n^A \alpha, \n^A \beta)   -((W + B)\circ \alpha, \beta)  \notag \\
 &=  \int_{\p M} \<K(x) \alpha(x), \beta (x) \> \label{gaf5}
  \end{align}
  where
  $W$ denotes the Riemannian Bochner-Weitzenboch operator,
  $B$ is the curvature of $A$, $\circ$ denotes a pointwise product operation,
  $$
  (\n^A \alpha, \n^A \beta ) = \sum_{i=1}^n (\n^A_{e_i} \alpha, \n^A _{e_i} \beta)
  $$
  locally, for any orthonormal frame field $e_1, \dots, e_n$ of $T(M)$ and
  $$K(x) : \Lambda^p(T_x(\p M) ) \rightarrow \Lambda^p(T_x(\p M) )$$
  is a symmetric operator, bounded uniformly in $x$, and dependent
   only on the second fundamental form of $\p M$, on the value of $p$ and on the choice of boundary condition in \eref{gaf4}. Moreover, $K(x) \ge 0$ for all
   $x \in \p M$ if $ M$ is convex in the sense that
    the second fundamental form  is  non-negative on $\p M$.
      \end{theorem}

 The proof depends on the following lemmas.
 It is important for our applications that
 the boundary terms in \eref{gaf5}  above do not depend on the gauge
  connection form. For this reason we are going to carry out explicitly
   what is otherwise a standard kind of integration by parts
  computation.

 \begin{notation}\label{notadc}
     {\rm(Adapted coordinates.) We will make use in this section
  of an adapted coordinate  system for a neighborhood $U$
 containing a part of the boundary of $M$.
This is a coordinate system $x = (x^1,\dots, x^n) $ in $U$  such that
a) $|x^j| <1$ for $j = 1,\dots, n-1$ and $-1 < x^n \le 0$,
 while $U\cap \p M = \{ x: x^n =0\}$.
  $(x^1,\dots, x^{n-1})$ form coordinates on $ U \cap \p M$.
b) The curve $ (-1, 0] \ni t \mapsto x(t) = (a^1,\dots, a^{n-1}, t)$ is a geodesic
normal to $\p M$ at $t=0$, and
 $\< \p/\p x^n, \p/ \p x^j \> =0$ on $U$ for $ j =1, \dots, n-1$.
See for example \cite[page 167]{RaS}  for the existence of such a
coordinate chart.
}
\end{notation}

% begin Lemma {lemgaf1}
           \begin{lemma} \label{lemgaf1}
Assume that $ A \in C^\infty( M; \L^1 \otimes \kf)$
and that $\alpha$ and $\beta$ are smooth $\kf$ valued p-forms on $ M$.
 Then
  \begin{align}
 (d_A \alpha, d_A \beta ) &+ (d_A^* \alpha, d_A^* \beta )
                    - (\n^A \alpha, \n^A \beta)   - ((W + B)) \alpha, \beta) \notag\\
  &= L^A(\alpha,\beta)
                                                                        \label{gaf12}
   \end{align}
   where
   \beq
   L^A(\alpha, \beta) = \int_{\p M} \{ \< \nu\wedge \beta, d_A  \alpha \>
    - \< \beta, \nu \wedge d_A^*  \alpha\> - \<\beta, \n^A_{\nu} \alpha\> \}. \label{gaf13}
   \eeq
Here $\nu$ is the  outward  drawn unit co-normal and $\n_\nu^A$ is
the covariant gradient in the normal direction.
   \end{lemma}

            \begin{proof}
      The Bochner-Weitzenboch formula for a ${\frak k}$ valued
 $p$-form on
   $ M$ is
   \beq
   \{d_A^*d_A + d_A d_A^*\}\alpha- (W+B)\circ \alpha
              = (\n^A)^* \n^A \alpha,                           \label{gaf11}
   \eeq
   which may be found in  \cite{BL}.
       We need only take the inner product of \eref{gaf11} with $\beta$
  and do three integrations by parts to deduce \eref{gaf12}.
   Two of the integrations by parts will follow from Stokes' theorem,
   \begin{align}
        (d_A \w, u) - ( \w, \delta_A u)  = (\nu \wedge \w, u)_{\p M}, \ \
        \w \in C^\infty(M),
        \ \ u \in C^\infty(M),                     \label{C27}
        \end{align}
  which itself can be derived from the standard Stokes theorem by observing
  first that the terms involving the connection form $A$ cancel on the left,
  in view of \eref{C7}, \eref{C9}, \eref{C13}
  and \eref{C11}, and  second,
   that the resulting identity
  holds for forms $\w \in C^\infty(M)$ and $u \in C^\infty(M)$ because it holds
  for the real valued components of these forms with respect to an orthonormal
  basis of $\frak k$.

Now  inserting first   $\w = \beta$, $ u = d_A \alpha$ into \eref{C27}
  and then  inserting $\w = d_A^* \alpha$, $u =\beta$
   into \eref{C27} we find, respectively,
  \begin{align*}
  \<d_A \beta, d_A,\alpha)
  &= (\beta, d_A^* d_A\alpha)
            + \int_{\p M} \< \nu \wedge \beta , d_A \alpha \>,\\
  \< d_A^* \alpha, d_A^*\beta\> &= (\beta,  d_Ad_A^* \alpha) -
  \int_{\p M} \< \nu \wedge d_A^* \alpha, \beta \> .
    \end{align*}
  Combining these with \eref{gaf11} we find that the left
  side of \eref{gaf12} is equal to
  $((\n^A)^* (\n^A) \alpha, \beta) - ((\n^A) \alpha, (\n^A) \beta)
  +(\nu \wedge \beta, d_A \alpha)_{\p M}
  - ( \nu \wedge d_A^* \alpha, \beta)_{\p M}$.

          To complete the proof of \eref{gaf12} it suffices to  show that
     \beq
     ((\n^A)^* (\n^A) \alpha, \beta) - ((\n^A) \alpha, (\n^A) \beta)
      = -\int_{\p M}\< \n_{\nu}^A \alpha, \beta\> .                         \label{gaf16}
      \eeq
   For the needed integration by parts we may write,
   with the help of a partition of unity,
   $ \alpha = \alpha_0 + \sum_{j=1}^r \alpha_j$, where
   $\alpha_0$ is supported in
    $M^{int}$ and each $\alpha_j$ is supported in an
    adapted coordinate patch $U_j$.
           For an arbitrary $\frak k$ valued
  $p$-form $\beta$ in $C^\infty( M)$
  the identity  \eref{gaf16} holds for  $\alpha_0$ and $\beta$
  by an integration by parts because there are no boundary terms.
 It suffices therefore
  to prove \eref{gaf16} for each $\alpha_j$.
 To this end we will prove \eref{gaf16}  in case $\beta \in C^\infty( M)$
   while $\alpha$ is supported in an adapted
   coordinate patch $U \subset  M$.

               For any smooth vector field $X$ on $U$ and real valued function
         $f \in C_c^\infty(U)$ we may apply the identity
         $\int_U Xf  +\int_U f(div\ X) = \int_{\p U} f (\nu\cdot X) $,
          to the real valued function
 $f(x) = \< \w(x), \beta(x) \>_{\L^p \otimes \kf}$ to find
    $$
          \int_U (div \ X)  \<\w, \beta\> + (\n_X^A \w, \beta) + ( \w, \n_X^A \beta) =
         \int_{\p M} \< \w, \beta\> (\nu \cdot X)
   $$
   for any p-form $\w \in C_c^\infty(U)$.
   We read off from this that the formal adjoint of $\n_X^A$ is given by
   $ (\n_X^A)^*\w = - \n_X^A \w- (div\ X) \w$ and that
   \beq
   ((\n_X^A)^* \w, \beta) = ( \w, \n_X^A \beta)
   - \int_{\p M} \< \w, \beta\> ( \nu\cdot X)          \label{gaf18}
   \eeq
   Choose an orthonormal frame field $ e_1,\dots, e_n$ in the
   coordinate patch $U$  and apply \eref{gaf18} with
    $X = e_j$ and $\w = \n_{e_j}^A\alpha$ to find
    \beq
    ( ( \n_{e_j}^A)^*\n_{e_j}^A \alpha, \beta) = (\n_{e_j}^A \alpha, \n_{e_j}^A \beta) - \int_{\p M} \< \n_{e_j}^A \alpha, \beta \> \nu\cdot e_j
    \eeq
    Summing over $j$ gives \eref{gaf16}.
  \end{proof}

\bigskip
Unlike Stokes' theorem, \eref{C27}, the connection form $A$ shows up in the boundary term $L^A(\alpha, \beta)$ of  \eref{gaf13}.
We may disentangle
 the $A$ dependence in $L^A(\alpha, \beta)$. We find
\begin{align}
L^A(\alpha,\beta)
 &= \int_{\p M} \{ \< \nu \wedge \beta, d \alpha \>
 - \< \beta, \nu \wedge d^* \alpha\> -\<\beta, \n_\nu \alpha \> \}\label{gaf19.1}\\
 &+ \int_{\p M} \{ \< \nu \wedge \beta, [ A \wedge \alpha]\>
  - \< \beta , \nu \wedge [ A\lrc \alpha]\> -\<\beta, [ A_\nu, \alpha ] \> \}      \label{gaf19.2}
 \end{align}
 where $A_\nu(x) \nu = A_{norm}(x)$ is the normal component of $A$ at $x$.
 It will be important for us that the boundary term be independent of $A$
 when the p-forms $\alpha$ and $\beta$ satisfy appropriate boundary conditions.

        \begin{lemma}\label{lemgaf2} The integrand in line \eref{gaf19.2}
is zero at a point $x \in \p M$ if either
\beq
\alpha_{tan}= \beta_{tan} = 0\ \ \text{at}\ x     \label{gaf20.1}
\eeq
or
\beq
\alpha_{norm}= \beta_{norm} = 0\ \ \text{at}\ x.  \label{gaf20.2}
\eeq
$A(x)$ need not satisfy any boundary condition in either case.
\end{lemma}

        \begin{proof}
 Fix $x \in \p M$.
Assume first that $  \alpha_{tan}= \beta_{tan} = 0\ \ \text{at}\ x $. Then
$\nu\wedge \beta =0$. So the first term in \eref{gaf19.2} is zero at $x$.
We assert that the remaining two terms cancel. Indeed, since $\beta_{tan} = 0$
we may write $ \beta = \nu \wedge \phi$ at $x$ with $\phi_{norm} =0$.
Then
$ \< \beta, \nu\wedge [ A \lrc \alpha] \>
= \<\nu \wedge \phi,\nu\wedge [ A \lrc \alpha] \>
=\< \phi, [ A \lrc \alpha]\>
= \< [A\wedge \phi], \alpha\>
 = \< [A_\nu, \nu\wedge \phi]+ \text{a tangential term},\alpha \>
 = - \< \nu \wedge \phi, [ A_\nu, \alpha] \>
 = - \< \beta, [ A_\nu, \alpha ] \>$.
  Thus the second and third terms in \eref{gaf19.2} cancel.

Assume next that $\alpha_{norm}= \beta_{norm} = 0\ \ \text{at}\ x$.
The middle term is zero because $\beta_{norm} =0$. We assert that
the first and third terms cancel. Indeed
$ \< \nu\wedge \beta, [ A \wedge \alpha]\>
= \< \nu\wedge \beta, [ A_\nu , \nu \wedge \alpha] + \text{a tangential term}\>
= \< \nu \wedge \beta, \nu \wedge [ A_\nu, \alpha]\>
= \< \beta, [ A_\nu, \alpha]\>$, which shows that the first and third
terms in  \eref{gaf19.2} cancel.
\end{proof}

%%%%% Remark on A dependence.

             \begin{remark}\label{remAdep}
{\rm It is  illuminating to understand when the
 integrand  in \eref{gaf19.2} is  identically zero, independently of boundary conditions on $\alpha$ and $\beta$.
It can be shown that

 a)  the integrand is zero at a point $x \in \p M$ for all $\alpha$ and
           $\beta$ if $A_{tan}(x) = 0$

 b) if $ \kf$ is semisimple and  $\alpha_{tan}(x) =0$ then
    there exist $A$ and $\beta$  such that the integrand is not zero at $x$.

 We omit the proofs.
 }
 \end{remark}

        \begin{notation}\label{notshape}
{\rm  (Extended shape operator)
An adapted coordinate system (see Notation \ref{notadc})
will be useful for describing the shape
operator and its extension to the exterior algebra.
 Writing $\p_j = \p / \p x^j$, the outward drawn unit normal and co-normal
are given by $\p_n$ and $\nu = dx^n$, respectively,  on $U\cap \p M$.
 The shape operator at a point
$P \in  U\cap \p M$ is given by $ S(X) = \n_X \p_n$ for $ X \in T_P (\p M)$,
\cite[page 217]{GHL},
where $\n_X$ is the Riemannian covariant derivative.
 The adjoint $S^* \in End(T_P^*(\p M))$ extends uniquely
to a derivation $Q$ of the exterior algebra $\L (T_P^*(\p M))$.
We may identify $\L (T_P^*(\p M))$ with the algebra of
exterior polynomials in the 1-forms $dx^1,\dots, dx^{n-1}$ with constant coefficients. The action of $Q$ on such an exterior polynomial $\w$ is given
by
\beq
- (\n_n \w )|_{\p M} = Q (\w|_{\p M}),                       \label{sh8}
\eeq
as one sees by observing first, that $\n_n \p_n =0$ because $t\mapsto (a_1,\dots, a_{n-1}, t)$ is a geodesic, second, that $\n_n$ therefore   leaves invariant
the span of $dx^1,\dots, dx^{n-1}$, third, that $S^* = \n_n^* = -\n_n$ on this span,
and finally, that $Q$ and $-\n_n$ are derivations of this algebra.
}
\end{notation}

         \begin{proof}[Proof  of Theorem \ref{thmgaf}]
In view of \eref{gaf12} and Lemma \ref{lemgaf2} we can ignore the connecton form $A$ and just show that the integrand  in \eref{gaf19.1} has the form assserted in \eref{gaf5}. Explicitly, we will show that
\beq
\< K(x) \alpha(x), \beta(x)\> =
\begin{cases}  \< \{ I_{\frak k} \otimes Q(x)\} \alpha(x), \beta(x)\>\ \
         &\text{if}\ \alpha_{norm} = \beta_{norm} =0 \\
\< \{ I_{\frak k} \otimes (*^{-1}Q(x)*)\} \alpha(x), \beta(x)\>\ \
         &\text{if}\ \alpha_{tan} = \beta_{tan} =0.
\end{cases}                   \label{sh11}
\eeq

        Assume first that $\alpha_{norm} = \beta_{norm} =0$. Choose an
        orthonormal basis $e_1, \dots, e_d$ of $\frak k$ and write
       $ \alpha = \sum_{i=1}^d e_i \alpha^i$ and $\beta = \sum_{i=1}^d e_i \beta^i$
        where $\alpha^i$ and $\beta^i$ are real valued p-forms.
    Then the integrand in    \eref{gaf19.1} is
        $$
        \sum_{i=1}^d \{ \< \nu \wedge \beta^i, d \alpha^i \>_{\L^{p+1}}
        -  \< \beta^i , \nu \wedge d^*\alpha^i \>_{\L^p}
        - \< \beta^i, \n_{\nu} \alpha^i \> \}
        $$
    at a point $ P \in \p M$. It suffices to show that this has the form
        $\sum_{i=1}^d \< Q(x) \alpha^i, \beta^i \>$ for then one can take
    $K(x) = I_{\frak k} \otimes Q(x)$ in \eref{gaf5}.

         Now $\alpha(x)_{norm} =0$ if and only if $\alpha^i(x)_{norm} =0$
          for each  $i$.
     Thus it suffices to prove that the integrand in \eref{gaf19.1}
          is equal to
         $ \< Q(x) \alpha(x), \beta(x) \>$ when $ \alpha$  and $\beta$
          are real valued p-forms such that
          $\alpha_{norm} = \beta_{norm}$ on $U\cap \p M$.  In this case
          the middle term in \eref{gaf19.1}, $ \< \beta, \nu \wedge d^* \alpha\> = 0$
          and we are left with
          $\< \beta, \nu \lrc d\alpha - \n_\nu \alpha\>$.

          We will compute this in an  adapted coordinate system.
We may write
\beq
\alpha(x) = \sum_{J<n} a_J(x) dx^J + \sum_{I<n} b_I(x) dx^I \wedge dx^n
\eeq
Here and below $ J = (j_1,\dots j_p)$ with $j_1 < \cdots < j_p <n$
 and $I =(i_1,\dots, i_{p-1})$ with $i_1 <\cdots   < i_{p-1} < n$. Moreover $b_I(x) =0$ if $x \in U\cap \p M$.
Then
\begin{align*}
 \nu\lrc (d \alpha(x))&= \sum_{J<n} \{ \nu \lrc \sum_{k=1}^{n-1} \p_k a_J(x) dx^k \wedge dx^J + \nu \lrc (\p_n a_J(x)) dx^n \wedge dx^J\} \\
 &+ \sum_{I<n} \{ \nu \lrc \sum_{k=1}^{n-1} \p_k b_I(x) dx^k \wedge dx^I \wedge dx^n\}\\
 & =  \sum_{J<n} \p_n  a_J(x) dx^J
 \end{align*}
 because $ \nu \lrc( dx^k \wedge dx^J) =0$ and $ \p_k b^I (x) =0$ on $U\cap \p M$
 for $ k = 1, \cdots, n-1$

   On the other hand, on $\p M$,
        \begin{align*}
 \n_\nu \alpha =  \n_n \alpha & =   \sum_{J<n} \{ (\p_n a_J) dx^J
              +  a_J \n_n (dx^J) \} \\
    & +\sum_{I<n}\{ (\p_n b_I) dx^I \wedge dx^n
          + b_I \n_n (dx^I \wedge dx^n)\}
 \end{align*}
 On $\p M$, therefore, we find some cancellation in the following difference
 and, since $b_I = 0$ on $\p M$, we arrive at
 \begin{align*}
 \nu \lrc (d \alpha(x)) - \n_\nu \alpha(x) = - \sum_{J<n} a_J(x) \n_n (dx^J)
 -\sum_{I<n} (\p_n b_I(x) ) dx^I \wedge dx^n
 \end{align*}
        Finally, since $\beta_{norm} =0$ we find, at $x \in \p M$, in view of \eref{sh8},
 \begin{align}
 \< \beta, \nu \lrc d \alpha - \n_\nu \alpha \>
 &= -\sum_{J<n} \< \beta, a_J \n_n (dx^J) \>\notag\\
 &= \sum_{J<n} \< \beta , a_J Q dx^J \>\notag\\
 & = \< \beta, Q(x) \alpha \>        \label{sh17}
 \end{align}
 This proves \eref{sh11} and \eref{gaf5} if $ \alpha_{norm} = \beta_{norm} =0$.

         In the case $\alpha_{tan} = \beta_{tan} = 0$ we may reduce to real
  valued forms  in the same way as above. Denoting the Hodge star operator
 on $\L(T^*(M))$ by $*$  we can reduce this case to the preceding by applying
 the preceding case to the $n-p$ forms $*\alpha$ and $*\beta$, which,
  as is well known, satisfy now
   $(*\alpha)_{norm} = (*\beta)_{norm} = 0$. Applying the identity
   \eref{sh17} to these two forms we find
 $$
 \< * \beta, \nu \lrc d*\alpha - \n_\nu *\alpha\>_{\L^{n-p}}
  = \< * \beta, Q(x) *\alpha\>_{\L^{n-p}}
 $$
 and therefore
 $$
  \<\beta , *^{-1} ( \nu \lrc (d *\alpha)) - *^{-1} \n_\nu * \alpha\> _{\L^p}
  = \< \beta, *^{-1} Q(x)* \alpha \>_{L^p}
  $$
  But $ *^{-1} \n_\nu * = \n_\nu$ while
  $$
  *^{-1} (\nu \lrc (d*\alpha)) = - \nu \wedge d^* \alpha
  $$
  by      \cite[Lemma 4.1, items (1), (6) and (10)]{MMT01}.
         Hence
\beq
- \<\beta, \nu \wedge d^*\alpha + \n_\nu \alpha\>
 = \< \beta, *^{-1} Q(x) * \alpha\>         \label{sh31}
\eeq
Now the first term  in the integrand in \eref{gaf19.1} is zero because
$\nu\wedge \beta = 0$. Thus \eref{sh31} shows that the integrand
in \eref{gaf19.1}  is $\<\beta(x), *^{-1} Q(x) * \alpha(x)\>$.
Hence we may take $K(x) = I_{\frak k} \otimes (*^{-1} Q(x) *)$ in this case.
This completes the proof of \eref{sh11} and \eref{gaf5}.

     Finally,  observe that if the second fundamental form is  greater
 than or equal to zero,
i.e. $S(x) \ge 0$ on $\p M$, then $S^* \ge 0$ also, as is also
$Q(x)$ and the unitary transform $*^{-1} Q(x) *$. The identity \eref{sh11} therefore shows that $K(x) \ge 0$ in both cases.
 \end{proof}

 % Begin  Gauge invariant Gaffney-Friedrichs inequality itself.
%  It is for Dir. AND Neu. bdy conds.

\subsection{A Gaffney-Friedrichs inequality in 3 dimensions} \label{secGF2}

          \begin{proof}[Proof of Theorem \ref{thmGF}]
          We resume the assumption that $M$ is a compact Riemannian 3-manifold
          with smooth boundary and assume now
          that $A$ and $\w$ are in $W_1(M)$.
We will write
\beq
\|\w\|_{W_1^A(M)}^2 = \|\n^A \w \|_{L^2(M)}^2 + \| \w \|_{L^2(M)}^2.    \label{gaf55}
\eeq
for any $\frak k$ valued p-form $\w$.
By Kato's inequality
\beq
\int_M | grad | \w|\, |^2 \le \| \n^A \w \|_2^2,        \label{gaf53}
\eeq
and, by Sobolev's inequality, there exists  a constant $\kappa$,
 depending on the  geometry of $M$
 but not on $A$, such that
$\|\w \|_6^2 \le (\kappa^2/2) ( \int_M | grad |\w |\, |^2 + \|\w \|^2_2 )$
for all $\w \in W_1(M)$. (See e.g., \cite[Theorem 7.26]{GT}.)
Hence
\begin{align}
\|\w \|_6^2 \le (\kappa^2/2) (\|\n^A \w \|_{L^2(M)}^2
          + \|\w \|_2^2)\ \ \text{for}\  \w\  \text{and}\ A \in W_1(M)    \label{gaf54}
\end{align}
Applying H\"older's inequality for the product $ |B(x)| |\w(x)| |\w(x)|$ we  therefore find (ignoring a factor $2^{-1/2}$)
\begin{align}
|(B \w, \w)| &\le c \|B\|_3 \|\w\|_2 \|\w\|_6  \notag \\
&\le c \|B\|_3 \|\w\|_2 \kappa \|\w\|_{W_1^A(M)} \notag \\
&\le (c\kappa)^2 \|B\|_3^2 \|\w\|_2^2 +(1/4) \|\w\|_{W_1^A(M)}^2 .  \label{gaf58}
\end{align}
     Define, for $0 < a \le 1$,  the fractional  Sobolev norm
\beq
\| f \|_{H_a(M)} = \| (1 - \Delta)^{a/2} f \|_{L^2(M)}\ \ \
                 \text{for}\ \ f: M  \rightarrow R,                              \notag
\eeq
where $\Delta$ denotes the self-adjoint Neumann Laplacian on real
valued functions on $M$.  By the  spectral theorem we have the
interpolation inequality $\|f\|_{H_a} \le \|f\|_{L^2(M)}^{1-a} \|f \|_{H_1}^a$.
Moreover for $1/2 < a \le 1$  the trace inequality
$\|f|\p M \|_{H_{a-1/2}(\p M)} \le \tau_a \| f \|_{H_a(M)}$  holds for some constant
$\tau_a$. (See e.g. \cite[Chapter 4, Proposition 4.5]{Tay1}.)
         In particular, taking $ a = 3/4$, and observing that
         $ \|f|\p M \|_{L^2(\p M)} \le \| f|\p M \|_{H_{1/4}(\p M)}$,
         it follows that  $\| f  |\p M\|_{L^2(\p M)} \le \tau \| f\|_{H_{3/4}(M)}$ for some
         $\tau < \infty$, and  therefore
  \beq
\| f  |\p M\|_{L^2(\p M)}  \le \tau \| f\|_{H_{3/4}(M)}
\le \tau  \|f \|_{L^2(M)}^{1/4} \|f \|_{H_1(M)}^{3/4}   \notag
\eeq
Taking account of \eref{gaf53} and putting $f = |\omega|$ in the last
 inequality we have
\begin{align}
\| \w \|_{L^2(\p M)}^2
&\le \tau^2 \|\w \|_{L^2(M)}^{1/2} \|\w\|_{W_1^A(M)}^{3/2}    \label{gaf60}\\
&\le (1/4) \Big(\frac{\tau^2}{\epsilon}\Big)^4
     \| \w \|_{L^2(M)}^2 +(3/4) \epsilon^{4/3} \| \w  \|_{W_1^A(M)}^2          \notag
\end{align}
by virtue of the convexity inequality
 $ uv \le (1/4) u^4 + (3/4) v^{4/3}$.
 With $K$ defined by \eref{gaf5} let $\|K\|_\infty = \sup_{x \in \p M} \| K(x)\|_{{\rm End}\L^p(T_x(\p M)}$.
 Choose $\epsilon$ so that $(3/4) \epsilon^{4/3} \|K\|_\infty = 1/4$
 to deduce
 \begin{align}
 |\int_{\p M} \< K(x) \w(x), \w(x) \>|
 &\le \| K\|_\infty \| \w \|_{L^2(\p M)}^2 \notag           \\
 &\le \gamma_1 \| \w \|_{L^2(M)}^2 + (1/4) \|\w \|_{W_1^A(M)}^2    \label{gaf62}
 \end{align}
 where $\gamma_1 = (27/4) (\tau^2 \|K\|_\infty)^4$. \eref{gaf62} holds for any
 p-form $\w \in W_1^A$. Note that neither $\w$ nor $A$ need
  satisfy any boundary
 conditions for the validity of \eref{gaf58} and \eref{gaf62}.

    Suppose now that $\w$ satisfies one of the boundary
  conditions in \eref{gaf49}.  Put $\alpha = \beta = \w$ in \eref{gaf5}
    to find, for smooth $A$ and $\w$,
 \begin{align}
\|\n^A \w \|_{L^2(M)}^2  &=  \|d_A \w \|_2^2  +\|d_A^* \w \|_2^2
- (W \w,\w)     \notag\\
& -(B \w,\w) - \int_{\p M} \<K(x) \w(x), \w(x)\>.      \label{gaf52}
\end{align}
This identity was derived in Theorem \ref{thmgaf} under the
assumption that $A$ and $\w$ are smooth.
 But all six terms in  \eref{gaf52} are jointly continuous in $A$ and $\w$ in $W_1\times W_1$. So \eref{gaf52} is valid for $A$ and $\w$ in $W_1(M)$
 also, if one simply interprets $d_A^* \w = d^*\w + [A\lrc \w]$.

We will bound the last two terms in \eref{gaf52} using
 \eref{gaf58} and \eref{gaf62}.
Inserting these inequalities into \eref{gaf52} we find
\begin{align}
\|\n^A \w \|_2^2 +\|\w\|_2^2
&\le \|d_A \w\|_2^2 + \|d_A^*\w\|_2^2 + (1 +\|W\|_\infty) \|\w\|_2^2\notag\\
&+ \gamma_2 \| \w \|_2^2 + (1/2) \| \w \|_{W_1^A(M)}^2          \label{gaf63}
\end{align}
with
$\gamma_2 = (\kappa c)^2 \|B\|_3^2 + (27/4) (\tau^2  \| K \|_\infty)^4$

  Shift the last term in \eref{gaf63} to the left side  to deduce
    \eref{gaf50} with
\beq
\lambda_3 = 1 + \|W\|_\infty + 7(\tau^2 \|K\|_\infty )^4
      + (\kappa c)^2 \|B\|_3^2,                        \label{gaf66}
 \eeq
 which is of the form asserted in \eref{gaf51}, wherein we may take
 \beq
 \lambda_M = 1 + \|W\|_\infty  +7(\tau^2 \|K\|_\infty )^4         \label{gaf67}
 \eeq
 and $\tau$ is the norm of the trace map $H_{3/4}(M) \rightarrow L^2(\p M)$.
\end{proof}

\begin{corollary}\label{corGFiconvex}
Suppose that $M\subset \R^3$ and is convex in the sense that its second fundamental form is non-negative. Let $A\in W_1(M)$.
If $\w \in W_1(M; \L^p\otimes \frak k)$
and either $\w_{tan} =0$ or $\w_{norm}=0$ then
\beq
 \| \n^A \w\|_2^2 + \|\w\|_2^2
 \le (4/3)\{\| d_A \w\|_2^2 +| d_A^* \w\|_2^2 + \l_3 \|\w\|_2^2\}  \label{gaf76}
 \eeq
 with
 $\l_3 = 1 +(c\kappa)^2 \|B\|_3^2 $.
\end{corollary}

         \begin{proof}
Put $\alpha = \beta =\w$ in \eref{gaf5} to find
\begin{align}
 \|\n^A \w\|_2^2 &=  \| d_A \w \|_2^2 + \| d_A^* \w \|_2^2 - ( B \w, \w)
      - \int_{\p M} \< K(x) \w(x), \w(x) \>   \notag \\
     & \le   \| d_A \w \|_2^2 + \| d_A^* \w \|_2^2 - ( B \w, \w), \label{gaf78}
   \end{align}
since $W= 0$ and $K(x) \ge 0$.
Insert the estimate \eref{gaf58}  to deduce \eref{gaf76}.
\end{proof}

        If one takes $M$ to be a cube in $\R^3$ then although $M$
  does not have a smooth boundary the identity \eref{gaf5} is easily
  verified directly and, since $K(x) = 0$ on the flat
sides of $\p M$, one finds no boundary terms.  The inequality \eref{gaf76}
holds, therefore, in this case also.

%Cor.  corGFS
       \begin{corollary}\label{corGFS} $($Gaffney-Friedrichs-Sobolev inequality.$)$
Assume  that  $M$ is a compact Riemannian 3-manifold with smooth boundary.
Let $A \in W_1(M)$  and suppose that $\|B\|_{L^3(M)} < \infty$.
If $\w \in W_1(M: \L^p\otimes \frak  k)$ and either
$\w_{tan}= 0$ or $\w_{norm}=0$ then
\beq
\|\w \|_{L^6(M)}^2
\le \kappa^2 ( \|d_A \w \|_2^2 + \|d_A^* \w \|_{L^2(M)}^2
     + \lambda_3\| \w \| _{L^2(M)}^2)                                  \label{gaf68}
\eeq
with $\lambda_3 $ given by \eref{gaf51}.
Moreover if $M\subset \R^3$ and is convex then one can take
$\lambda_3 = 1 +(c\kappa)^2 \| B\|_3^2$.
\end{corollary}

                       \begin{proof}
    Combine \eref{gaf50} and \eref{gaf54}. Use \eref{gaf76} if $M\subset \R^3$
   and  is convex.
\end{proof}

       \begin{example}\label{exball}
{\rm  Take $M$ to be a closed ball of radius $R$ in $\R^3$. In this case
 $W =0$ and the principal curvatures are both $1/R$. Hence
$$
     K(x) =
     \begin{cases}  &  1/R\ \text{if}\ p=1\\
                         & 2/R\ \text{if}\ p =2\\
     \end{cases}
$$
in the case $\alpha_{norm} = \beta_{norm} =0$.
In the case $\alpha_{tan} = \beta_{tan} = 0$ the two lines should be interchanged.
}
\end{example}

\begin{remark}{\rm
 The following theorem is a slight
variant of Theorem \ref{thmGF}, in which the $A$ dependent constant,
$\lambda_3$ in   \eref{gaf51}, is replaced by a constant $\lambda_2$
depending quartically on $\|B\|_2$ instead of quadratically on $\|B\|_3$.
 Both forms of these inequalities will be needed.
 }
 \end{remark}

% \{GFS ineq. for finite energy. \}

\begin{theorem}\label{GFi'}
          $($Gauge invariant Gaffney-Friedrichs inequality for finite\linebreak
    energy.$)$
Suppose, as before, that dimension $M =3$ and that $M$ has
 a smooth boundary.
Assume that $A \in W_1(M)$.
Then  for any $\frak k$ valued p-form $\w$ in $W_1(M)$ with
  $\w_{tan} =0$ or $\w_{norm}=0$
the inequality
\beq
(1/2)\{\|\n^A \w \|_{L^2(M)}^2 + \|\w \|_2^2\}
    \le  \|d_A \w \|_{L^2(M)}^2  + \|d_A^* \w \|_{L^2(M)}^2
            + \lambda_2\|\w\|_{L^2(M)}^2                                       \label{gaf50'}
\eeq
holds with a constant $\lambda_2$ that depends quartically on $\|B\|_{L^2(M)}$
and on the geometry of $M$ but not otherwise on the connection
 form $A$. Explicitly, $\lambda_2$ may be taken to be given by \eref{gaf71'},
 where $\tau$ is a trace map norm.
\end{theorem}

                 \begin{proof}
        The proof largely duplicates the proof of Theorem \ref{thmGF} with
  some small  changes.
From the interpolation  inequality
 $\| \w \|_4^2 \le \| \w \|_2^{1/2} \| \w \|_6^{3/2}$ and \eref{gaf54} we find
            \begin{align}
|(B\w, \w)| \le c\|B\|_2 \| \w \|_4^2
\le c\|B \|_2 \| \w \|_2^{1/2}
               \kappa^{3/2} \| \w\|_{W_1^A(M)}^{3/2}      \label{gaf57'}
          \end{align}
Combining \eref{gaf52}, \eref{gaf57'}, and \eref{gaf60} we find
\begin{align}
\|\n^A \w \|_2^2
&\le \|d_A \w\|_2^2 + \|d_A^*\w\|_2^2 + \|W\|_\infty \|\w\|_2^2\notag\\
&+\{ c\|B\|_2  \kappa^{3/2} \} \|\w\|_2^{1/2} \|\w\|_{W_1^A(M)}^{3/2}\notag \\
 &+ \{ \|K\|_\infty \tau^2 \}\|\w\|_2^{1/2} \|\w\|_{W_1^A(M)}^{3/2}
                                                                                     \label{gaf63'}
\end{align}
Denote by $\gamma$ the  sum of the two terms in braces.
Then the last two terms add to
$$
(\gamma/\epsilon) \|\w\|_2^{1/2} (\epsilon \|\w\|_{W_1^A(M)}^{3/2})
\le (1/4) (\gamma/\epsilon)^4 \|\w\|_2^2 +(3/4) \epsilon^{4/3} \|w\|_{W_1^A(M)}^2
$$
by imitation of the convexity argument in \eref{gaf60}.
Choose $\epsilon$ so that $(3/4) \epsilon^{4/3} = 1/2$.
Then $(1/4)\epsilon^{-4} <1$ and \eref{gaf63'} yields
\begin{align*}
\|\n^A \w \|_2^2 + \| \w \|_2^2
&\le \|d_A \w\|_2^2 + \|d_A^*\w\|_2^2 + (1 +\|W\|_\infty) \|\w\|_2^2\\
&+(\gamma)^4 \| \w \|_2^2 + (1/2) \| \w \|_{W_1^A(M)}^2.
\end{align*}
Shift the last term to the left side to deduce  \eref{gaf50'} with
\beq
\lambda_2 = 1 + \|W\|_\infty
  +  ( c\|B\|_2 \kappa^{3/2}  +  \tau^2\| K\|_\infty )^4.    \label{gaf71'}
\eeq
\end{proof}

% Cor for $M\subset \R^3$  is convex.
       \begin{corollary}\label{corGFi'convex} If $M\subset \R^3$  and
  is convex then  one can take
\beq
\lambda_2 = 1 +\kappa^6 (c \|B_0\|_2)^4.     \label{gaf72'}
\eeq
\end{corollary}

              \begin{proof} Comparing \eref{gaf78} with \eref{gaf52} we see that
 we need only set $K$ and $W$ equal to zero in \eref{gaf71'}
 to derive\eref{gaf72'}.
\end{proof}

%Cor GFS'
         \begin{corollary}\label{corGFS'} (Gaffney-Friedrichs-Sobolev inequality.)
Assume  that  dimension $M =3$ and that $M$ has a smooth boundary.
Let $A \in W_1(M)$.
If $\w \in W_1(M; \L^p\otimes \frak k)$ and either $\w_{tan}=0$ or $\w_{norm}=0$
then
\beq
\|\w \|_{L^6(M)}^2
\le \kappa^2 ( \|d_A \w \|_2^2 + \|d_A^* \w \|_{L^2(M)}^2
     + \lambda_2\| \w \| _{L^2(M)}^2)                                  \label{gaf68'}
\eeq
with $\lambda_2$ given by \eref{gaf71'}, in general, or by \eref{gaf72'} if $M\subset \R^3$ and is convex.
\end{corollary}

             \begin{proof}
        Combine \eref{gaf50'} and \eref{gaf54}.
\end{proof}

In the following remark we resume the notation for minimal and maximal operators from Section \ref{secDN}.

\begin{remark}\label{remcoerc}
    {\rm For a $\kf$ valued p-form $\w$ on $M$ with $p=1$ or $2$ define
\begin{align}
Q_N(\w) &= \| D\w\|_2^2 + \| D^* \w \|_2^2 + \| \w\|_2^2,\ \ \
               \w \in \D(D) \cap \D(D^*)              \label{GF83'}\\
Q_D(\w) &= \| d\w\|_2^2 + \| d^* \w \|_2^2 + \| \w\|_2^2,\ \
                        \w \in \D(d) \cap \D(d^*).     \label{GF84'}
\end{align}
Both of these quadratic forms are coercive in the sense that their  domains
are contained in $W_1$ and each controls the $W_1$ norm as in \eref{gaf50} with $A=0$ and $\lambda_3 = \lambda_M$, This is the content of
\cite[Lemma 4.5]{Mo56}. See also \cite{Mt01}. The Laplacians associated to
these closed quadratic forms will be used in Section \ref{secST}.
}
\end{remark}

\section{Sobolev inequalities for solutions} \label{secSobsol}

Throughout this section we will assume that
    $A \in C^\infty((0,T) \times M:\L^1\otimes \frak k)$, with $T \le \infty$,  and  satisfies
\begin{align}
          A'(s) = - \delta_{A(s)} B(s) \ \  \text{on}\ \ (0,T),    \label{B8}
 \end{align}
 where $\delta_A$, defined in \eref{C7} and \eref{C11}, is to be interpreted as
 a differential operator without boundary conditions.
 We will also assume that either
  \begin{align}
 (D)\ \ \ \ \, \ \ A(s)_{tan} &=0\ \ \  \text{for}\ \  0 < s <T       \label{B9}\\
\text{or}\ \  (M)\ \ \ \  B(s)_{norm}&=0  \ \ \  \text{for}\ \  0 < s <T   \label{B10}
 \end{align}

      We are going to establish apriori estimates for solutions to
 the Yang-Mills heat equation  \eref{B8} over $(0,T)$. It will be necessary to
 integrate by parts in Lemma \ref{lemId2} and the use of the maximal or minimal operators $D_A$ or $d_A$  and their Hilbert space adjoints
 will be a very useful bookkeeping tool for this.
 The gauge invariant Sobolev inequalities in Hodge format,
 established in Section \ref{secGFS},
 simplify when  applied to a form $\w$  which  is annihilated by any one of these four operators.
        In particular, when $A(\cdot)$ is a solution to \eref{B8},
 all Sobolev estimates can be conveniently expressed in terms
 of  the time derivatives $A^{(n)}$ or $ B^{(n)}$.

% Subsection:  pointwise and integral identities.

\subsection{Pointwise and integral identities} \label{secSobsol1}

%Lemma {Id1}    (Pointwise identities)
                 \begin{lemma} \label{lemId1} $($Pointwise Identities.$)$
    Suppose that $A(\cdot)$ is a smooth solution to the differential equation \eref{B8}
  and satisfies either \eref{B9} or \eref{B10}.  Then the following identities
  hold, wherein the symbol $d_A$ is  the minimal operator
  in case the Dirichlet boundary condition \eref{B9} is assumed,
   or represents the maximal operator $D_A$
  in case the Marini boundary condition \eref{B10} is assumed.
\begin{align}
B' &= d_A A'        \label{B31}\\
A'' + d_A^* B' &= -[A'\lrc B]  \label{B32}\\
d_A^* A' &= d_A^* A''  = 0\label{B36}
\end{align}
\end{lemma}

              \begin{proof}
Let us first compute the derivatives in all cases, ignoring
boundary conditions, but recalling that $\delta_A = d_A^*$
  in all cases, aside from boundary conditions.  \eref{B31} follows from the definition of $B$.
Differentiate \eref{B8} with respect to  $s$ to derive \eref{B32}.
 By \eref{B8} we have $d_A^* A' = -(d_A^*)^2 B =
-[B\lrc B] = 0$, which is half of \eref{B36}. Differentiate this identity
 with respect to $s$ to find
$0=(\p /\p s)(d_A^*A') = d_A^* A'' +[A'\lrc A'] =  d_A^* A'' - [A'\cdot A'] =
 d_A^* A'' $, since $[A'\cdot A'] = 0$. This proves  \eref{B36}.

        Concerning the boundary conditions, consider first the Dirichlet case,
  \eref{B9}. Since $A(s)_{tan} = 0$ for all $s \in (0, T)$ we may differentiate
  this equation   with respect to $s$ at a point on $\p M$ and find $A'(s)_{tan}=0$.
  Thus the application of the minimal operator $d_A$ in \eref{B31} is justified.
  Since $d_A^*$ is a maximal operator there is no boundary issue in
  \eref{B32} or \eref{B36}.

  In the Marini case $d_A$ is now the maximal operator $D_A$.
   So there is no domain issue in \eref{B31}.
   We may differentiate the equation
  \eref{B10} with respect to time to find $B'(s)_{norm} = 0$. By \eref{C19N}
  $B(s)$ and $B'(s)$ are therefore both in the domain of the minimal
  operator $D_A^*$.   Thus all the terms in \eref{B32} are well defined.
  Moreover \eref{DN53} shows that $D_A^* B$ is again in the domain of  the minimal operator $D_A^*$.  From this and \eref{B8}  it follows that
  $A'$ is in the domain of $D_A^*$ and  from \eref{C19N}  it now follows
   that $A'(s)_{norm} =0$.  Of course then  $A''(s)_{norm}=0$ also and so
   $A''(s) \in \D(D_A^*)$.
   This justifies the identities in \eref{B36}.
 \end{proof}

 % Lemma:  {Id2} Integral identities.     Version 7/4/09
             \begin{lemma} \label{lemId2} $($Integral Identities.$)$
  Suppose that    $A(\cdot)$ is a smooth solution to
   the differential equation  \eref{B8} and satisfies
    either \eref{B9} or  \eref{B10}. Then
\begin{align}
  (d/ds) \|B(s) \|_2^2
                             & = -2 \|A'(s)\|_2^2,      \label{5.21} \\
  (d/ds) \|A'(s)\|_2^2
                             &= - 2\|B'(s) \|_2^2   -2 ([A'(s) \wedge A'(s)], B(s)). \label{5.22}
\end{align}
\end{lemma}

                   \begin{proof}
   It was emphasized in Lemma \ref{lemId1} that, whether one assumes
  Dirichlet or Marini  boundary conditions, $B$ and its time derivatives
  as well as $A'$ and its time derivatives all lie in the domain of the corresponding minimal operators $d_A$ or $D_A^*$, respectively, and
  of course in the domain of the corresponding maximal
   operators $d_A^*$ or $D_A$. All of the integrations by parts implicit in the following computations are thereby justified under either boundary condition \eref{B9} or \eref{B10}. We will write the proof for the Dirichlet
    boundary condition. This uses the minimal operator $d_A$.
  But the proof is identical for the Marini boundary
   condition \eref{B10}. One need only replace $d_A$ by the
    maximal operator $D_A$.
  \begin{align*}
 (1/2)(d/ds) \|B(s)\|_{L^2}^2 &= (B', B)\\
 & = (d_A A', B)\\
 &= (A', d_A^* B)\\
 &= -\|A'(s)\|_{L^2}^2.
 \end{align*}
 This proves  \eref{5.21}.
 In view of \eref{B32} and \eref{B31} we have
  \begin{align*}
 (1/2)(d/ds) \|A'(s)\|^2 & = (A''(s), A'(s)) \\
      &= (-d_A^* B' - [A' \lrc B], A')\\
      &=- (B', d_A A') -([A' \lrc B], A'),
      \end{align*}
 which proves \eref{5.22}.
\end{proof}

% This ends the subsection on identities.

% Subsection: Sobolev inequalities for solutions

\subsection{Sobolev inequalities for smooth solutions} \label{secSobsol2}

The derivation of the Sobolev inequalities \eref{Sob1} and \eref{Sob2} relies on use of more differentiability
than is available from the definition of strong solution. We will assume therefore that $A(\cdot) $ is
a smooth solution. But it will be shown in Corollary \ref{corapstrong}, by an approximation
 procedure for strong solutions, that \eref{Sob1} holds for all strong solutions.
 It can also be shown that
  \eref{Sob2}  holds for  strong solutions. But the proof relies
  on higher order apriori estimates which will not be needed in this paper.

 %Lemma 3:  Sob. ineqs. for solutions.
\begin{lemma}   \label{lemSob}
$($Sobolev inequalities for smooth solutions.$)$
Suppose that $A(\cdot)$  is a smooth solution to   \eref{B8}
and satisfies  either \eref{B9} or \eref{B10}.
       Then there is a Sobolev constant $\kappa$
   and,  for each $s \in (0, T)$, a constant $\lambda(s)$,
 depending only on $\| B(s)\|_3$,  $($cf. \eref{gaf50} and \eref{gaf51}$)$,
 or on $\| B(s)\|_2$, $($cf. \eref{gaf50'} and \eref{gaf71'}$)$,
  such that, suppressing  $s$,
    \begin{align}
\|B\|_6^2 &\le \kappa^2( \|A'\|_2^2 +  \l \|B\|_2^2)             \label{Sob1}\\
\|A'\|_6^2& \le \kappa^2 ( \|B'\|_2^2 +\l \|A'\|_2^2)             \label{Sob2}
\end{align}
\end{lemma}

           \begin{proof}
  All of these inequalities follow from the inequality
\beq
\|\omega\|_6^2 \le
\kappa^2 (\|d_A \omega\|_2^2
         +\|d_A^* \omega\|_2^2  + \l \|\omega\|_2^2)             \label{Sob5}
\eeq
  in the presence of an identity that simplifies one of the terms.
The inequality \eref{Sob5} itself, which is valid for
 $A\in W_1$ and $\w \in W_1$, follows from \eref{gaf68}, with
$\lambda(s) = \lambda_3$,  or, respectively, \eref{gaf68'},
with $\lambda(s) = \lambda_2$. As in the preceding subsection, the symbol $d_A$ represents the minimal operator in the case of Dirichlet boundary conditions, \eref{B9}, or the maximal operator $D_A$ in the case of
 Marini boundary conditions, \eref{B10}.

Thus, in order to derive \eref{Sob1} set $\w = B$ in \eref{Sob5}
and  observe that $d_A B= 0$ by
 Bianchi's identity  \eref{DN72} (for Dirichlet case) or \eref{DN70}
  (for the Marini case),  while $d_A^* B = -A'$.
  Similarly, in order to derive \eref{Sob2} from \eref{Sob5} choose
  $\w = A'(s)$ in \eref{Sob5} and  observe that $d_A^* A' = 0$
   by \eref{B36} while $d_A A' = B'$.
  \end{proof}

%    This ends the  subsection ``Sobolev inequalities for solutions.''

 %%%%%    BEGIN FINITE ACTION
\section{Apriori estimates for finite action} \label{secfa}
Throughout this section $A(\cdot)$ will again denote a solution to
 the Yang-Mills heat equation \eref{B8} which is
in $C^\infty((0,T) \times M)$. We impose in advance no restriction on the behavior
of $A(t)$ as $t\downarrow 0$.

   \begin{definition}\label{def-fa}
{\rm  We say that a solution $A(\cdot)$ to \eref{B8} satisfying either
Dirichlet or Marini boundary condtions, \eref{B9}, \eref{B10},
respectively, has {\it finite action} if
   \beq
   \alpha(t) \equiv \int_0^t s^{-1/2} \| B(s) \|_2^2\ ds < \infty\
                            \text{for some}\ t >0.                         \label{fa6}
   \eeq
}
\end{definition}
Observe that if $\alpha(t) < \infty$ for some $t >0$, then $\alpha(t) < \infty$
for all $t <T$ because, by  \eref{5.21},
 $\|B(s)\|_2^2$ is non-increasing in $s$.

%%%%%%%%%%Lemma 1    Order zero in B for finite action.
\subsection{Order 0}    \label{secfa1}

     \begin{proposition}\label{fa-ord0} $($Order zero in $B$.$)$ If $A(\cdot)$ has finite action then
\beq
 t^{1/2} \| B(t) \|_2^2 + 2 \int_0^t s^{1/2}  \| A'(s) \|_2^2 ds
 = (1/2)\alpha(t)                                                                \label{fa10}
 \eeq
 \end{proposition}

              \begin{proof}
   For $s > 0$ multiply the identity
   $(d/ds) \| B(s)\|_2^2 = - 2 \| A'(s) \|_2^2$, from \eref{5.21},
    by $ s^{-1/2}$    to obtain
   $$
   (d/ds) (s^{-1/2} \|B(s) \|_2^2)
          = -(1/2) s^{-3/2} \| B(s) \|_2^2 - 2 s^{-1/2} \|A'(s) \|_2^2.
   $$
   Let $ 0 <\sigma <t$. Integrate the last identity over $(\sigma, t)$
   to find
   \begin{align*}
   t^{-1/2}& \| B(t) \|_2^2   + 2 \int_\sigma^t s^{-1/2}  \| A'(s) \|_2^2 ds \\
   &=\sigma^{-1/2} \| B(\sigma) \|_2^2
   -(1/2) \int_\sigma^t s^{-3/2} \|B(s) \|_2^2 ds
     \end{align*}
   We can now integrate this identity with respect to $\sigma$ over the interval $(0, t)$
   and reverse the order of the $s$ and $\sigma$ integrals to arrive at \eref{fa10}.
  \end{proof}

  %%%%%%%Lemma 2 $\|B(s)\|_3^2$ is integrable, for finite action.
             \begin{lemma}\label{fa-ord0lem} If $A(\cdot)$ has finite action then
  \beq
  \int_0^t \| B(s)\|_3^2 ds < \infty \ \ \text{for all} \ \ t >0
  \eeq
  Explicitly,
  \beq
  \int_0^t \| B(s) \|_3^2 ds \le \alpha(t)
  \{  \alpha(t) \kappa^3 c^2 +(1 + 2 t \lambda_M)^{1/2}  \}
         (\kappa/2),                                                            \label{fa21}
  \eeq
  where $\lambda_M$ is defined in \eref{gaf67} and
  depends only on the geometry of $M$.
  \end{lemma}

                    \begin{proof}
 Interpolation shows that
      $\|B(s) \|_3^2 \le \|B(s)\|_2 \| B(s) \|_6.$
 Hence
 \begin{align}
 \int_0^t \| B(s) \|_3^2 ds
 &\le \int_0^t \{ s^{-1/4} \|B(s)\|_2 \} \{ s^{1/4} \|B(s) \|_6 \} ds \notag\\
 & \le \{ \int_0^t s^{-1/2} \|B(s) \|_2^2 ds \}^{1/2}
       \{ \int_0^t s^{1/2} \|B(s) \|_6^2 ds \}^{1/2} .            \label{fa22}
 \end{align}
 By \eref{Sob1}, with
 $\lambda = \lambda_3(s) =\lambda_M + (\kappa c)^2 \|B(s)\|_3^2$,
  we have
 \beq
 \int_0^t s^{1/2} \| B(s) \|_6^2 ds \le \kappa^2 \int_0^t s^{1/2} \{\|A'(s) \|_2^2
        + \lambda_3(s) \| B(s) \|_2^2 \} ds.       \label{fa25}
 \eeq
 The inequality \eref{fa10} shows that
 $\int_0^t s^{1/2} \| A'(s) \|_2^2 ds \le (1/4) \alpha(t)$ and also \linebreak
 $s^{1/2} \|B(s) \|_2^2 \le (1/2) \alpha(s) \le (1/2) \alpha(t)$,
   which, in view of \eref{fa25}, yields
 \begin{align*}
 \int_0^t s^{1/2} \| B(s)\|_6^2 ds
 & \le  \kappa^2\{ (1/4)\alpha(t) + \int_0^t (1/2) \alpha(t) \lambda_3(s) ds \}\\
 & = (\kappa^2 \alpha(t)/4)\{ 1 + 2 \int_0^t \lambda_3(s) ds \}\\
 & = (\kappa^2 \alpha(t)/4)
     \{1 + 2t\lambda_M + 2(\kappa c)^2 \int_0^t \|B(s) \|_3^2 ds \}.
 \end{align*}
 Inserting this into \eref{fa22}, and mindful of \eref{fa6}, we find
 \beq
 \int_0^t \|B(s) \|_3^2 ds \le (\kappa \alpha(t)/2)
  \{ 1 + 2t\lambda_M + 2(\kappa c)^2 \int_0^t \|B(s)\|_3^2 ds \}^{1/2}.   \label{fa28}
 \eeq
 But if $b,c$ and $y$ are nonnegative numbers and $y^2 \le by +c$ then
 $(y -b/2)^2 \le (b/2)^2 + c \le (b/2 + c^{1/2})^2$, showing that $y \le b + c^{1/2}$.

  Let $y = \int_0^t \|B(s)\|_3^2 ds$ and square the inequality
 \eref{fa28} to find
 $y^2 \le by + c$ with $b = (\kappa \alpha/2)^2 2(\kappa c)^2$ and
  $c = (\kappa \alpha/2)^2 (1 + 2t\lambda_M)$. The inequality $y \le b + c^{1/2}$
  is then \eref{fa21}.
 \end{proof}

        \begin{notation}\label{notpsi}
{\rm  For a solution $A(\cdot)$ of \eref{B8} of finite action let
 \beq
 \psi_s^t = (t-s) \lambda_M
 + 2 (\kappa c)^2\int_s^t \| B(\sigma) \|_3^2 d\sigma              \label{fa31}
 \eeq
 for $0 \le s \le t < T$ and let $\psi(t) = \psi_0^t$.
      By Lemma \ref{fa-ord0lem} we know that $\psi_s^t < \infty$ for
     $0 \le s \le t < T $.
 }
 \end{notation}

 \begin{remark}{\rm
 If $A(\cdot)$ has finite energy, i.e., $ \|B_0\|_2 < \infty$, then \linebreak
  $ \alpha(t) \le \int_0^t s^{-1/2}  \|B_0\|_2^2 = 2t^{1/2} \| B_0\|_2^2$ because
 $\|B(s)\|_2^2$ is non-increasing, by \eref{5.21}. It follows, therefore, from \eref{fa21} that
 $\int_0^t \| B(s) \|_3^2 ds$ is bounded by a function of $t$ and $\| B_0\|_2$.
In particular
\beq
\psi(t) \le \beta(t, \|B_0\|_2)                                  \label{fa33}
\eeq
 for some jointly increasing continuous function
 $\beta$ on $[0,\infty)^2$ that depends
 only  on the geometry of $M$.

 It seems worth mentioning here that
 one can also prove   $ \int_0^t \| B(s)\|_2^4 ds \le (1/2) \alpha(t)^2$
 under the assumption  of finite action, i.e., when $\alpha(t) <\infty$.
 This can provide an alternative approach to some estimates because
 of the quartic appearance of $\|B\|_2$ in  \eref{gaf71'}.
 }
 \end{remark}

\subsection{Order 1}     \label{secfa2}

       \begin{proposition}\label{fa-ord1} $($Order one in $B$.$)$
 Suppose that $A(\cdot)$  is a solution of  finite action.   Then
  \beq
  t^{3/2} \| A'(t) \|_2^2 + \int_0^t e^{\psi_s^t} s^{3/2} \| B'(s) \|_2^2 ds
  \le (3/8) e^{\psi(t)} \alpha(t)   \ \ \ \text{for}\ 0 < t < T  .               \label{fa32}
  \eeq
  \end{proposition}

               \begin{proof}
      H\"older's inequality  and \eref{Sob2} give
        \begin{align*}
       2 |( [ A'(s) \wedge A'(s) ], B(s))|
&\le 2c \| B(s)\|_3 \|    A'(s) \|_2 \| A'(s) \|_6\\
&\le   \kappa^2c^2 \|B(s) \|_3^2 \| A'(s) \|_2^2 + \kappa^{-2} \| A'(s)\|_6^2 \\
&\le (\kappa c)^2 \| B(s)\|_3^2 \|A'(s)\|_2^2
       + \lambda_3(s) \| A'(s)\|_2^2   + \|B'(s)\|_2^2,
        \end{align*}
   where
 $ \lambda_3(s) = \lambda_M +(\kappa c)^2 \|B(s)\|_3^2$.
    The total coefficient of $\|A'(s)\|_2^2$ on the right is therefore at most
    $\psi'(s)$, by the definition \eref{fa31}.Thus
  \[
  -\| B'(s) \|_2^2 - 2 ( [ A'(s) \wedge A'(s) ], B(s)) \le \psi'(s) \| A'(s)\|_2^2.
  \]
  Hence \eref{5.22} now yields
 $ (d/ds) \| A'(s) \|_2^2  \le - \| B'(s)\|_2^2  + \psi'(s) \| A'(s)\|_2^2,$
and therefore
\beq
    (d/ds) \Big( e^{-\psi(s)} \| A'(s) \|_2^2 \Big )
                           \le - e^{-\psi(s)} \| B'(s) \|_2^2.                  \label{fa36}
    \eeq
  Consequently
     $$(d/ds) \Big( s^{1/2} e^{-\psi(s)} \| A'(s) \|_2^2 \Big )
     + s^{1/2} e^{-\psi(s)} \| B'(s) \|_2^2
     \le (1/2) s^{-1/2} e^{-\psi(s)} \| A'(s) \|_2^2.
    $$
 Upon integrating this inequality  from $\sigma >0$ to $t$ we find
  \begin{align*}
  t^{1/2} e^{-\psi(t)} \| A'(t)\|_2^2 &
  + \int_\sigma^t s^{1/2} e^{-\psi(s)} \| B'(s) \|_2^2 ds\\
  & \le (1/2) \int_\sigma^t s^{-1/2} e^{-\psi(s)} \| A'(s) \|_2^2 ds.
  +\sigma^{1/2} e^{-\psi(\sigma)} \| A'(\sigma) \|_2^2.
  \end{align*}
  We may now integrate this inequality over $(0, t)$ with respect to $\sigma$
  and interchange the $\sigma$ and $s$ integrals  to deduce
\begin{align*}
  &t^{3/2} e^{-\psi(t)} \| A'(t)\|_2^2
  + \int_0^t s^{3/2} e^{-\psi(s)} \| B'(s) \|_2^2 ds \\
  &\le (3/2) \int_0^t s^{1/2} e^{-\psi(s)} \| A'(s) \|_2^2 ds\\
  &\le (3/8) \alpha(t),
  \end{align*}
  wherein we have used \eref{fa10} in the last line. Multiply by $e^{\psi(t)}$ to arrive at \eref{fa32}.
 \end{proof}

% BEGIN FINITE ENERGY

\section{Apriori estimates for finite energy}    \label{secfe}

We will need  to understand the nature of the singularities of
the various gauge covariant spatial derivatives  of $B(t)$ as $ t\downarrow 0$
under the sole assumption of finite initial energy.
The word ``order'', below,  refers to the number of spatial derivatives
    of $B$ involved in the inequalities. For example
    $A'$ involves one spatial derivative of $B$ by virtue of
    the equation $A' = - d_A^* B$.

    All of the estimates in Sections \ref{secSobsol} and \ref{secfa} as well as those
    in the next theorem and corollary require use of Dirichlet (D) or Marini (M)
     boundary conditions. But in Section \ref{secfe4} we will have to replace Marini boundary conditions
     by Neumann boundary conditions, which are stronger.

    \begin{theorem} \label{thmfe}  Let $ 0 < T \le \infty$. Suppose that
$A(\cdot)$ is a smooth solution to \eref{B8} and  satisfies
either \eref{B9} or \eref{B10}. Assume further that $\|B_0\|_2 < \infty$.
Then $\|B(t)\|_2$ is non-increasing and there exist continuous
non-decreasong functions $C_j :[0,\infty)^2 \rightarrow [0,\infty)$,
for $j = 1,2,3$, such that
\begin{align}
\| B(t)\|_2^2 + 2 \int_0^t \| A'(s)\|_2^2 ds
        &= \|B_0\|_2^2,\  \qquad  \qquad      \,      \text{Order}\  0   \label{fe5}\\
t \| A'(t) \|_2^2 +  \int_0^t e^{\psi_s^t} s \|B'(s) \|_2^2ds
        &\le C_1( t, \| B_0\|_2)    \qquad     \       \text{Order}\ 1    \label{fe6}
\end{align}
and
\begin{align}
\int_0^t \| B(s)\|_6^2 ds
          & \le C_2(t, \|B_0\|_2),   \qquad           \text{Order}\  0           \label{fe11}   \\
t \|B(t)\|_6^2 + \int_0^t e^{\psi_s^t} s \| A'(s)\|_6^2 ds
         &\le C_3(t, \|B_0\|_2),    \qquad             \text{Order}\ 1            \label{fe12}
\end{align}
where $\psi_s^t$ is defined in \eref{fa31}.
\end{theorem}

\begin{corollary}\label{corfe45} For each $p \in [2,6)$,
 there is a continuous, non-decreasing
function $C_4:[0,\infty)^2 \rightarrow [0,\infty)$ such that
\beq
\int_0^t \| A'(s)\|_p ds \le C_4(t, \|B_0\|_2).\qquad \qquad  \text{Order}\ 1    \label{fe80}
\eeq
\end{corollary}

The proofs will be given in the next subsection.

\subsection{Proofs of finite energy apriori estimates.}   \label{secfe1}

 Throughout these proofs we will use the Sobolev inequalities
 \eref{Sob1}, \eref{Sob2} with $\lambda = \lambda(s) = \lambda_2(s)$,
 which is given by  \eref{gaf71'}.
 $\lambda_2(s)$ depends only on $\|B(s)\|_2$ and
  is non-increasing in  $s$ because $\|B(s)\|_2$ is non-increasing, by virtue of
  \eref{5.21}.
   Since, in this section, we are  interested  only in finite energy
  initial data we will simplify inequalities by using
 \beq
 \lambda_2(s) \le \lambda_2(0) \equiv \lambda_0         \label{gfs2}
 \eeq
 where
 $\lambda_0 = 1 +\|W\|_\infty +(c\|B_0\|_2 \kappa^{3/2}
 +\tau^2 \| K \|_\infty )^4$
  in general and is given by
  $\lambda_0 = 1 +(c\|B_0\|_2)^4 \kappa^6$
  if $M$ is a convex subset of $\R^3$.
  In particular, making the choice $\lambda(s) = \lambda_2(s)$ in
  \eref{Sob1} and \eref{Sob2} we will use these inequalities in the form
  \begin{align}
\|B(s)\|_6^2 &\le \kappa^2( \|A'(s)\|_2^2 + \lambda_0 \|B_0\|_2^2),   \label{005}\\
\|A'(s)\|_6^2 & \le \kappa^2 ( \|B'(s)\|_2^2 + \lambda_0 \|A'(s)\|_2^2) \label{006}
\end{align}
for any smooth solution over $(0,T)$.

 \bigskip

% Proposition 0   (Order 0)

\begin{proof}[Proof of \eref{fe5} and \eref{fe11}] \label{ord0prop}
         Integrate the identity \eref{5.21} over $(0,t)$ to find
$
\int_0^t \| A'(s)\|_2^2 ds = (1/2) (\|B_0 \|_2^2 - \|B(t) \|_2^2),
$
 which is \eref{fe5}.
 Taking the integral of the inequality \eref{005} over $(0,t)$ and using  \eref{fe5}
 we find
\beq
\int_0^t\|B(s)\|_6^2 ds
            \le  \kappa^2\|B_0\|_2^2 \{ 1/2 +  t\l_0\},             \label{002}
\eeq
which proves \eref{fe11}.
 \end{proof}

\bigskip

\begin{proof}[Proof of \eref{fe6}] \label{ord1prop}
   Integrate \eref{fa36} over $(\sigma,t)$ and multiply by $e^{\psi(t)}$ to find
 \beq
\| A'(t) \|_2^2 + \int_\sigma^t  e^{\psi_s^t}  \|B'(s)\|_2^2 d s
           \le e^{\psi_\sigma^t} \| A'(\sigma)\|_2^2  \le e^{\psi(t)} \| A'(\sigma)\|_2^2.                 \label{110}
\eeq
   Integrate this inequality with respect to $\sigma$ over $(0, t)$,
     reverse the order
    of integration in the double integral and then apply \eref{fe5}
     to deduce
    \beq
t\|A'(t)\|_2^2  + \int_0^t e^{\psi_s^t} s\|B'(s) \|_2^2 ds
              \le e^{\psi(t)} \|B_0\|_2^2/2,      \label{111}
\eeq
which implies \eref{fe6}.
 \end{proof}

 \bigskip

%    [Proof of \eref{fe12}]
  \begin{proof}[Proof of \eref{fe12}]       \label{ord1c1}

  Add $\lambda_0 \|B_0\|_2^2
  + \int_\sigma^t e^{\psi_s^t} \lambda_0 \|A'(s)\|_2^2 ds$
  to both sides of \eref{110} and use \eref{005} and \eref{006} to find
   \begin{align}
  \kappa^{-2} \Big\{\|B(t)\|_6^2 &+ \int_\sigma^t e^{\psi_s^t} \| A'(s)\|_6^2 ds\Big\} \notag\\
  &\le  e^{\psi(t)} \|A'(\sigma)\|_2^2 + \lambda_0 \|B_0\|_2^2
  + \int_\sigma^t e^{\psi_s^t} \lambda_0 \|A'(s) \|_2^2 ds \label{fe21}\\
  &\le  e^{\psi(t)} \|A'(\sigma)\|_2^2 + \lambda_0 \|B_0\|_2^2
  + e^{\psi(t)} \lambda_0 \|B_0\|_2^2/2,  \notag
  \end{align}
  wherein we have used $e^{\psi_s^t} \le e^{\psi(t)}$ and
   \eref{fe5} for the last term.
      Integrate with respect to $\sigma$ over $(0, t)$,
   reverse the $\sigma$ and $s$ integrals on the left
   and  apply \eref{fe5} to the first term on the right, to arrive at
      \begin{align*}
  t \| B(t)\|_6^2 &+ \int_0^t e^{\psi_s^t} s \| A'(s)\|_6^2 ds\\
  &\le \kappa^2\Big\{ e^{\psi(t)} \| B_0\|_2^2 /2 + \lambda_0 t \| B_0\|_2^2 + e^{\psi(t)} \lambda_0 (t/2) \|B_0\|_2^2\Big \},
  \end{align*}
  which proves \eref{fe12}.
 \end{proof}

 The proof of Corollary \ref{corfe45} depends on the following
  interpolation lemma.

\begin{lemma} \label{lem23.1} (Interpolation.)
 Let $0\le a < b <\infty$ and let $ 2 \le p <6$. Suppose that
 $f:(a,b) \rightarrow L^2(M)\cap L^6(M)$ is continuous.
 Then
 \beq
 \int_a^b \| f(s)\|_p ds   \le \Big( \int_a^b s^{\frac{3}{p} - \frac{3}{2}} ds\Big)^{1/2}
     \Big(      \int_a^b \| f(s)\|_2^2 ds \Big)^{\alpha/p}
            \Big(  \int_a^b s \|f(s)\|_6^2 \Big)^{3\beta/p}                 \label{sth90}
\eeq
where $ p = 2\alpha + 6 \beta$, $\alpha + \beta = 1$ and $0\le \beta <1$.
\end{lemma}

                      \begin{proof}
  By interpolation $ \|f(s)\|_p \le \|f(s)\|_2^{2\alpha/p} \|f(s)\|_6^{6\beta/p}$.
   Hence
$$
     \int_a^b \| f(s)\|_p ds
     \le \int_a^b \{ s^{-3\beta/p}\}\{ \|f(s)\|_2^{2\alpha/p} \}
              \{ (s^{1/2} \| f(s) \|_6)^{6\beta/p} \} ds.
     $$
  Apply H\"older's inequality  to the product of the three functions
      in braces to find
\begin{align*}
  &   \int_a^b \| f(s)\|_p ds\\
     & \le \Big(\int_a^b\{s^{-3\beta/p}\}^q ds \Big)^{1/q}
       \Big(\int_a^b\| f(s)\|_2^{2\alpha r/p}ds\Big)^{1/r}
        \Big(\int_a^b\{s^{1/2} \|f(s)\|_6\}^{6\beta m/p} ds \Big)^{1/m}
\end{align*}
     provided $q,r,m$ are nonnegative and $q^{-1} + r^{-1} + m^{-1} = 1$.
     Choose $ q=2$, $r= p/\alpha$ and $m = p/(3\beta)$  and observe that
     $6\beta/p = (3/2) - (3/p)$ to arrive at
     \eref{sth90}.
\end{proof}

\begin{proof}[Proof of Corollary \ref{corfe45}] Choose $(a,b) = (0,t)$ and
 $f(s, x) = |A'(s,x)|_{\Lambda^1\otimes \frak k}$ in Lemma \ref{lem23.1}.
 Since $p<6$ the exponent in the first factor is $(3/p) -( 3/2)  > -1$.
  Therefore the first factor on the  right in       \eref{sth90} is finite.
  The second and third factors on the right are also finite, by \eref{fe5}
  and \eref{fe12} respectively.
    \end{proof}

\subsection{Growth of $\| A(t)\|_{W_1(M)}$.}    \label{secfe4}

In the previous sections all apriori estimates were gauge invariant.
However for our proof of long time existence of solutions
 we will need estimates that depend on $A_0$ itself, not just on
its gauge equivalence class.
Correspondingly,  we will have to  replace Marini boundary conditions by the stronger
 Neumann boundary conditions in order to get estimates on $A(t)$ itself, not just
 on  certain of its derivatives.

 The smoothness hypothesis in the following
theorem, that $A(\cdot) \in C^\infty((0,T) \times M)$, will be removed in Section \ref{secLTE}, Corollary \ref{growthstrong}.

  \begin{theorem} \label{thmM6} There is a continuous increasing function
$C_5: [0,\infty)^2 \rightarrow [0, \infty)$,
depending only on the geometry of $M$,
 such that, for any strong solution
to the Yang-Mills heat equation satisfying Neumann, \eref{N1}, \eref{N2},
or Dirichlet, \eref{D1}, boundary conditions on an interval  $[0,T)$,
with $0 < T \le \infty$,  there holds
\beq
\| A(t)\|_{W_1(M)} \le C_5(t, \| A_0\|_{W_1(M)}), \ 0 \le t <T,      \label{M30}
\eeq
under the additional hypothesis that $ A(\cdot) \in C^\infty((0,T) \times M)$.
\end{theorem}

The proof depends on the following estimates, which will be derived for
smooth $A(\cdot)$. But the smoothness requirement will be removed in
Corollary \ref{corapstrong}, thereby proving the following four inequalities for any
strong solution satisfying Neumann or Dirichlet boundary conditions.

\begin{lemma}\label{lemM5} Suppose that $A(\cdot)$ is a strong solution
to  the Yang-Mills heat equation satisfying Neumann, \eref{N1}, \eref{N2},
or Dirichlet, \eref{D1}, boundary conditions on an interval  $[0,T)$,
with $0 < T \le \infty$. Assume also that $A \in C^\infty((0,T))$. Then
\begin{align}
\| A(t) \|_2
          &    \le \| A_0\|_2 + t^{1/2} \|B_0\|_2 ,                            \label{M81'}\\
\|A(s)\|_4
        &  \le  \| A_0\|_4  +C_4(t, \|B_0\|_2), \ 0 <s \le t,               \label{M78'}\\
\| dA(t)\|_2
       &\le \| B_0\|_2 +(c/2) \Big( \|A_0\|_4 + C_4(t, \|B_0\|_2)\Big)^2 \label{M79'}\\
 \text{and}\ \ \
 \| d^* A(t) \|_2
 &\le \| d^* A_0\|_2
   +c\Big( \|A_0\|_4 + C_4(t, \|B_0\|_2)\Big) C_4(t, \|B_0\|_2),        \label{M80'}
\end{align}
where $C_4$ is defined by \eref{fe80} for $p =4$.
\end{lemma}

        \begin{proof}
The identity
\beq
 A(s) = A_0 + \int_0^s A'(\sigma) d \sigma,                \label{M77'}
 \eeq
 is valid for any strong solution, even if not smooth on $(0,T)$.
  We may take the $L^2$ norm  in  \eref{M77'} (with $s =t$)
  to  find $ \| A(t) \|_2
  \le \| A_0 \|_2 + \int_0^t \| A'(\sigma) \|_2 d\sigma
                 \le \| A_0\|_2 + t^{1/2} (\int_0^t \|A'(\sigma)\|_2^2 )^{1/2} $.
  \eref{M81'} now follows from \eref{fe5}.

   The rest of the proof hinges on the estimate \eref{fe80} for $p=4$,
  which asserts
\beq
   \int_0^t \| A'(\sigma)\|_4 d\sigma \le C_4( t, \|B_0\|_2)      \label{M55'}
   \eeq
   for some non-decreasing continuous function
   $C_4:[0, \infty)^2 \rightarrow [0,\infty)$.
   Now \eref{M77'} implies that
 $  \|A(s)\|_4  \le \| A_0\|_4 + \int_0^s \| A'(\sigma)\|_4 d\sigma$,
 which proves \eref{M78'} in view of \eref{M55'}.
          Observe next the identities
   \begin{align}
   dA(t) &= B(t) - (1/2) [A(t) \wedge A(t)]                   \label{M72'}\\
   d^* A'(s)& = [ A(s)\cdot A'(s)].                                \label{M73'}
  \end{align}
      The first just rewrites the definition of curvature \eref{ymh3},
 while the second just rewrites the first identity in \eref{B36}.
 It follows that
 \begin{align}
 \| dA(t)\|_2 &\le \| B(t)\|_2 +(c/2) \| A(t)\|_4^2,  \label{M74'}
 \end{align}
 which yields \eref{M79'} upon insertion of \eref{M78'},
 given that $\|B(t)\|_2$ is non-increasing.
  Finally, the identity \eref{M73'} gives
\beq
  d^* A(t) = d^* A_0 + \int_0^t [ A(s)\cdot A'(s)] ds, \label{M75'}\\
\eeq
and therefore
\begin{align}
 \| d^* A(t) \|_2 &\le \| d^* A_0\|_2
          + c \int_0^t \| A(s) \|_4 \| A'(s)\|_4 ds                                \label{M76'}\\
&\le \| d^* A_0\|_2 + c\sup_{0<s \le t} \|A(s)\|_4 \int_0^t\| A'(s)\|_4 ds.\notag
\end{align}
\eref{M80'} now follows from \eref{M78'} and \eref{M55'}. Notice that
$\|d^* A_0\|_2 < \infty$  because, by assumption,
$A(\cdot)$ maps $[0, T)$ into $W_1$.
\end{proof}

  Note: The proof of \eref{M55'}  relies on
   use of third spatial derivatives of $A$ in the identity \eref{5.22}, and therefore is
   not immediately applicable to a strong solution.
   Moreover the identity \eref{M73'} and its consequences,
   \eref{M75'} and \eref{M76'}, also uses the third spatial derivatives of $A$.
   However we will construct
   in Section \ref{secLTE} an approximation method  that allows us to prove
   \eref{fe80}, and in particular \eref{M55'}, as well as \eref{M76'},  for all strong solutions and indeed
   with the same function $C_4(\cdot, \cdot)$.
   This entire proof   will  then apply to all strong solutions without the additional hypothesis that $A(\cdot) \in C^\infty((0,T)) \times M)$.

\begin{proof}[Proof of Theorem \ref{thmM6}]
We are going to make use of the Gaffney-Friedrichs inequality \eref{gaf50}
  with $A=0$ in that inequality and $\w$ chosen to be the form $A(t)$ of the
  present theorem, with $t >0$. In this case \eref{gaf50} asserts that
  \beq
  (1/2) \| A(t)\|_{W_1}^2
  \le \| d A(t) \|_2^2 + \| d^* A(t) \|_2^2 + \lambda_M \| A(t) \|_2^2 \label{M70}
  \eeq
  This is applicable because $A(t) \in W_1(M)$ and either
  $A(t)_{norm} =0$ or $A(t)_{tan}=0$.
   The three terms on the right may be estimated by \eref{M79'}, \eref{M80'}
   and \eref{M81'} respectively. The theorem now follows
   if one takes into account
   that $\|A_0\|_2, \|A_0\|_4, \|B_0\|_2$ and $\|d^* A_0\|_2$ are all dominated by
   a linear or quadratic polynomial in $\|A_0\|_{W_1}$, given the definition
   \eref{ymh2} of the $W_1$ norm.
  \end{proof}

%%%%%%%%%%%%  New section:  Short time existence

\section{Short time existence and uniqueness for the parabolic equation} \label{secST}

 In this section we will prove Theorem \ref{thmpara} for both sets
  of boundary conditions \eref{ST11N} and \eref{ST11D} simultaneously
  by encoding the boundary conditions into appropriate Sobolev spaces
  and then using a common approach.The Sobolev spaces
 will be the quadratic form domains of the absolute and relative Laplacians,
 \cite{Co}, \cite{RaS},
 as described in Remark \ref{remcoerc}.

 \begin{notation}\label{notSob}{\rm
  Define
\begin{align}
(N)\ \ \ \ \Delta_N &=-(D^*D + D D^*) \label{ST17}\\
\text{or}\ \  (D)\ \ \ \  \Delta_D & = -(d^* d + d d^*)      \label{ST18}
\end{align}
Here $D$ and $d$ are the maximal and minimal exterior derivative operators, respectively, discussed in Section \ref{secDN}. They act on p-forms.

     For both kinds of boundary conditions we are going to write
   simply $H_1(M)$ (or $H_1(M;\L^1\otimes \frak k)$ when clarity demands)
     for the form domain of $\Delta_N$ or $\Delta_D$, namely the domains,
     respectively,  of the quadratic forms $Q_N$ or $Q_D$ in
      \eref{GF83'} and \eref{GF84'}.
      This defines two distinct notions of $H_1$.
       Thus, writing $\Delta$ for either the absolute or relative Laplacian
     $\Delta_N$ or $\Delta_D$,
         Remark \ref{remcoerc}
         allows us to write the Sobolev norm as
     \beq
     \|\w\|_{H_1} = \| (1- \Delta)^{1/2} \|_{L^2(M)}          \label{ST19}
     \eeq
     in both cases.
            We remind the reader that Remark \ref{remcoerc}
   shows that
      a form $\w \in W_1(M)$ is in the Neumann version of $H_1(M)$
      if and only if $\w_{norm}=0$ and is in the Dirichlet version of $H_1(M)$
      if and only if $\w_{tan}=0$.

     Throughout this section we will write $d$ for the exterior derivative
     with the understanding that this represents the maximal or minimal
     version, in agreement with the boundary conditions.
}
\end{notation}

Recall that we are dealing with a product bundle and may therefore
apply this definition to $A$ itself.

In the next section we will separate out the non-linear terms in
 the parabolic equation \eref{ST11} and reformulate it as an
 integral equation in a more or less standard way.
        A natural abstract setting for producing solutions to the
   integral equation may be found, for example,  in  \cite[Chapter 15]{Tay3}.
 But we are going to use the following modified path  space within which
 to seek solutions in order to get some precise regularity at the same time.

%Begin notation for Path space norm.
            \begin{notation} \label{notST7}{\rm (Path space.)
 Let $0<T <\infty$. Denote by  $\P_T$  the set of continuous functions
 $$
 C:[0,T] \rightarrow H_1(M)
 $$
 such that $\|C(t)\|_\infty $, $\|dC(t)\|_\infty$ and  $\|d^*C(t)\|_\infty$
 are finite for each $t >0$ and
  \begin{align}
 \infty > \|C\|_{\P_T}
    \equiv \sup_{0 < t \le T} \Big\{&\|C(t)\|_{H_1(M)}
     +t^{1/4} \|C(t)\|_{L^\infty(M)} \notag \\
     &+t^{3/4}\Big(\|d C(t)\|_\infty + \| d^* C(t)\|_\infty\Big) \Big\}.    \label{ST23}
 \end{align}
 }
      \end{notation}
      Notice that the last two terms are well defined because the
  boundary conditions on $C(t)$ agree with the choice of $d$ as a
 minimal or maximal operator.

%begin Theorem EUP

    \begin{theorem} \label{thmEUP} Let $A_0 \in H_1(M)$ and suppose that
    $\beta \ge \|A_0\|_{H_1(M)}$. Then there exists
$T >0$ depending only on $\beta$ such that the integral equation
  \eref{ST25} has a solution in $\P_T$.
  The solution is unique in $\P_T$.
Moreover  the solution is strongly differentiable for $t >0$ as a function into $L^2(M)$.
For $t >0$, $C(t)$ is in $\D(\Delta)$ and
\eref{ST11}  holds.
Further,  the solution lies in $C^\infty((0,T)\times M; \L^1\otimes \frak k)$.
\end{theorem}

\subsection{The integral equation and strong solutions}\label{secST1}

We will prove Theorems \ref{thmEUP} and \ref{thmpara} in this section.

     We are going to operate mostly with the integral form of the equation
     \eref{ST11} as follows.
Writing $ B \equiv B_C = dC +(1/2) [C\wedge C]$, we can compute that
\beq
d_C^* B + d_Cd^*C = (d^*d + d d^*)C - X(C),        \label{ST13}
\eeq
where $X$ is the first order nonlinear differential operator
 on $\frak k$ valued
1-forms defined by
\beq
-X(C) = - [C\lrc B] +(1/2) d^*[C\wedge C] + [ C, d^*C], \ \ \ \
       C:M \rightarrow \Lambda^1\otimes \frak k                       \label{ST14}
\eeq
   The term $d_C d^*C$ in \eref{ST13} contributes
the term $dd^*$ to the second order operator on the right, thereby making
the operator on the right elliptic. Without this term the equation \eref{ST11}
would be only weakly parabolic.

The terms in $X(C)$ which are cubic in $C$ involve no derivatives  of $C$
while the terms which are quadratic all involve a factor of one spatial
derivative of $C$.
We may write this symbolically as
\beq
X(C) = C^3 + C \cdot \p C.                                     \label{ST14.1}
\eeq
$X(C)$ contains all the non-linear terms in Eq.\ \eref{ST11}, which can now be rewritten as
\beq
C'(t) = \Delta C(t) + X(C(t)), \ \ C(0) = A_0,     \label{ST15}
\eeq
wherein $\Delta$ is given by \eref{ST17} or \eref{ST18}.

Informally, the equation \eref{ST15} is equivalent to the integral equation
\beq
C(t) = e^{t\Delta} A_0
+ \int_0^t e^{(t-\sigma)\Delta} X(C(\sigma)) d\sigma.   \label{ST25}
\eeq
The regularity lemma, Lemma \ref{lemST9}, will show that \eref{ST25} implies \eref{ST15}.

%begin Lemma ST3b
       \begin{lemma}\label{lemST3b}
Let  $C(\cdot) \in \P_T$. Define
\beq
F(\sigma) = C(\sigma)^3 + C(\sigma)\cdot \p C(\sigma).    \label{ST46}
\eeq
and  define $F_j(\sigma)$ similarly for paths $C_j$, $ j = 1,2$.
Let $2 \le q \le \infty$.
Suppose that $\|C\|_{\P_T} \le R$ and $\|C_j\|_{\P_T} \le R$.
Then there are  constants $a_k$ independent of
$C(\cdot)$, $q$, $R$ and $T$  such that, for $0 < \sigma <T$,
\begin{align}
\| F(\sigma)\|_q &\le \sigma^{-(3/2)(\frac{1}{2} - \frac{1}{q})}
\{( R^3a_1) + \sigma^{-1/4} (R^2a_2)\}   \label{ST47}\\
\|F_1(\sigma) - F_2(\sigma)\|_q
&\le  \sigma^{-(3/2)(\frac{1}{2} - \frac{1}{q})}
        \|C_1 - C_2\|_{\P_T}\{( R^2a_3) + \sigma^{-1/4}(Ra_4)\}  \label{ST48}
\end{align}
\end{lemma}

                \begin{proof}
In the interpolation inequality $\| f\|_b \le \|f\|_a^{a/b} \|f\|_\infty^{1 - (a/b)}$
for $ 1 \le a \le b$ chose $a =2, b = q$ to find
$\|f\|_q \le \|f\|_2^{2/q} \|f\|_\infty^{1 - (2/q)}$. Take $f = | \p C(\sigma)|$
to deduce
\begin{align*}
\| \p C(\sigma) \|_q
   &\le \| \p C(\sigma) \|_2^{(2/q)} \| \p C(\sigma) \|_\infty^{1 - (2/q)} \\
&\le \| C(\sigma) \|_{H_1}^{(2/q)} \Big( \sigma^{-3/4} \| C\|_{\P_T}\Big)^{1- (2/q)} \\
&\le \sigma^{-(3/4)(1 - (2/q))} \|C\|_{\P_T},
\end{align*}
from which follows
        \begin{align*}
\| C(\sigma) \cdot \p C(\sigma) \|_q
  \le c\| C(\sigma)\|_\infty \| \p C(\sigma) \|_q
\le \sigma^{-1/4} \sigma^{-(3/4)(1 - (2/q))} c\|C\|_{\P_T}^2.
\end{align*}
Thus the term $C(\sigma)\cdot \p C(\sigma)$ in $F(\sigma)$ is correctly estimated  by the  second term on the right of  \eref{ST47}.
Now choose $a =6, b = 3q$ to find
$ \| f\|_{3q} \le \|f\|_6^{2/q} \| f\|_\infty^{1 - (2/q)}$
and take $f = |C(\sigma)|$ to deduce
        \begin{align*}
\| C(\sigma)^3\|_q &\le c^2 \{\| C(\sigma) \|_{3q}\}^3 \\
&\le c^2 \{ \| C(\sigma) \|_6^{(2/q)} \| C(\sigma) \|_\infty^{ 1- (2/q)} \}^3\\
&\le c^2 \{ (\kappa \| C(\sigma)\|_{H_1} )^{(2/q)}
 ( \sigma^{-1/4} \| C\|_{\P_T})^{1-(2/q)} \}^3 \\
 &= c^2 \{ \kappa^{(2/q)} \sigma^{-(1/4)(1 - (2/q))} \| C\|_{\P_T}\}^3,
  \end{align*}
  which completes the verification of \eref{ST47}.
  The proof of \eref{ST48} proceeds the same way but for differences
  in this cubic polynomial.
\end{proof}

\begin{remark} {\rm We will need to use some heat kernel estimates for the
absolute and relative Laplacians on a compact n-dimensional Riemannian
manifold with smooth boundary.  If $\Delta$ denotes either of these
Laplacians then  $e^{t\Delta}$ is given by  an integral kernel  $K_t(x,y)$,
and
$t^{n/2}|K_t(x,y)|+t^{(n+1)/2}| \text{grad}_x K_t(x,y)|$ is  bounded
on any interval $0<t \le T$.
See \cite[Proposition 5.3]{RaS} for a proof.

     It follows by interpolation that, in three dimensions, given $T_0 \in (0, \infty)$,
     there is a constant $c_1$ depending only on $T_0$ such that,
     for $ 1 \le q \le p \le \infty$ and $0 < t \le T_0$,
\begin{align}
\|e^{t\Delta}\|_{q\rightarrow p}
       & \le c_1 t^{-(3/2)((1/q)- (1/p))}  ,                                   \label{hk1}\\
 \|\p e^{t\Delta}\|_{q\rightarrow p}
      &\le c_1t^{-1/2} t^{-(3/2)((1/q)- (1/p))}  ,\ \
                          \text{with}\ \p = d\  \text{or}\ \p = d^*,         \label{hk4}\\
 \| e^{t\Delta}\|_{L^2\rightarrow H_1} &\le c_1 t^{-1/2}.         \label{hk2}
\end{align}
(\eref{hk2} actually follows directly from the spectral theorem.)
In particular,  each of the following are bounded by $c_1= c_1(T_0)$
 for $0<t \le T_0$.
\beq
t^{3/4} \|e^{t\Delta}\|_{2\rightarrow \infty}, \ \qquad
t^{1/4} \| e^{t\Delta} \|_{6\rightarrow \infty},\qquad\
t^{1/4} \|e^{t\Delta} \|_{3/2 \rightarrow 2}   . \label{hk3}
\eeq
}
\end{remark}

%begin Lemma ST7

             \begin{lemma}\label{lemST7} Let $0 < T_0 <\infty$.
   There is a constant $c_0$ depending on $T_0$ such that
     for any $ T \in (0, T_0]$ and any
$A_0 \in H_1(M)$ the path
$ [0,T] \ni t \mapsto C_0(t)\equiv e^{t\Delta} A_0$
lies in $\P_T$ and
\beq
\| C_0(\cdot) \|_{\P_T} \le c_0 \| A_0\|_{H_1}     \label{ST101}
\eeq
\end{lemma}

                         \begin{proof}
  Since $e^{t \Delta}$ is a contraction in the $H_1$ norm \eref{ST19},
   we have   $\|C_0(t)\|_{H_1} \le \| A_0\|_{H_1}$. Furthermore, by \eref{hk3},
  $$
  t^{1/4} \| C_0(t) \|_\infty = t^{1/4} \| e^{t\Delta} A_0\|_\infty
  \le t^{1/4} \| e^{t\Delta}\|_{6\rightarrow \infty} \|A_0\|_6
  \le  c_1 \kappa \| A_0\|_{H_1}
$$
  Writing $\p = d$ or $d^*$, the last two terms in \eref{ST23}
  are dominated  for the path
  $C_0(\cdot)$ in accordance with   the inequalities
   $$
   \|\p C_0(t) \|_\infty
  = \| \p e^{t \Delta} A_0\|_\infty
  \le \| e^{t \Delta} \|_{2\rightarrow \infty} \|A_0 \|_{H_1}
  \le   c_1t^{-3/4} \|A_0\|_{H_1}.
 $$
      Multiply by $t^{3/4}$ and add to the previous two inequalities
  to arrive at   \eref{ST101}.
  \end{proof}

  %begin new lemma lemST8a for strong sols.
               \begin{lemma}\label{lemST8a}
               Let $0<T <\infty$ and $0 < \alpha <1$.
      There is a constant $c_{T,\alpha}$ such that, for all $\epsilon >0$,
    \begin{align}
      \| (e^{\epsilon\Delta} -1) e^{s\Delta} \|_{L^2 \rightarrow H_1}
 &\le \epsilon^\alpha s^{-\frac{1}{2} - \alpha} c_{T,\alpha}\ \
           \text{for}\ \ 0 <s \le T,            \label{ST118}\\
    \| (e^{\epsilon\Delta} -1) e^{s\Delta} \|_{L^2 \rightarrow L^\infty}
&\le \epsilon^\alpha s^{-\frac{3}{4} - \alpha} c_{T,\alpha}\  \
         \text{for}\ \ 0 <s \le T.       \label{ST119}
    \end{align}
  \end{lemma}

                       \begin{proof}
            Let $E = (1 -\Delta)^{1/2}$ and let $b>0$.
           We assert that  there are constants $c_\alpha$ and $\hat c_{b,T}$
            such that
    \beq
  \| E^{-2\alpha} (1 -e^{\epsilon \Delta}) \|_{2\rightarrow 2}
    \le \epsilon^\alpha c_\alpha\ \ \text{and}\ \
    \| E^{2b} e^{s\Delta} \| _{2\rightarrow 2}
   \le s^{-b} \hat c_{ b,T}
   \eeq
for all $\epsilon >0$ and the specified ranges of $\alpha$ and $s$.
The first follows from the spectral theorem for $-\Delta$ and the inequalities
$ (1+x)^{-\alpha} (1 - e^{-\epsilon x}) = (1 +\epsilon^{-1} y)^{-\alpha}(1 - e^{-y})
  \le \epsilon^{\alpha} c_\alpha$, which holds for all $x >0$, wherein we have put
  $y = \epsilon x$.
   The second follows similarly from the inequalities
   $(1+x)^b e^{-sx} = (1+ s^{-1} y)^b e^{-y} \le s^{-b} \hat c_{b,T}$,
    wherein we have put $y = sx$.

 Defining  $2b = 1 +2\alpha$ in the second line below, we see that
    \begin{align*}
 \| (e^{\epsilon\Delta} -1) e^{s\Delta} \|_{L^2 \rightarrow H_1}
& = \| E(e^{\epsilon\Delta} -1) e^{s\Delta} \|_{L^2 \rightarrow L^2}\\
&\le \| E^{-2\alpha} (e^{\epsilon \Delta} -1) \| _{2\rightarrow 2}
      \| E^{1+2\alpha} e^{s\Delta}\|_{2\rightarrow 2}\\
&\le \{\epsilon^\alpha c_\alpha\} \{ s^{ -\frac{1}{2} - \alpha} \hat c_{b, T}\},
\end{align*}
which proves \eref{ST118}. Choosing next $b =\alpha$, we see that
\begin{align*}
    \| (e^{\epsilon \Delta} - 1) e^{s\Delta} \|_{L^2 \rightarrow L^\infty}
&\le \| e^{(s/2)\Delta} \|_{2\rightarrow \infty}
\| E^{-2\alpha} (e^{\epsilon  \Delta} - 1)\|_{2\rightarrow 2}
   \| E^{2\alpha} e^{(s/2) \Delta}\|_{2\rightarrow 2} \\
&\le \{c_T s^{-3/4}\}          \{\epsilon^\alpha c_\alpha\}
          \{ (s/2)^{-\alpha} \hat c_{\alpha, T} \},
\end{align*}
where
$c_T = \sup_{0<s \le T} s^{3/4} \|e^{(s/2) \Delta}\|_{2\rightarrow \infty} < \infty$
in three dimensions by \eref{hk3}.
\end{proof}

%%%%% Lemma ST8b
      \begin{lemma}\label{lemST8b} (H\"older continuity.)
  Suppose that $C(\cdot) \in \P_T$ and $\|C\|_{\P_T} \le R$.
   Let $ 0 < \alpha < 1/4$
   and let $0 < a <T$. Define
  \beq
  \rho(t) = \int_0^t e^{(t-\sigma)\Delta} F(\sigma) d\sigma.     \label{ST120}
   \eeq
   Then there is a constant $c_2$ depending only on  $a$ and $\alpha$ and on
   $R$ and $T$ such that
   \beq
   \| \rho(t) - \rho(r) \|_\infty + \| \rho(t) - \rho(r) \|_{H_1} \le c_2 (t-r)^\alpha \ \text{for}\  a \le r <t < T.       \label{ST121}
   \eeq
   If, moreover, $C(\cdot)$ is a solution to the integral equation \eref{ST25}, then
   $[a, T) \ni \sigma \mapsto F(\sigma) \in L^2(M)$ is H\"older continuous
   of order $\alpha$.
      \end{lemma}

                         \begin{proof}
   Taking $ 0 < a \le r <t <T$, we may write
   \beq
   \rho(t) - \rho(r) = \int_0^r \Big( e^{(t-r)\Delta} - 1\Big)
    e^{(r-\sigma)\Delta} F(\sigma) d\sigma
         + \int_r^t e^{(t-\sigma)\Delta} F(\sigma) d\sigma.
   \eeq
   We need to estimate the $H_1$ norm and $L^\infty$ norm of each of
   these two integrals. In all four integrals we will use
    \eref{ST47} with $q=2$, namely
   \beq
   \| F(\sigma)\|_2 \le (R^3 a_1) + \sigma^{-1/4} (R^2 a_2),
                     \ 0<\sigma <T,                                           \label{ST124}
   \eeq
    from which follows,
    $ c_{a,T} \equiv \sup_{ a \le \sigma <T} \| F(\sigma) \|_2 < \infty.$
            Using \eref{ST118} and \eref{ST119}
 with $ \epsilon = t-r$ and $s = r-\sigma$, as well as \eref{hk2}, we find
\begin{align*}
\| \rho(t) - &\rho(r) \|_{H_1} \le \int_0^r \| ( e^{(t-r)  \Delta} -1)e^{(r-\sigma)\Delta} \|_{L^2 \rightarrow H_1} \| F(\sigma)\|_2 d\sigma\\
& \ \ \ \   \qquad    + \int_r^t \| e^{(t-\sigma)\Delta} \|_{L^2\rightarrow H_1}d\sigma
       \sup_{a \le \sigma <T} \|F(\sigma)\|_2 \\
 &      \le (t-r)^\alpha c_{T, \alpha} \int_0^r (r-\sigma)^{-\frac{1}{2} - \alpha} \| F(\sigma)\|_2 d\sigma + \int_r^t (t-\sigma)^{-1/2} d\sigma\  c_{a,T} \\
 &\le (t-r)^\alpha c_3 + (t-r)^{1/2} c_4.
\end{align*}
\eref{ST124} shows that $c_3 <\infty$ if $\alpha <1/2$.
Similarly, by \eref{ST119} and \eref{hk3},
\begin{align*}
\| \rho(t) - &\rho(r)\|_\infty
\le \int_0^r \| (e^{(t-r)\Delta} -1) e^{(r-\sigma) \Delta } \|_{2\rightarrow \infty}
            \| F(\sigma)\|_2 d\sigma \\
&\ \ \ \ \ \ \ \ \ \  + \int_r^t  \| e^{(t-\sigma)\Delta} \|_{2\rightarrow \infty}
         \sup_{a \le \sigma <T} \|F(\sigma)\|_2\\
&\le (t-r)^\alpha c_{T,\alpha} \int_0^r (r-\sigma)^{-\frac{3}{4} - \alpha} \| F(\sigma)\|_2 d\sigma + \int_r^t (t-\sigma)^{-3/4} c d \sigma\  c_{a,T}\\
&\le  (t-r)^\alpha c_5 +  (t-r)^{1/4} c_6.
\end{align*}
In view of \eref{ST124}, the constant $c_5 < \infty$ if $\alpha < 1/4$.
This proves \eref{ST121}.

   Now \eref{ST121} shows that the integral term in \eref{ST25}
      is H\"older continuous on $[a, T)$
      into $L^\infty\cap H_1$ in the sum norm.
          So is the term $e^{t\Delta} A_0$ as one sees from the inequalities
 $ \| (e^{t\Delta} - e^{r\Delta}) A_0 \| _{H_1}
           \le (t-r)^\alpha r^{-\frac{1}{2} - \alpha} C_{T, \alpha} \|A_0\|_2$ and
     $ \| (e^{t\Delta} - e^{r\Delta}) A_0 \| _\infty
      \le (t-r)^\alpha r^{-\frac{3}{4} - \alpha} C_{T, \alpha} \|A_0\|_2$,
     which follow from \eref{ST118} and \eref{ST119}, respectively.
     Hence $ [a,T) \ni \sigma \mapsto C(\sigma) \in L^\infty\cap H_1(M)$ is bounded
      and H\"older continuous of order $\alpha$. Therefore the
     term $C(\sigma)\cdot \p C(\sigma)$ in $F(\sigma)$ is H\"older continuous
      into $L^2(M)$  while the term $C(\sigma)^3$ is Holder continuous into
      $L^\infty(M)$   and therefore into $L^2(M)$.
     \end{proof}

% Lemma ST9, new version 8/16/09
    \begin{lemma}\label{lemST9} $($Strong solution.$)$
 Suppose that $C(\cdot)$ is a solution to the integral equation
  \eref{ST25} lying in $\P_T$. Define $\rho(t)$ by \eref{ST120}.
  Then, for $t >0$,   $\rho(t) \in \D(\Delta)$ and is strongly
  differentiable as a function into $L^2(M)$. Moreover
  \beq
  \rho'(t) = \Delta \rho(t) + F(t)                    \label{ST131}
  \eeq
  In particular $C(t) \in \D(\Delta)$ for $t >0$. $C(\cdot)$  is strongly differentiable
 on $(0,T)$ into $L^2(M)$, and the differential equations \eref{ST15}
 and \eref{ST11} both hold.
  \end{lemma}

                       \begin{proof}
         For $a \le s <t$ define
 \beq
 \rho_s(t) = \int_0^s e^{(t-\sigma)\Delta} F(\sigma) d\sigma.   \notag
 \eeq
  Since $t -\sigma \ge t- s >0$ for all $\sigma$ in the integrand, $\rho(t)$ is in $\D(\Delta)$ and
  \beq
  \Delta\rho_s(t) = \int_0^s \Delta e^{(t-\sigma) \Delta} F(\sigma) d\sigma
    \ \ \             \text{for}\ \ a \le s <t                \notag
  \eeq
  We are going to show that  $\Delta\rho_s(t)$ converges in $L^2(M)$
  as $\epsilon \equiv t-s \downarrow 0$ and in fact uniformly for
  $ t \in [a,b] \subset (0,T)$.
  Observe first that if $ a \le s_1 <s_2 < t$ then, for $0< \alpha < 1/4$,
  and with $c_7$ denoting the H\"older constant for $F(\sigma)$ on $[a, T)$
  into $L^2(M)$,
 \begin{align*}
  \| \Delta \int_{s_1}^{s_2} e^{(t-\sigma)\Delta} (F(\sigma) - &F(t))d\sigma \|_2
   \le \int_{s_1}^{s_2} \| (t-\sigma) \Delta
        e^{(t-\sigma)\Delta} \frac{F(\sigma) - F(t)}{t-\sigma} \| d\sigma\\
   &\le \int_{s_1}^{s_2} \| (t-\sigma) \Delta e^{(t-\sigma)\Delta} \|_{2\rightarrow 2} \frac{\| F(\sigma) - F(t)\|_2}{t-\sigma} d\sigma \\
   &     \le \int_{s_1}^{s_2} c_7 (t -\sigma)^{\alpha -1} d\sigma \rightarrow  0
   \end{align*}
   as $ s_1 < s_2$ both increase to $t$, and uniformly  for
   $t \in [a,b]\subset (0,T)$.    Therefore,
   \begin{align*}
   \| \Delta ( \rho_{s_2}(t) - &\rho_{s_1}(t) ) \|_2 \\
   &= \| \int_{s_1}^{s_2} \Delta e^{(t-\sigma)\Delta} F(t) d\sigma + \int_{s_1}^ {s_2} \Delta e^{(t-\sigma)\Delta}(F(\sigma) - F(t)) d\sigma \|_2 \\
   &\le \| e^{(t-\sigma)\Delta}|_{s_1}^{s_2} F(t) \|_2 +o(1) \rightarrow 0
   \end{align*}
   as $s_1 < s_2 \uparrow t$. Moreover, since $F(t)$ is continuous
    on $[a,T)$ into $L^2(M)$, we may conclude that
  $e^{\epsilon \Delta} F(t) - F(t) \rightarrow 0$ uniformly for $t \in [a,b]$.
  Clearly $\rho_s(t) \rightarrow \rho(t)$ in $L^2(M)$ as $s\uparrow t$ and uniformly for $t \in [a,b]$. Since $\Delta$ is a closed operator it now follows that
  $\rho(t) \in \D(\Delta)$ and $t \mapsto \Delta \rho(t)$ is continuous into $L^2$.

    To prove \eref{ST131} observe that for  $ 0 <r \le t_0  \le t$ we have
  \begin{align*}
  & \rho(t) - \rho(r) = \int_r^t e^{(t-\sigma)\Delta} F(\sigma) d\sigma + \int_0^r (e^{(t-\sigma)\Delta} - e^{(r-\sigma)\Delta}) F(\sigma)d\sigma \\
  & = \int_r^t e^{(t-\sigma)\Delta} F(t_0) d\sigma
             + \int_r^t e^{(t-\sigma)\Delta}(F(\sigma) - F(t_0)) d\sigma
             + (e^{(t-r)\Delta} - 1) \rho(r) \\
   &= \Delta^{-1} ( e^{(t-r) \Delta} -1) F(t_0)
             +\int_r^t e^{(t-\sigma)\Delta}(F(\sigma) - F(t_0)) d\sigma
              + (e^{(t-r)\Delta} - 1) \rho(r).
   \end{align*}
   Divide by $t-r$ and note that as $t-r \downarrow 0$ one has
    $$
   (t-r)^{-1} \Delta^{-1} (e^{(t-r)\Delta} -1) F(t_0) \rightarrow F(t_0),
   $$
   while
   $$ (t-r)^{-1} \| \int_r^t e^{(t-\sigma)\Delta} (F(\sigma) - F(t_0))d\sigma \|_2
   \le (t-r)^{-1} \int_r^t \| F(\sigma) - F(t_0)\|_2 d\sigma \rightarrow 0.
   $$
      Moreover
   $$
    (t-r)^{-1} (e^{(t-r)\Delta} -1) \rho(r) = (t-r)^{-1} \Delta^{-1} ( e^{(t-r)\Delta} -1) \Delta \rho(r) \rightarrow \Delta \rho(t_0)
    $$
     because $ r \mapsto \Delta \rho(r)$ is      continuous into $L^2$.
     This proves \eref{ST131}.

     Now $C(t) = e^{t\Delta}A_0 + \rho(t)$ by \eref{ST25} and \eref{ST120}.
   Both terms are in the domain of $\Delta$  for $t>0$ and both are differentiable
   on $(0,T)$ into $L^2(M)$. The equation $C'(t) = \Delta C(t) + F(t)$ now follows
   from \eref{ST131}.
       We may rearrange the terms in  \eref{ST15} to deduce that
  the differential equation \eref{ST11} holds. We will show explicitly in the next
  corollary that $B_{C(t)} \in W_1(M)$, which is implicit in \eref{ST11},
  the rearranged   version of \eref{ST15}.
   \end{proof}

   % Boundary conditions
     \begin{corollary}\label{corST10} $($Boundary conditions.$)$
Under the hypotheses of Lemma \ref{lemST9}, $DC(t)$ and $D^*C(t)$,
resp. $dC(t)$ and $d^*C(t)$,
are in $W_1(M)$ for $t >0$
in the Neumann, resp. Dirichlet cases, as is also $B_{C(t)}$.
 Moreover,
$C$ satisfies the following respective boundary conditions for $t >0$.
\begin{align}
&(N)\  C(t)_{norm}=0,\  (DC(t))_{norm} =0, \  (B_{C(t)})_{norm} =0. \label{ST140}\\
&(D)\  C(t)_{tan} =0, \ \ \ (dC(t))_{tan} =0,\  (B_{C(t)})_{tan} =0,\
                   (d^*C(t))_{tan}=0.                                                   \label{ST141}
\end{align}
\end{corollary}

                \begin{proof}
Writing $d$ for both the minimal and maximal operators, we see that
  in both cases $C(t) \in \D(d^*d) \cap \D(dd^*)$ for $t >0$ by Lemma
   \ref{lemST9}. We may apply Proposition \ref{propDN4} with
    $A=0$ and therefore    $B=0$.
    Take $\w = C(t)$ in that proposition. Since $C(t) \in \D(d^*d)$
we have $C(t) \in \D(d)$ while  $dC(t) \in \D(d^*)$. But  also $dC(t) \in \D(d)$
by \eref{DN50} in case (N) or by \eref{DN51} in case (D). Therefore
$dC(t) \in \D(d^*)\cap \D(d)$. By the Gaffney Friedrichs inequality \eref{gaf50}
it now follows that $dC(t) \in W_1(M)$.
    The same argument applies to $d^*C(t)$, upon use of \eref{DN52}  and \eref{DN53} since $C(t) \in \D(dd^*)$. Thus $d^* C(t) \in W_1$ also.
    Further, since $C(t)$ is bounded for each $t>0$ and in $W_1$,
     it follows that $[ C(t)\wedge C(t) ]$ is in $W_1$ and so,
     therefore, is $B_{C(t)}$.
 This proves the first assertion of the corollary.

   Concerning the boundary conditions \eref{ST140} and \eref{ST141},
   there is a slight difference in the two cases and we will therefore
   distinguish between the minimal and maximal operators $d$ and $D$
   in a repeated application of Lemma \ref{lemDN1}.

      In case  (N), since
    $C(t) \in \D(D^*) \cap W_1$, \eref{C19N} shows that $C(t)_{norm} = 0$.
   Since also $DC(t) \in \D(D^*) \cap W_1$, \eref{C19N} also
   shows that $(D C(t))_{norm} =0$.  But
   $(B_{C(t)})_{norm} = (DC(t))_{norm} +(1/2) [ C(t) \wedge C(t) ]_{norm}
   = 0 + [C(t)_{norm} \wedge C(t)] = 0$. This establishes \eref{ST140}.

   In case (D), since $C(t) \in \D(d)\cap W_1$, \eref{C19D} shows that
   $C(t)_{tan}=0$. And, since $d^*C(t) \in \D(d)\cap W_1$, \eref{C19D} also show that $(d^*C(t))_{tan} =0$. This proves two of the equalities in \eref{ST141}.
   Taking now $\w = C(t)$ in Proposition \ref{propDN4} we see that \eref{DN51}
    implies $dC(t) \in \D(d)$ and, since $dC(t) \in W_1$, \eref{C19D}
     shows that     $(dC(t))_{tan} =0$.
   Finally,
   $(B_{C(t)})_{tan}= (dC(t))_{tan} +(1/2) [ C(t)_{tan} \wedge C(t)_{tan} ] = 0$.
    \end{proof}

\bigskip
 % Begin proof of E and U for parabolic equation.
  \begin{proof}[Proof of Theorem \ref{thmEUP}]
  Let $A_0 \in H_1(M)$, choose  $T_0 =1$ in Lemma \ref{lemST7},
  and let  $c_0$ be as described in that lemma.
  Choose $R > 2c_0 \|A_0\|_{H_1}$.
  For $C(\cdot) \in \P_T$ define
  \beq
  W(C)( t) = e^{t\Delta} A_0 + \int_0^t e^{(t-\sigma)\Delta} F(\sigma) d\sigma,
         \ \ 0 \le t \le T.
  \eeq
  We will show that for $T$ sufficiently small $W$ takes
  \beq
  \P_{T,R} \equiv \{ C \in \P_T: \|C\|_{\P_T} \le R\}
  \eeq
  into itself and is a strict contraction on this  set. $\P_{T,R}$ is non-empty by
  Lemma \ref{lemST7} for any $T \le 1$.
  Observe first that, by \eref{ST47} with $q=2$, we have, for $0 \le t \le T \le 1$,
  \begin{align}
  \int_0^t \| e^{(t-\sigma)\Delta} F(\sigma) \|_{H_1} d\sigma
  &\le \int_0^t \| e^{(t-\sigma)\Delta}\|_{L^2\rightarrow H_1}
                  \| F(\sigma)\|_2 d\sigma          \notag\\
  &\le \int_0^t (t-\sigma)^{-1/2}c_1
              \{ (R^3 a_1) +\sigma^{-1/4} (R^2 a_2)\} d\sigma ,     \label{ST105}
  \end{align}
  while, for any $q \in [2,\infty]$,
   \begin{align}
  &\int_0^t \| e^{(t-\sigma)\Delta} F(\sigma) \|_\infty d\sigma
          \le \int_0^t \| e^{(t-\sigma)\Delta}\|_{q\rightarrow \infty}
                     \| F(\sigma)\|_q d\sigma      \notag\\
  &\le \int_0^t (t-\sigma)^{-(3/2q)}c_1
         \sigma^{-(3/2)( \frac{1}{2} - \frac{1}{q})}
                    \{ (R^3 a_1) +\sigma^{-1/4} (R^2 a_2)\} d\sigma    \label{ST106}
  \end{align}
  by \eref{hk1} and
  \begin{align}
  \int_0^t &\|\p e^{(t-\sigma)\Delta} F(\sigma)\|_\infty d\sigma
  \le \int_0^t \| \p e^{(t-\sigma)\Delta} \|_{q\rightarrow \infty}
                 \| F(\sigma) \|_q d\sigma     \notag\\
  &\le  c_{q,\infty} \int_0^t (t-\sigma)^{-(3/2q) - (1/2)} \sigma^{-(3/2)( \frac{1}{2} - \frac{1}{q})} \{ (R^3 a_1) +\sigma^{-1/4} (R^2 a_2)\}  \label{ST107}
  \end{align}
  by \eref{hk4}.
       Although these inequalities are valid  for any $ q \in [2,\infty]$,
 nevertheless,  for $q=3$, the last integrand has a non-integrable singularity,
 $(t-\sigma)^{-1}$, and for $q \le 3$ it is even worse.
           Moreover for $q = \infty$ two of the four integrands
           in \eref{ST106} and \eref{ST107} contain
     the non-integrable singularity  $\sigma^{-1}$. But use of any
     $q \in (3, \infty)$ will yield usable estimates and in fact will yield
     the   same $t$ dependence of the integrals.
   For simplicity  we will use $q=6$ in these estimates.

  The six explicit $\sigma$ integrals in \eref{ST105} - \eref{ST107}
  may all be done by substituting
  $\sigma =t r$. Choosing $q =6$ in \eref{ST106} and \eref{ST107},
  so that $\sigma^{-(3/2)( \frac{1}{2} - \frac{1}{q})} = \sigma^{-1/2}$,
  and keeping in mind the three different powers of  $t$ dictated by
  the definition \eref{ST23},
  one arrives at six integrals $t^\delta \int_0^t (t-\sigma)^{-\beta} \sigma^{-\gamma}d\sigma = c_{\delta,\beta,\gamma} t^{1+\delta - \beta -\gamma}$
  which are all finite with the choice $q =6$. Choosing $\delta =0$ for
  \eref{ST105}, $\delta = 1/4$ for \eref{ST106} and $\delta = 3/4$ for
  \eref{ST107} and adding, we find
   \begin{align*}
  &\int_0^t \| e^{(t-\sigma)\Delta} F(\sigma) \|_{H_1} d\sigma
  +  t^{1/4}\int_0^t \| e^{(t-\sigma)\Delta} F(\sigma) \|_\infty d\sigma \\
   &+t^{3/4} \int_0^t \| \p e^{(t-\sigma)\Delta} F(\sigma) \|_\infty d\sigma
   \le t^{1/2}\{ c_8 a_1 R^3 \}
   + t^{1/4}\{ c_9 a_2 R^2\}.
   \end{align*}
   Hence, in view of \eref{ST101}, and taking the supremum over $t \in [0,T]$,
   we find
   \beq
   \|W(C)\|_{\P_T} \le c_0 \| A_0\|_{H_1} + T^{1/2} (c_{10} R^3)  \label{ST109}
   + T^{1/4}(c_{11} R^2).
   \eeq
   Thus for $T$ sufficiently small, depending on $R$, the second and third terms on the right add to at most $ R - c_0 \| A_0\|_{H_1}$. Therefore
   $W$ takes $\P_{T,R}$ into itself.

   For two elements $C_1$ and $C_2$ in $\P_{T,R}$ the estimate
    \eref{ST48} yields, just as in the preceding estimates,
   \beq
   \| W(C_1) - W(C_2) \|_{P_T}
   \le \| C_1- C_2 \|_{\P_T} \{ T^{1/2} (c_{16} R^2)
                                  + T^{1/4} (c_{17} R) \} ,                 \label{ST110}
   \eeq
   since the term $e^{t\Delta} A_0$ cancels in the difference.
   The coefficient of $\| C_1- C_2\|_{\P_t}$ may be made less than $1/2$
   by choosing $T$ sufficiently small, depending on $R$. The map
   $W$ has therefore a unique fixed point in $\P_{T,R}$.

      Suppose now that $\hat C$ is another  solution to
   \eref{ST25}
  in $\P_T$.
Let $R_1 = \|\hat C\|_{\P_T}$. Then $R_1 > R$. Choose $T_1 \le T$
corresponding to $R_1$ as in the argument following \eref{ST110} with
 $R$ replaced by $R_1$. By what has just been proven we have uniqueness
 of solutions to \eref{ST25}
 in $\P_{T_1,R_1}$. Since $C$, restricted to $[0,T_1]$, is in $\P_{T_1, R_1}$
 it follows that  $\hat C$ and $C$
 coincide on $[0, T_1]$. We may now apply the same argument
 on the interval $[T_1, 2T_1]$ (using the same $R_1$) to conclude that
 $\hat C$ coincides with $C$ on the entire interval $[0, 2T_1]$. And so on.
 This proves uniqueness of solutions to \eref{ST25}   in $\P_T$.

       Now Lemma \ref{lemST9} shows that the solution $C(t)$
  to the  integral equation \eref{ST25} is actually a solution to
  the differential equation \eref{ST15}.
        We may therefore apply  \cite[Proposition 3.2, page 289]{Tay3} to conclude
 that the solution $C(\cdot)$ is in
 $C^\infty((0,T)\times M; \L^1\otimes \kf)$. Rearranging the terms gives \eref{ST11}.
 \end{proof}

 \bigskip

 \begin{proof}[Proof of Theorem \ref{thmpara}]
  Choose $T\in (0,\infty)$ as in Theorem \ref{thmEUP} and denote by $C(\cdot)$
  the solution to the integral equation \eref{ST25}.
     Then $C(\cdot)$ lies in $\P_T$ and is therefore a continuous function
     from $[0, T)$ into $W_1$. Equation \eref{ST25} shows that $C(0) = A_0$.
     Corollary \ref{corST10} proves that $B_{C(t)}$ and $d^* C(t)$ are
      in $W_1$ for $t >0$, which is
     the claim a) in Theorem \ref{thmpara}, and proves as well that $C(\cdot)$
     satisfies all  the required boundary conditions,
      \eref{ST11N}, resp. \eref{ST11D}.
      For $t \in (0, T)$ Lemma
     \ref{lemST9} shows that $C(t)$ is strongly differentiable into $L^2(M)$
     and that the differential equation \eref{ST11} holds. The smoothness
     of $C(\cdot)$ is proved in Theorem \ref{thmEUP}.
     The boundedness of  $t^{3/4}\|B_{C(t)}\|_\infty$, required in condition
     f) of Theorem \ref{thmpara}, follows from the fact that $C(\cdot)$ lies
     in $\P_T$. Indeed, the norm definition \eref{ST23} shows that,
     for $t \in (0,T)$,
    \begin{align}
  t^{3/4}\| B_{C(t)}\|_\infty
 &\le t^{3/4}\{ \| d C(t) \|_\infty + (c/2) \| C(t) \|_\infty^2 \}      \notag\\
 &\le t^{3/4} \| d C(t) \|_\infty + t^{1/4} (c/2)( t^{1/4} \|C(t)\|_\infty)^2 \notag \\
 &\le \| C\|_{\P_T} +  t^{1/4}(c/2) \| C\|_{\P_T}^2.         \label{ST150}
 \end{align}

   {\bf Uniqueness for the parabolic equation.}
  A standard proof of existence and uniqueness for a semilinear parabolic equation  may be found in \cite[Chapter 15, Section 1]{Tay3}.
  It is based on a simpler path space
  than the space $\P_T$ (see \eref{ST23}) that we have been
   using and relies on simpler estimates:
  Let $\hat \P_T = \{ C(\cdot) \in C( [0,T); H_1(M)):
   \sup_{0 \le t <T} \|C(t) \|_{H_1(M)} < \infty\}$ .
   If $C(\cdot) \in \hat \P_T$ then
   $F(\sigma)$
   (see \eref{ST46}) is a continuous function into $L^{3/2}(M)$ because
   $C(\sigma)^3 \in L^6 \cdot L^6\cdot L^6 \subset L^2(M)$ while
   $C(\sigma)\cdot \p C(\sigma) \in L^6 \cdot L^2 \subset L^{3/2}(M)$.
   Moreover $\| e^{t\Delta} \|_{L^{3/2} \rightarrow H_1}
   \le \| e^{(t/2)\Delta}\|_{L^2 \rightarrow H_1}
   \| e^{(t/2)\Delta}\|_{L^{3/2}\rightarrow L^2} = O(t^{-1/2} t^{-1/4})$
       by \eref{hk2} and \eref{hk3}.
   Since $ 3/4 <1$ the integral equation \eref{ST25} has a unique solution
   in $\hat \P_T$ for a given $A_0 \in H_1$ and small enough $T$.

         Thus if $C(\cdot) \in \hat \P_T$ and is in addition strongly
    differentiable into $L^2(M)$ and satisfies \eref{ST15} then the identity
   $$
    C(t) - e^{t\Delta} C(0) = \int_0^t (d/d\sigma)\Big( e^{(t-\sigma)\Delta} C(\sigma)\Big) d \sigma = \int_0^t e^{(t-\sigma)\Delta} F(\sigma) d\sigma
    $$
   shows that $C(\cdot)$ satisfies the integral equation \eref{ST25} and uniqueness then follows.
   The last integrand is an integrable function into
   $H_1$ because $F:[0,t] \rightarrow L^{3/2}(M)$ is continuous,
    as we have seen above.
   The solution to \eref{ST15} is unique, therefore, under the hypothesis that it is
  continuous and bounded  on $[0,T)$ into $H_1$ and strongly differentiable on $(0,T)$ into $L^2$.
\end{proof}

   At the price of a more complicated proof we have used the
        smaller space $\P_T$ in our existence proof in order to derive
     the regularity properties implicit in the norm \eref{ST23}.
   These regularity properties will transfer over to the induced
  solution $A$, (see Lemma \ref{lemDS} or \ref{lem3a}) of the
  Yang-Mills heat equation and will be important ingredients
  in our uniqueness proof for that weakly parabolic equation.

%New subsection - apriori for parabolic.

\subsection{An apriori estimate for the parabolic equation} \label{secST2}

The apriori estimates in Section \ref{secfe} have parallels for the
 parabolic equation.
 But they get rapidly more complicated for the parabolic equation
 as the order of the inequality increases.
 We will need the following lowest order estimate.  It will not artificially
 decompose the nonlinear terms in \eref{ST11}, as does the method of the
 the previous subsection.

% Lemma 21
            \begin{lemma}\label{lemST31}    Assume that $C(\cdot)$ satisfies the conclusions of Theorem \ref{thmEUP}. Then
  $\| B_{C(t)}\|_2$  is non-increasing on $[0, T)$ and in fact
 \beq
 \|B_{C(t)}\|_2^2 + 2 \int_0^t \| d_{C(s)}^* B_{C(s)}\|_2^2ds
                            = \|B_0\|_2^2.                \label{ST202}
 \eeq
 In particular,
 \beq
 \| B_{C(t)}\|_2 \le \| B_0\|_2.                                         \label{ST203}
 \eeq
 \end{lemma}
            \begin{proof}  For ease in reading define $\beta(t) = B_{C(t)}$.
 For $t >0$, $\beta(t)$ is in the domain of $d_{C(t)}^*$ and therefore
  in the domain of the square of this operator,
  by \eref{DN52} and \eref{DN53}, since $[ B\lrc B] =0$.
 In fact these identities show that
 $(d_{C(t)}^*)^2 \beta(t) = \beta(t)\cdot \beta(t) = 0$ for both (N) and (D)
 boundary conditions. Moreover $\beta(\cdot)$ is smooth on $(0,T)\times M$
 by Theorem \ref{thmEUP}.
  The following computation is therefore justified for $t >0$.
  \begin{align*}
  (1/2) (d/dt) \| \beta(t)\|_2^2 &= ( \beta'(t), \beta(t))\\
  &= (d_{C(t)} C'(t), \beta(t)) \\
  &= (C'(t),  d_{C(t)}^* \beta(t))\\
  &= - ( d_{C(t)}^* \beta(t) + d_{C(t)} d^*C(t), d_{C(t)}^* \beta(t))\\
  &= - \| d_{C(t)}^* \beta(t)\|_2^2 - (d^* C(t), (d_{C(t)}^*)^2 \beta(t))\\
 & = - \| d_{C(t)}^* \beta(t)\|_2^2
  \end{align*}
  Since $C(\cdot)$ is continuous on $[0, T)$ into $W_1 \cap L^4(M)$,
  $B_{C(t)}$ is continuous into $L^2(M)$ on $[0,T)$.
  We may therefore integrate the last equality over $[0, t]$ to
  deduce \eref{ST202} and \eref{ST203}.
 \end{proof}

%%%%%%%% Begin horizontal E and U
\section{Short time existence and uniqueness for the Yang-Mills heat equation} \label{secST3}

 In this section we will prove the short time existence portions
  of Theorems \ref{thm1N}  and \ref{thm1D} along with uniqueness.
The space $H_1$ refers to either of the quadratic form domains
defined in  Remark \ref{remcoerc}
 and used in Section \ref{secST},
with  the $H_1$ norm given by \eref{ST19}. $d_A$ represents
 the minimal or maximal
operator, in agreement with the boundary conditions.

% Theorem: Main short time E and U theorem for the horizontal  equation.

       \begin{theorem} \label{thmSTE}
Let $A_0 \in H_1(M)$ and suppose that $\beta \ge \|A_0\|_{H_1(M)}$.
Then there exists $T >0$, depending only on $\beta$,  and a
continuous function
\beq
A(\cdot) : [0,T) \rightarrow H_1(M)\ \ \text{with}\ \ A(0) = A_0
\eeq
such that

a$)$  $B(t) \in H_1(M)$ for each $t \in (0, T)$,

b$)$  $A(t)$ is a strongly differentiable function into $L^2(M)$ on $(0,T)$,

c$)$  $ A'(t) = - d_{A(t)}^* B(t)$.

\noindent
Moreover, $A(\cdot)$  satisfies the regularity condition

f$)$ \     $t^{3/4} \| B(t)\|_\infty$   is bounded on $(0,T)$.
 \end{theorem}

The previous theorem will be deduced from the following, which
 makes precise the informal procedure described
 in Lemma \ref{lemDS}.

           \begin{theorem}\label{thmSTE3}
   Suppose that $C(\cdot)$ is a solution to \eref{ST11} satisfying
  conditions a$)$,  b$)$, c$)$ and f$)$ of Theorem \ref{thmpara} with $T <\infty$.
 Let $0 < \epsilon <T$ and, for each $x \in M$, denote by $g_\epsilon(t, x)$ the solution to the ordinary differential equation
\beq
 (d/dt) g_\epsilon(t, x) = (d^* C(t,x)) g_\epsilon(t,x),\ \ \epsilon \le t <T,\ \
      g_\epsilon(\epsilon) = I_{\mathcal V}.                                      \label{sth5}
 \eeq
 Then $g_\epsilon \in C^\infty ( [\epsilon, T) \times M; K)$.
 Define
 \beq
 A_\epsilon(t) = C(t)^{g_\epsilon(t)} =  g_\epsilon(t)^{-1} C(t) g_\epsilon(t)
  +  g_\epsilon(t)^{-1} d g_\epsilon(t), \ \ \epsilon \le t <T .               \label{sth1}
 \eeq
 Then
 $A_\epsilon \in C^\infty ( [\epsilon , T) \times M ; \Lambda^1 \otimes \frak k)
 \cap  H_1(M)$
  for $\epsilon \le t <T$.
 There exists a continuous function
 \beq
 A(\cdot) : [0,T) \rightarrow H_1(M; \Lambda\otimes \frak k)
 \eeq
  such that  the curvature $B(t)$ of $A(t)$ is in $H_1$ for $t >0$ and
  the strong $L^2$ derivative $A'(t)$ exists for all $t >0$.
  Furthermore
 \begin{align}
  &\sup_{\epsilon \le t <T} \| A(t) - A_\epsilon(t) \|_{H_1} \rightarrow 0
           \ \  \ \    \ \ \text{as}\ \epsilon \downarrow 0,                         \label{sth2} \\
   &\sup_{\epsilon \le t <T}t^{1/2} \|A'(t) - A_\epsilon'(t)\|_{L^2} \rightarrow 0
   \ \ \text{as}\ \epsilon \downarrow 0,                           \label{sth2a} \\
&\sup_{\epsilon \le t <T} t^{1/2} \| B(t) - B_\epsilon(t)\|_{H_1} \rightarrow 0
                        \   \ \text{as}\ \epsilon \downarrow 0     \label{sth3}      \\
\text{and} \ \ \ \ \ \ \ \  &\sup_{\epsilon \le t <T} t^{3/4}  \| B(t) - B_\epsilon(t)\|_\infty \rightarrow 0 \ \  \ \text{as}\ \epsilon \downarrow 0.           \ \ \ \ \ \   \label{sth4}
\end{align}
\iffalse
 \beq
 \sup_{\epsilon \le t <T} \| A(t) - A_\epsilon(t) \|_{H_1} \rightarrow 0 \ \ \text{as}\ \epsilon \downarrow 0.         \label{sth2}
 \eeq
 \fi
 Moreover $A(\cdot)$ satisfies all the conditions of Theorem \ref{thmSTE}.
 \end{theorem}

   \begin{notation} {\rm If $u(x) \in End\ \mathcal V$ for each $x \in M$ we will
write $\|u\|_\infty = \sup_{x \in M} \|u(x)\|_{op}$, where the subscript
op denotes the operator norm on  the finite dimensional inner product
 space $\mathcal V$.
 In case $u$ is a function into $\kf \subset End\ \mathcal V$ the operator
 norm and the $\kf$ norm are equivalent and we will not
  distinguish between them. Compare Notation \ref{not2.1}.
  Although all products of $\frak k$ valued forms
  have been, until now,  commutator products, as e.g. in \eref{ST14},
   we will need to estimate  more   general operators on $\V$ in the following.
 }
 \end{notation}

      \begin{corollary}  \label{corSTE3}
 The functions $g_\epsilon$ converge to a continuous function
  $g:[0, T) \times M \rightarrow K \subset End\ \V$ in the sense that
 \begin{align}
 &\sup_{\epsilon \le t <T}\| g(t) - g_\epsilon(t)\|_\infty
       \rightarrow 0   \ \ \ \ \ \text{as}\ \    \epsilon \downarrow 0       \label{sth9}\\
 \text{and}\ \  & \sup_{\epsilon \le t <T}
 \| h(t) - g_\epsilon(t)^{-1} d g_\epsilon(t) \|_{W_1(M)} \rightarrow 0
            \ \ \  \text{as}\ \   \epsilon \downarrow 0,                      \label{sth10}
 \end{align}
 for some continuous function
 $h : [0, T) \rightarrow W_1(M ; \Lambda^1 \otimes \frak k)$.
  Here $g(0) = I_\V$  and $h(0) = 0$.   $A$ is given by
 \beq
 A(t) = g(t)^{-1} C(t) g(t) + h(t)   .                         \label{sth12}
 \eeq
 \end{corollary}

         The proofs of theses two theorems and corollary will be given at the
  end of this section.

 \begin{remark} {\rm Since $g_\epsilon(t)$ is given fairly explicitly by
 \eref{sth5} in terms of the solution $C(\cdot)$ to the parabolic equation \eref{ST11},  it would seem natural to prove \eref{sth9} and \eref{sth10} first,
 from which \eref{sth2} would follow easily. But we have not been able to
 find a direct proof of the estimates on $g_\epsilon(t)^{-1} d g_\epsilon(t)$
  needed for proving \eref{sth10}.  Instead we will prove \eref{sth9}
  and \eref{sth2} first, using apriori estimates from Section \ref{secfe}.
  }
  \end{remark}

\subsection{g estimates} \label{secST3.1}

The following computations underlie the procedure described in
 Lemma \ref{lemDS}.
       Throughout this subsection $d$ and $d^*$ act on all smooth forms on $M$.
       Boundary conditions on forms will be described explicitly when appropriate.

            \begin{lemma}\label{lem3a}
Let $C \in C^\infty ((a,b) \times M; \L^1\otimes \frak k)$.
 For each $x \in  M$, let $g(t,x)$ be a solution to the ordinary differential equation
\beq
g'(t, x) g(t,x)^{-1} = d^* C(t,x),\ \  t \in (a,b).                             \label{st2}
\eeq
Define
\beq
  C^g(t,x) =  g(t,x)^{-1} C(t,x) g(t,x) + g(t,x)^{-1} dg(t, x).      \label{a21}
\eeq
Then
\begin{align}
(g^{-1} dg)'  &= g^{-1} (dd^*C) g \qquad \qquad  \text{and}   \label{st3}\\
(C^g)' &= g^{-1} (C' + d_Cd^*C) g.                           \label{a20}
\end{align}
Let  $A(t,x) =  C^g(t,x)$ and assume that $C$ satisfies \eref{ST11} over the interval $(a,b)$. Then
\beq
A'(t) + d_{A(t)}^*B_{A(t)} = 0 \ \  \text{on}\ \ (a,b),                    \label{a30}
\eeq
and further,
\beq
A'(t) = - g(t)^{-1} \{ d_{C(t)}^* B_{C(t)}\} g(t).                            \label{a31}
\eeq
\end{lemma}
          \begin{proof}
    The easily verifiable identity $ (g^{-1}dg)' = g^{-1}\{ d(g' g^{-1}) \}g$
    proves \eref{st3}, given \eref{st2}.
  Writing $V = g'g^{-1}$ we can compute
\begin{align*}
 (C^g)'  &= g^{-1}\{C' + [C, g'g^{-1}] \}g  +(g^{-1}d g)'  \notag \\
& = g^{-1}\{C' + [C,V] +dV\}g   \notag
\end{align*}
which is \eref{a20} when \eref{st2} holds.
    In particular, if $C' + d_C d^*C = - d_C^* B_C$ over $(a,b)$ then
    \eref{a20} shows that
    $ A' = g^{-1} (- d_C^* B_C)g = - d_A^* B_A$.
    Here we have used the usual gauge transformation identities,
    $B_A = g^{-1}(B_C) g$ and $ d_A^* B_A = g^{-1}( d_C^* B_C) g$
    when $A =C^g$.
    \end{proof}

      \begin{lemma}\label{lemg4} {\rm (Boundary conditions for $g$.)}
  Suppose that
$C(\cdot) \in C^\infty((0,T))$ and satisfies the differential equation
$(d/dt) C = - (d_C^* B_C + d_C d^*C)$ on $(0,T)$
along with one of the two boundary conditions
 \eref{ST11N} or \eref{ST11D}. Define $g_\epsilon$ by \eref{sth5} and let
 \beq
 h_\epsilon(t) = g_\epsilon(t)^{-1} d g_\epsilon(t),\ \ \ \epsilon \le t < T. \label{g21}
 \eeq
 If $C$ satisfies the Neumann boundary contition \eref{ST11N}, then
 \beq
 (N)\ \ \ \ \ \ \ \ h_\epsilon(t)_{norm} =0 , \ \ \ \ \ \epsilon \le t <T,             \label{g22}
 \eeq
 and if $C$ satisfies the Dirichlet boundary  condition \eref{ST11D}, then
 \beq
 (D)\ \ \ \ \ \ \ \ h_\epsilon(t)_{tan} =0, \ \ \ \ \ \ \         \epsilon \le t <T .     \label{g23}
 \eeq
 In particular, $h_\epsilon(t) \in H_1$ in both cases.
 \end{lemma}

                \begin{proof}
     Since $g_\epsilon(\epsilon) = I_{\V}$ it follows that $h_\epsilon(\epsilon)=0$.
 It suffices, therefore, to show that the normal, respectively tangential, component
 of $h_\epsilon'(t)$ is zero on $[\epsilon, T)$. The identity  \eref{st3}
 shows that $h_\epsilon'(t) = g_\epsilon(t)^{-1} \{dd^* C(t)\} g_\epsilon(t)$
 and therefore it suffices to show that the normal, respectively tangential,
 component of $dd^*C(t)$ is zero for $\epsilon \le t < T$.

        In case (N) we have, by \eref{ST11N},  $C(t)_{norm} = 0$  and $(B_{C(t)})_{norm}=0$. From the first equality it follows that $C'(t)_{norm} =0$
        and from the second equality it follows, with the help of \eref{DN61},
        that $(d_{C(t)}^* B_{C(t)})_{norm} =0$. Therefore \eref{ST11} shows that
        $(d_{C(t)} d^*C(t))_{norm} =0$.
      Hence $(dd^*C(t))_{norm} = - [C(t), d^*C(t)]_{norm}
      = - [ C(t)_{norm}, d^*C(t)] = 0$. This proves case (N).

 In case (D), we have
 $(d^*C(t))_{tan} \equiv (d^* C(t))|_{\p M} =0$  by \eref{ST11D}.
 Therefore $(dd^* C(t))_{tan} =0$ by \eref{DN60} (with $A =0$).
 This proves case (D).

       Since $h_\epsilon(t) \in C^\infty(M)$ and satisfies the right boundary conditions it is in $H_1$ in both cases.

     Although this proves the lemma, it may be worth noting that in case (D)
   the defining   equation \eref{sth5} shows directly that
  $g_\epsilon'(t)|_{\p M} = 0$ because   $d^* C(t)|_{\p M} =0$.
  Hence $g_\epsilon(t) =I_{\V}$ on $\p M$ and therefore its tangential derivative, $h_\epsilon(t)_{\tan}$ is zero.
 \end{proof}

      \begin{corollary}\label{corg5}
      {\rm (Boundary conditions for $A_\epsilon $.)}
       Define $A_\epsilon(t)$ by \eref{sth1}.
 Then $ A_\epsilon(t) \in H_1(M)$  in both Neumann and Dirichlet cases,
 for $\epsilon \le t <T$.
 \end{corollary}

                   \begin{proof}
    Since $A_\epsilon(t)$ is in $C^\infty(M)$ and $C(t)$ satisfies the
 right boundary conditions, the definition \eref{sth1} shows that we need only prove that $g_\epsilon(t)^{-1} d g_\epsilon(t)$ satisfies the correct boundary conditions. But this is the assertion of Lemma \ref{lemg4}.
    \end{proof}

  % Lemma lem20''
          \begin{lemma}\label{lem20''}
 Define $g_\epsilon: [\epsilon, T) \rightarrow K$ as in \eref{sth5}. Then
 \beq
 \sup_{\epsilon \le t <T} \| g_\delta(t) - g_\epsilon(t)\|_\infty \rightarrow 0
  \ \text{as}       \ 0 <\delta <\epsilon \downarrow 0.   \label{sth17''}
 \eeq
 Moreover there is a unique function
 $g \in C([0, T) \times M; K)$ such that $g(0) = I_\V$ and such that, for each
 $a \in (0, T)$, $g_\epsilon$ converges to $g$ uniformly on $[a,T)\times M$.
\end{lemma}

               \begin{proof}
 For ease in reading let $ V(t,x) = d^* C(t,x)$. All the estimates that need to be made are pointwise in $x$. For each $x \in M$, $V(t,x)$ is a continuous  function
 on $(0, T)$ into $\frak k$ and $\int_0^T \|V(s)\|_\infty ds < \infty$ by \eref{ST23}.
 We will suppress the  $x$ dependence in the following.
 If $0 < \delta <\epsilon$ then the function $[\epsilon, T)\ni  t \mapsto g_\delta(t) g_\delta(\epsilon)^{-1}$ satisfies the initial value problem \eref{sth5},
 and consequently,
 \beq
 g_\delta(t) = g_\epsilon(t) g_\delta(\epsilon), \ \ \epsilon \le t <T. \label{sth18''}
 \eeq
 Since $g_\epsilon(t)$ is unitary it follows that
 \beq
 \|g_\delta(t) - g_\epsilon(t)\|_{op}
 = \| g_\delta(\epsilon) - I_\V\|_{op},\            \epsilon \le t <T.  \label{sth19''}
 \eeq
 But
 \beq
 \| g_\delta(\epsilon) - I_\V \|_{op}  = \|\int_\delta^\epsilon g_\delta'(s) ds \|_{op}
 \le \int_\delta^\epsilon \| V(s)\|_\infty ds \rightarrow 0 \ \text{as}\ \epsilon\downarrow 0.   \label{sth20''}
 \eeq
 This proves \eref{sth17''}.
 The existence of a uniform limit $g$ over each set $[a,T) \times M$ now follows
 and the limit is clearly independent of $a$. Moreover letting $\delta\downarrow 0$ in \eref{sth20''} shows that
 $\| g(\epsilon)- I_\V\|_\infty \le \int_0^\epsilon \|V(s)\|_\infty ds$,
  and therefore  $g$ is continuous on all of $[0, T) \times M$ if defined to
 be $I_\V$ at $t =0$.
 \end{proof}

\subsection{$A$ estimates}   \label{secSTE3.2}

Our goal in this section is to show that the smooth forms $A_\epsilon(t)$ and
$B_\epsilon(t)$ converge in strong senses as $\epsilon\downarrow 0$.
Since $A_\epsilon(\cdot)$ is in $C^\infty((\epsilon, T)) \times M)$,
 all of the apriori estimates derived in Sections \ref{secSobsol}, \ref{secfa} and  \ref{secfe}      are
 applicable in this subsection.

 With a view toward applying the Gaffney-Friedrichs inequality   \eref{gaf50}
 (with $A=0$ in that inequality),
 we are going to make estimates in the next few lemmas of $\|d\w\|_2$
 and $\|d^*\w\|_2$ for  several different choices of $\w$.

         All four lemmas in this section depend on the apriori estimates of order
         one in Section \ref{secfe}.

   \begin{lemma}\label{lemA55'}
\begin{align}
\int_\epsilon^t \| A_\epsilon'(s)\|_4 ds
        &\le C_4(t, \| B_0\|_2),\  \epsilon \le t <T  ,               \label{sth100a'}\\
\|A_\epsilon(t)\|_4
&\le    \|C(\epsilon)\|_4 + C_4(t, \|B_0\|_2), \ \epsilon \le t <T ,   \label{sth105a'}\\
 \text{and}\qquad   \qquad      \|B_\epsilon(t)\|_2 &\le \| B_0\|_2, \       \epsilon \le t <T.  \label{sth27a'}
\end{align}
\end{lemma}

                           \begin{proof}
Since $A_\epsilon$ is a solution to \eref{ymh10}
     over the interval  $(\epsilon, T)$, we may apply  \eref{M55'} over the
      interval $[\epsilon,T)$ to find
$ \int_\epsilon^t \| A_\epsilon'(s)\|_4 ds
        \le C_4(t-\epsilon, \| B_{A_\epsilon(\epsilon)}\|_2)$
        for $ \epsilon \le t <T $. Since $C_4$ is monotone in both arguments
        and    $\| B_{A_\epsilon(\epsilon)}\|_2 = \|B_{C(\epsilon)}\|_2 \le \|B_0\|_2$,
        \eref{sth100a'} follows. The derivation of \eref{sth105a'} from
    \eref{sth100a'} is similar to the derivation of  \eref{M78'}, considering that
    $A_\epsilon(\epsilon) = C(\epsilon)$.
            Further,
  $\|B_\epsilon(t) \|_2 = \| g_\epsilon(t)^{-1} B_{C(t)} g_\epsilon(t)\|_2
 = \| B_{C(t)}\|_2 \le \|B_0\|_2$ by \eref{ST203}, proving \eref{sth27a'}.
\end{proof}

\begin{lemma}\label{lemA56}
As $0 < \delta \le \epsilon \downarrow 0$ the following limits hold.
\begin{align}
 \int_\epsilon^T \| A_\delta'(s) - A_\epsilon'(s)\|_4 ds &\rightarrow0.\
                     \label{sth101a} \\
 \sup_{\epsilon \le t < T} \|A_\delta(t) - A_\epsilon (t) \|_4 &\rightarrow 0.
                    \label{sth106a}      \\
\sup_{\epsilon \le t < T} \|A_\delta(t) - A_\epsilon (t) \|_2 &\rightarrow 0.
                     \label{sth107a} \\
 \sup_{\epsilon \le t < T}\| B_\delta(t) - B_\epsilon(t)\|_2 &\rightarrow 0.
                    \label{sth28a}
 \end{align}
 \end{lemma}

                   \begin{proof}
         To prove \eref{sth101a} observe that gauge transformations relate well to
  the forms $A'$ in that
  \beq
  A_\delta'(s) = (Ad\ g_\delta(\epsilon)^{-1}) A_\epsilon'(s) ,
               \ \ \ s \ge \epsilon                               \label{sth102a}
  \eeq
  because $-A_\delta'(s) = d_{A_\delta(s)}^*B_{A_\delta(s)}
  = (Ad\ g_\delta(\epsilon)^{-1} )d_{A_\epsilon(s)}^*B_{A_\epsilon(s)} $.
  Therefore
    \begin{align*}\int_\epsilon^t \| A_\delta'(s) - A_\epsilon'(s)\|_p ds
      &=\int_\epsilon^t \| (Ad\ g_\delta(\epsilon)^{-1} - I_{\frak k})
   A_\epsilon'(s)\|_pds\\
 &\le \| g_\delta(\epsilon) - I_\V\|_\infty    C_8(t, \| B_0\|_2),
  \end{align*}
   from which \eref{sth101a} follows.

           To prove \eref{sth106a} we may again use the identity
   $A_\epsilon(t) = C(\epsilon) +\int_\epsilon^t A_\epsilon'(s) ds$  to find
 \begin{align*}
 \| A_\delta(t) - A_\epsilon(t) \|_4 =\|C(\delta) - C(\epsilon)
   + \int_\delta^t A_\delta'(s) ds - \int_\epsilon^t A_\epsilon'(s) ds \|_4\\
 \le \| C(\delta) - C(\epsilon) \|_4
  +\int_\delta^\epsilon \| A_\delta'(s) \|_4 ds
  + \int_\epsilon^t \| A_\delta'(s) - A_\epsilon'(s)\|_4 ds.
 \end{align*}
 The first term goes to zero as $\delta < \epsilon\downarrow 0$ because
 $C(\cdot)$ is continuous into $H_1$ and therefore into $L^4(M)$.
  The third term goes to zero uniformly
 for $t \in [\epsilon, T)$ by \eref{sth101a}. The middle term is equal to
 $\int_\delta^\epsilon \| d_{C(s)}^* B_{C(s)}\|_4ds$ by \eref{a31} and
  goes to zero  because the integrand is integrable
   over $[0, T)$ by \eref{sth100a'}.
    Replace $L^4$   by $L^2$ in this proof to arrive at  \eref{sth107a}.

    Now
  $$ \| B_\delta(t) - B_\epsilon(t) \|_2
  =\|( Ad\ g_\delta(t)^{-1} -  Ad g_\epsilon(t)^{-1})B_{C(t)}\|_2
 \le \|Ad\ g_\delta(\epsilon) - I\|_\infty\|B_0\|_2$$ by \eref{sth19''}.
  Thus \eref{sth28a} now follows from \eref{sth20''}.
   \end{proof}

\begin{lemma} \label{lemA57}
As $0<\delta \le \epsilon \downarrow 0$ the following limits hold.
\begin{align}
\sup_{\epsilon \le t <T} t^{1/2} \|A_\delta'(t) - A_\epsilon'(t) \|_2
                  \rightarrow 0.                                                 \label{sth111a}\\
\sup_{\epsilon \le t <T}t^{3/8} \| B_\delta(t) - B_\epsilon(t) \|_4
                  \rightarrow 0 .                                               \label{sth36ba}\\
\sup_{\epsilon \le t <T} \|d^*( A_\delta(t) - A_\epsilon(t))\|_2
                 \rightarrow 0   .                                               \label{sth62a}\\
\sup_{\epsilon \le t <T} \|d( A_\delta(t) - A_\epsilon(t))\|_2
                  \rightarrow 0   .                                               \label{sth72a}
\end{align}
\end{lemma}

                \begin{proof}
Since $A_\delta(\delta) = C(\delta)$ we may apply the apriori estimate \eref{fe6} to $A_\delta(t)$ on the interval $[\delta, T)$ to find
      $(t -\delta) \| A_\delta'(t)\|_2^2 \le C_1( t -\delta, \| B_{C(\delta)}\|_2)$.
      By \eref{sth102a} $\| A_\delta'(t)\|_2 = \| A_\epsilon'(t)\|_2$  for
       $0 < \delta \le \epsilon\le t $
       while $\|B_{C(\delta)}\|_2 \le \|B_0\|_2$ by \eref{ST203}.
      Since $C_1(\cdot, \cdot)$ is nondecreasing in both arguments
       we find $(t -\delta) \| A_\epsilon'(t)\|_2^2 \le C_1(t, \|B_0\|_2)$.
       We may now let $\delta \downarrow 0$ to find
       \beq
       t^{1/2} \| A_\epsilon'(t) \|_2 \le C_1(t, \|B_0\|_2)^{1/2},
          \epsilon \le t <T                                                     \label{sth110a}
       \eeq
            The assertion \eref{sth111a} now follows from the inequality
       $t^{1/2}\| A_\delta'(t) - A_\epsilon'(t)\|_2
   \le \| Ad\ g_\delta(\epsilon)^{-1} - I_{\frak k} \|_\infty t^{1/2} \| A_\epsilon'(t)\|_2
 \le \| Ad\ g_\delta(\epsilon)^{-1} - I_{\frak k} \|_\infty  C_1(T, \|B_0\|_2)^{1/2}$.

         To prove \eref{sth36ba} observe that
    $(t-\delta)\|B_\delta(t)\|_6^2 \le C_3(t-\delta, \| B_{C(\delta)}\|_2)
    \le C_3(t, \|B_0\|_2)$ by \eref{fe12} applied over the interval $[\delta, T)$.
    Since
    $\|B_\delta(t)\|_6
    =\|Ad\ g_\delta(\epsilon)^{-1})B_\epsilon(t)\|_6
    =\|B_\epsilon(t)\|_6$ we can let $\delta\downarrow 0$ to find
    $t\|B_\epsilon(t)\|_6^2 \le$                    \linebreak
    $  C_3(t, \|B_0\|_2)$.
     Interpolation between $L^2$ and $L^4$ now gives, in view of \eref{sth27a'},
     \beq
      t^{3/8}\|B_\epsilon(t)\|_4 \le \|B_\epsilon(t)\|_2^{1/4}
      (t^{3/8})\|B_\epsilon(t)\|_6^{3/4}
    \le \|B_0\|_2^{1/4} C_3(t, \|B_0\|_2)^{3/8} .                        \label{sth112a}
    \eeq
    Hence
    \beq
    t^{3/8}\| B_\delta(t) - B_\epsilon(t)\|_4 \le \| Ad\ g_\delta(\epsilon)^{-1} - I_{\frak k} \|_\infty \|B_0\|_2^{1/4} C_3(t, \|B_0\|_2)^{3/8},
    \eeq
 which proves \eref{sth36ba}.

        To prove \eref{sth62a} observe that,
 since $A_\epsilon( s)$ is a $C^\infty$ solution to the Yang-Mills heat equation \eref{B8}, the argument giving the identity \eref{M75'} gives
 \beq
    d^* A_\epsilon(t)      = d^*C(\epsilon)
    + \int_\epsilon^t [ A_\epsilon(s) \cdot A_\epsilon'(s)] ds.      \label{sth65a}
    \eeq
    because $A_\epsilon(\epsilon) = C(\epsilon)$.
                 Using \eref{sth65a} for both $\epsilon$     and $\delta$ we find
        \begin{align*}
        \| d^*\{ A_\delta(t) - A_\epsilon(t) \} \|_2
        & \le \| d^* \{ C(\delta) - C(\epsilon)\} \|_2
        + \int_\delta^\epsilon \| [ A_\delta(s) \cdot A_\delta'(s) ] \|_2 \\
        &+ \int_\epsilon^t \| [A_\delta(s)\cdot A_\delta'(s) ]
            - [A_\epsilon(s) \cdot A_\epsilon'(s) ] \|_2 ds.
        \end{align*}
        The first term on the right goes to zero as
         $0 < \delta < \epsilon \downarrow 0$ because $C(\cdot)$
          is continuous into $H_1$.
     The second term goes to zero because
     $\|A_\delta(s)\|_4$ is bounded,   by \eref{sth105a'} while
   $\|A_\delta'(s)\|_4 = \| d_{C(s)}^* B_{C(s)}\|_4$,
   which is integrable over $(0,T)$ by \eref{sth100a'}.
    The third term goes to zero  as $0 < \delta < \epsilon \downarrow 0$ in view
    of \eref{sth105a'}, \eref{sth106a}, \eref{sth100a'} and \eref{sth101a},
     which show that  $\|A_\delta(s)\|_4$ is bounded, that
    $\| A_\delta(s) - A_\epsilon(s) \|_4 $ goes to zero uniformly in $s$ over $[\epsilon, T)$, while $\|A_\epsilon'(s)\|_4$ is bounded in $L^1(\epsilon,T)$ and
    $\| A_\delta'(s) - A_\epsilon'(s) \|_4$ goes to zero in $L^1(\epsilon, T)$.

 To prove \eref{sth72a} we use again the identity
 $dA_\epsilon =B_\epsilon - (1/2) [ A_\epsilon \wedge A_\epsilon]$
 to arrive at
  \beq
 \| d\{ A_\delta(t) - A_\epsilon(t) \} \|_2 \le \| B_\delta(t) - B_\epsilon(t) \|_2
    +(1/2) \| [ A_\delta(t) \wedge A_\delta(t)]
    - [A_\epsilon(t) \wedge A_\epsilon(t)] \|_2.               \notag
    \eeq
    The first term on the right goes to zero uniformly for $ t \in [\epsilon,T]$ by \eref{sth28a} while the second term goes similarly
    to zero in virtue of \eref{sth105a'} and \eref{sth106a}.
  \end{proof}

\begin{lemma} \label{lemB58'}
There is a non-decreasing continuous function
$C_6: [0,\infty)^2 \rightarrow [0,\infty)$, depending only on the
geometry of $M$, such that
\beq
t^{1/2} \big( \| d B_\epsilon(t)\|_2 + \| d^*B_\epsilon(t)\|_2\big)
            \le C_6(t, \|B_0\|_2, \| C(\epsilon)\|_{W_1}) \ \ \epsilon \le t < T.    \label{sth38a'}
\eeq
Moreover, as $0 < \delta \le \epsilon \downarrow 0$ the following limits hold.
   \begin{align}
 \sup_{\epsilon \le t <T} t^{1/2} \| d(B_\delta(t) - B_\epsilon(t)) \|_2 \rightarrow 0.
          \label{sth39a}\\
 \sup_{\epsilon \le t <T} t^{1/2} \| d^*(B_\delta(t) - B_\epsilon(t)) \|_2 \rightarrow 0.
    \label{sth39a'}
 \end{align}
\end{lemma}

                        \begin{proof} The Bianchi identity and \eref{a30} yield,
   respectively,
 \begin{align}
 dB_\epsilon(t)&= -[ A_\epsilon(t) \wedge B_\epsilon(t)],        \label{sth40} \\
  d^* B_\epsilon(t)& = -A_\epsilon'(t) - [ A_\epsilon(t) \lrc B_\epsilon(t)].
                                                                                                \label{sth41}
 \end{align}
  Therefore,
 \begin{align}
 &t^{1/2} \{ \|d B_\epsilon(t)\|_2 + \| d^* B_\epsilon(t)\|_2 \}  \notag\\
 &\le t^{1/2}\{\| [ A_\epsilon(t) \wedge B_\epsilon(t)]\|_2 + \|A_\epsilon'(t)\|_2
 + \| [A_\epsilon(t) \lrc B_\epsilon(t)] \|_2  \}                    \notag\\
 &\le t^{1/2} \{\|A_\epsilon'(t)\|_2
           +2c \|A_\epsilon(t)\|_4 \| B_\epsilon(t) \|_4\}               \notag \\
 &\le C_1( t, \|B_0\|_2)^{1/2}
+ 2 c \big(\|A_\epsilon(t)\|_4\big)
      t^{1/2} \| B_\epsilon(t) \|_4    \notag\\
      &  \le C_6(t, \|B_0\|_2, \|C(\epsilon)\|_{H_1})    \label{sth41a}
 \end{align}
    for some continuous function $C_6$, by virtue of
     \eref{sth110a}, \eref{sth105a'} and\eref{sth112a}.

          Using the identity \eref{sth41} for $\epsilon$ and $\delta$ we may write
    \begin{align*}
  t^{1/2} \| d^*( B_\delta(t) - B_\epsilon(t)) \|_2
  &\le t^{1/2} \| A_\delta'(t) - A_\epsilon'(t)\|_2 \\
  &+ t^{1/2} \| [A_\delta(t) \lrc B_\delta(t)] - [A_\epsilon(t) \lrc B_\epsilon(t)]\|_2.
  \end{align*}
  The first term goes to zero uniformly for $\epsilon \le t <T$ by \eref{sth111a}.
  The second term goes similarly to zero by combining
   \eref{sth105a'}, \eref{sth106a} with \eref{sth36ba} and \eref{sth112a}.
  A similar argument applies to $t^{1/2}\| d(B_\delta(t) - B_\epsilon(t))\|_2$
  by using \eref{sth40}.
  \end{proof}

\begin{theorem} \label{thm9.A}
There exist non-decreasing continuous functions $C_7, C_8, C_9$
from $[0,\infty)^2 \rightarrow [0,\infty)$ such that
\begin{align}
\|A_\epsilon(t) \|_{H_1} &\le C_7(t, \| C(\epsilon)\|_{H_1}),
                       \   \  \epsilon \le t < T,                                  \label{sth81a'}\\
t^{1/2} \| B_\epsilon(t)\|_{H_1}
&\le C_8 (t, \| C(\epsilon)\|_{H_1}),\ \epsilon \le t < T,            \label{sth93a'}\\
t^{3/4} \|B_\epsilon(t)\|_\infty & \le C_9(T, \|C\|_{\P_T} ) ,
                                \ \      \ \epsilon \le t < T.                   \label{sth35'}
\end{align}
Moreover, as $0 <\delta \le \epsilon \downarrow 0$ the following limits hold.
\begin{align}
\sup_{\epsilon \le t <T} \| A_\delta(t) - A_\epsilon(t) \|_{H_1}
        &\rightarrow 0.                                                        \label{sth82a'} \\
\sup_{\epsilon \le t <T}t^{1/2} \| B_\delta(t) - B_\epsilon(t)\|_{H_1}
       &\rightarrow 0.                                                                 \label{sth94a'}\\
\sup_{\epsilon \le t <T} t^{3/4} \|B_\delta(t) - B_\epsilon(t)\|_\infty
      & \rightarrow 0.                                                     \label{sth36'}
\end{align}
\end{theorem}

       \begin{proof}
Apply \eref{M30} to the smooth solution $A_\epsilon$ over $[\epsilon, T)$ and
recall that $A_\epsilon(\epsilon) = C(\epsilon)$. We find that
$\|A_\epsilon(t) \|_{H_1(M)} \le C_5(t-\epsilon, \| C(\epsilon)\|_{H_1(M)}),
                           \  \epsilon \le t < T $.
     The monotonicity of $C_5$ in its first
  argument now yields \eref{sth81a'} with $C_7 = C_5$.

                 To prove \eref{sth82a'} apply the Gaffney-Friedrichs inequality
 \eref{gaf50} with $A=0$ and $\w = A_\delta - A_\epsilon$. The inequality
 \eref{sth82a'} then follows from \eref{sth107a}, \eref{sth62a} and \eref{sth72a}.

The Gaffney-Friedrichs inequality \eref{gaf50}, with $A=0$ and with
   $\w = B_\epsilon(t)$ gives
   $$
   (1/2)t \| B_\epsilon(t)\|_{H_1}^2 \le t\| dB_{\epsilon}(t)\|_2^2
   + t\| d^*B_{\epsilon}(t)\|_2^2 +\lambda_M t\|B_\epsilon(t)\|_2^2.
      $$
   This, along with the inequality \eref{sth38a'} and $\|B_\epsilon(t)\|_2 \le \|B_0\|_2$, proves \eref{sth93a'}.

   Similarly, the Gaffney-Friedrichs inequality \eref{gaf50}, with $A=0$ and
   with $\w = B_\delta - B_\epsilon$, proves \eref{sth94a'} in view of
   \eref{sth39a}, \eref{sth39a'} and \eref{sth28a}.

  Since $t^{3/4} \| B_\epsilon(t)\|_\infty = t^{3/4} \| B_{C(t)}\|_\infty$ the inequality
  \eref{ST150} proves \eref{sth35'}  with
  $C_9(T, \| C\|_{\P_T}) = \|C\|_{\P_T} + T^{1/4}(c/2) \| C\|_{\P_T}^2$.

 Finally, for $\epsilon \le t <T$,  we have
 $ t^{3/4} \| B_\delta(t) - B_\epsilon(t) \|_\infty
   \le  \| Ad\ g_\delta(\epsilon) - I\|_\infty C_9(T, \|C\|_{\P_T})$,
     which goes to  zero uniformly for $t \in [\epsilon, T)$ by
     \eref{sth20''}.
 \end{proof}

\subsection{Proof of Theorem \ref{thmSTE3}}

\indent

      If $ 0< a < T$ then \eref{sth82a'} shows that
 $A_\epsilon|_{[a,T)}$ is uniformly Cauchy in $H_1$ norm as
  $\epsilon \downarrow 0$. The limit is clearly independent of
   $a >0$ and defines a continuous function
 $ A: (0, t) \rightarrow H_1$, being a uniform limit of continuous (in fact $C^\infty$)
 functions on each interval $[a, T)$. Define $A(0) = A_0$.
  We need to show that
  the so extended function is continuous at $t =0$.
     Since $C(\cdot)$ is continuous on $[0, T)$ into $H_1$,
  given $\alpha >0$, there exists $\gamma >0$, such that
     a) $\sup_{0\le t \le \gamma} \|A_0 - C(t) \|_{H_1} < \alpha$ and,
             by \eref{sth82a'},
     b) $ \sup_{\epsilon \le t < T} \| A_\delta(t) - A_\epsilon(t) \|_{H_1} < \alpha$ if
      $0 < \delta \le \epsilon \le \gamma$.
      Suppose that $0 < t_0 \le \gamma$.
Then $\| A_0 - C(t_0)\|_{H_1} <\alpha$ by a).
Letting $\delta \downarrow 0$ in b) shows that $\|A(t_0) - A_\epsilon(t_0)\|_{H_1}
\le \alpha$ if $\epsilon \le t_0$. Take $\epsilon = t_0$. Then
$\| A(t_0) - A_0 \|_{H_1}
  \le \| A(t_0) - A_{t_0}(t_0)\|_{H_1} + \| C(t_0) - A_0\|_{H_1} <2\alpha$.
     This proves the existence of a
      continuous funtion $A: [0,T) \rightarrow H_1(M)$ taking the
   correct initial value, $A_0$, and defined as
  the limit, in the sense of \eref{sth2},
   of the $C^\infty$ functions
  $A_\epsilon : [\epsilon, T) \rightarrow H_1(M)$.

   Now \eref{sth111a} shows
    that, for each $a >0$, the derivatives $A_\epsilon'(t)$ converge uniformly
  on $[a,T)$, as functions into $L^2(M)$. It follows that $A(t)$ is a strongly
  differentiable function on $(0,T)$ into $L^2(M)$ and that $A'(t) = L^2$ limit
  of $A_\epsilon'(t)$  for each $t>0$. In fact, letting $\delta \downarrow 0$ in \eref{sth111a} proves \eref{sth2a}.

  The curvature $B(t)$ of $A(t)$ is well defined because $A(t) \in H_1(M)$.
  Since, for each $t>0$, $A_\epsilon(t)$ converges to $A(t)$ in $H_1$
  by \eref{sth82a'}
  it follows that $B_\epsilon(t)$ converges in $L^2$ to $B(t)$.
   But \eref{sth94a'}
       shows that, for each $t>0$,   $B_\epsilon(t)$
    is Cauchy in $H_1$ norm as $\epsilon \downarrow 0$.
    Hence $B_\epsilon(t)$ converges in $H_1$ norm to an element
      in $H_1$, which is also the $L^2$ limit, $B(t)$.
  Thus $B(t)$ is in $H_1$ for each $t>0$ and
  $\| B_\epsilon(t) - B(t)\|_{H_1} \rightarrow 0$ for each $t>0$. Therefore $d^*B_\epsilon(t)$ converges to $d^*B(t)$
  in $L^2$ while also $B_\epsilon(t)$ converges to $B(t)$ in $L^4$.
  Hence $ [A_\epsilon(t) \lrc B_\epsilon(t)]$ converges to $[A(t)\lrc B(t)]$ in $L^2$
  in view of \eref{sth106a}.
  Therefore $d_{A_\epsilon(t)}^* B_\epsilon(t)$
  converges in $L^2$ to $d_{A(t)}^* B(t)$ for each $t>0$.
  It now follows that $A'(t) = - d_{A(t)}^* B(t)$ for $0 < t <T$.
   Furthermore, taking the limit in \eref{sth94a'} as $\delta \downarrow 0$
   proves \eref{sth3}.

  A similar argument, based on  \eref{sth36'}, shows that,
  for each $t>0$,    one has $\| B_\epsilon(t) - B(t) \|_\infty \rightarrow 0$ as
    $\epsilon \downarrow 0$. In particular, condition f) of Theorem \ref{thmSTE}
   follows from \eref{sth35'}. Moreover, \eref{sth4} follows from \eref{sth36'}
   by letting $\delta \downarrow 0$.
 This proves Theorem \ref{thmSTE3}.

\subsection{Proof of Corollary \ref{corSTE3}}

\indent

For $\epsilon \le t <T$ define
$$
h_\epsilon(t) = g_\epsilon(t)^{-1} d g_\epsilon(t) \ \ \ \ \text{and}\ \ \ \   C_\epsilon(t) = g_\epsilon(t)^{-1} C(t) g_\epsilon(t).
$$
Since
\beq
h_\epsilon(t) = A_\epsilon(t) - C_\epsilon(t),      \label{h19}
\eeq
 we have
$\|h_\epsilon(t)\|_6 \le \|A_\epsilon(t)\|_6 + \| C(t)\|_6$. Hence
\beq
\|h_\epsilon(t)\|_6 \le \kappa(\|A_\epsilon(t)\|_{H_1} + \|C(t)\|_{H_1})\\
\le               \kappa \{C_7( t, \|C(\epsilon)\|_{H_1}) + \|C\|_{\P_T}\}   \notag
\eeq
by \eref{sth81a'} and \eref{ST23}.
Moreover, from \eref{sth82a'} and \eref{sth17''}  we find
\beq
\sup_{\epsilon\le t <T}\|h_\delta(t) - h_\epsilon(t) \|_6
\le \kappa \sup_{\epsilon\le t <T}\| A_\delta(t) - A_\epsilon(t) \|_{H_1}
+ \sup_{\epsilon\le t <T}\| C_\delta(t) - C_\epsilon(t) \|_6  \rightarrow 0,  \label{h20}
\eeq
as $0 <\delta \le \epsilon \downarrow 0$.
  We assert that
 \beq
 \sup_{\epsilon \le t <T} \| C_\delta(t) - C_\epsilon(t) \|_{H_1} \rightarrow 0
 \ \ \text{as} \ \ 0<\delta \le \epsilon \downarrow 0   \label{h21}
 \eeq
 It suffices to compute derivatives for some local orthonormal frame
  field $e_1, e_2, e_3$.
  We have
  $
  \n_j C_\delta(t) =  (Ad\ g_\delta(t)^{-1}) \n_j C(t)
  + [ C_\delta(t), \< h_\delta(t), e_j\>]     \notag
 $
 and therefore, denoting by $\| \cdot \|_2$ an $L^2$ norm
 over a coordinate patch,
  we find
 \begin{align*}
 \| \n_j \big( C_\delta(t) - C_\epsilon(t) \big) \|_2
& \le \| Ad\ g_\delta(t) - Ad\ g_\epsilon(t) \|_\infty \| \n_j C(t) \|_2\\
& + \| [ \{ C_\delta(t) - C_\epsilon(t) \}, h_\delta(t) \< e_j\> ] \|_2 \\
& +\| [C_\epsilon(t), \{ h_\delta(t) - h_\epsilon(t)\} \<e_j\> ]\|_2.
 \end{align*}
 As $0<\delta \le \epsilon \downarrow 0$, the first term goes to zero,
 uniformly for $\epsilon \le t <T$,   by \eref{sth17''},
 since $ \|\n_j C(t)\|_2 \le  \| C\|_{\P_T}$.
  Since $\| h_\delta(t) \< e_j\> \|_6\le \| h_\delta(t) \|_6$
  remains  bounded as $\delta \downarrow 0$
   and uniformly so over $t \in [ \epsilon, T)$, while
    $\|C_\delta(t) - C_\epsilon(t) \|_3
    \le \|Ad\ g_\delta(\epsilon) - I \|_\infty \| C(t)\|_3
    \rightarrow 0$ uniformly over $[\epsilon, T)$ because
     $\|C(t) \|_3$   is dominated by $\| C(t) \|_{H_1} \le \| C\|_{\P_T}$,
     the second term also goes to zero uniformly over $[\epsilon, T)$.
      The third term is dominated by
      $\|C(t) \|_3 \| h_\delta(t) - h_\epsilon(t) \|_6$,
   which goes to zero uniformly over $[\epsilon, T)$  by \eref{h20}.

   Upon adding the contributions to $\| C_\delta(t) - C_\epsilon(t) \|_{H_1}^2$
    from finitely many coordinate patches that cover $M$ the
    assertion \eref{h21} follows.
    From \eref{h21} and \eref{sth82a'}  we deduce that
   \begin{align}
   \sup_{\epsilon \le t <T} \| h_\delta(t) - h_\epsilon(t) \|_{H_1}
   &\le \sup_{\epsilon \le t <T} \| A_\delta(t) - A_\epsilon(t) \|_{H_1}
   + \sup_{\epsilon \le t <T}  \| C_\delta(t) - C_\epsilon(t)\|_{H_1}\notag \\
   &\rightarrow 0\ \ \text{as}\ \ 0 <\delta \le \epsilon \downarrow 0.   \label{sth125}
   \end{align}
   Thus, for each $a \in (0,T)$, the $h_\epsilon$ converge uniformly
    over $[a,T)$ in $H_1$ to a function $h$ which is clearly independent of $a$ and defines a continuous function on $(0, T)$ into $H_1$.

      Now \eref{h19} shows that, for $ 0 < \delta <\epsilon \le t$, we have
  \beq
   \| h_\delta(t) - h_\epsilon(t)\|_{H_1} \le \| A_\delta(t) - A_\epsilon(t) \|_{H_1}
            +\|C_\delta(t) - C_\epsilon(t)\|_{H_1}.  \notag
   \eeq
We have shown that all three differences  converge as $\delta\downarrow 0$
 and we may conclude that
   \beq
   \| h(t) - h_\epsilon(t)\|_{H_1} \le \| A(t) - A_\epsilon(t) \|_{H_1}
            +\|C(t) - C_\epsilon(t)\|_{H_1}.     \notag
   \eeq
   Take $t = \epsilon$.
   Since $g_\epsilon(\epsilon) = I_\V$ on the fiber $\mathcal V$ we have
   $h_\epsilon(\epsilon) =0$ and $C_\epsilon(\epsilon) = C(\epsilon)$
   and therefore $\|h(\epsilon)\|_{H_1} \le \|A(\epsilon) - C(\epsilon)\|_{H_1}$.
    But $A(\epsilon)$ and $C(\epsilon)$ both converge
   to $A_0$ in $H_1$.
     Hence $\|h(\epsilon)\|_{H_1} \rightarrow 0$ as $\epsilon\downarrow 0$.
      Thus $h$ is continuous on $[0,T)$ into $H_1$ if one defines $h(0) =0$.
      The identity \eref{sth12} now follows for each $t >0$
   by taking the $L^2(M)$ limit in \eref{sth1} as $\epsilon \downarrow 0$.
   At $t =0$ the equation \eref{sth12} just  asserts that $A_0 = A_0$
   because $h(0) =0$ and, by Lemma \ref{lem20''}, $g(0) = I_\V$.
    This completes the proof of Corollary \ref{corSTE3}.

% Begin Uninqeness for ymh.

\subsection{Uniqueness of solutions}
\label{secST4}

% Theorem thmunique.
\begin{theorem} \label{thmunique}
Let $T \le \infty$.
Let $A_1(\cdot)$ and $A_2(\cdot)$ be two strong solutions
to \eref{ymh10}  on the interval $[0,T)$ and having the
 same initial data in $W_1(M)$.  Assume that either
  \begin{align}
(N)\ \ \ B_j(t)_{norm} &= 0 \ \text{for}\ j= 1,2 \
                                          \text{and}\ t >0                \label{U3N}\\
\text{or}\ \ \ \ \ \ (D)\ \ \  A_j(t)_{tan}\ \  &= 0 \ \text{for}\ j= 1,2\
                                     \text{and}\ t >0.                      \label{U3D}
\end{align}
Then $A_1(t) = A_2(t)$ on
$[0, T)$
\end{theorem}

The proof depends on  the next lemma.

\begin{lemma}\label{lemU1} $($An identity.$)$
  Suppose that $A_1$ and $A_2$
 are two strong solutions satisfying either \eref{U3N} or \eref{U3D}.
 $($See Definition {\rm \ref{defstrsol}}.$)$
Then, for $t >0$,
\begin{align}
(d/dt) &\| A_1(t)- A_2(t) \|_2^2
= -2\| B_1(t) - B_2(t)\|_2^2       \notag\\
&-  (B_1(t) + B_2(t),[(A_1(t) - A_2(t))\wedge (A_1(t) - A_2(t))])  \label{U5}
\end{align}
\end{lemma}

                   \begin{proof}
   Consider first the Neumann boundary condition \eref{U3N}. In this case the
  heat equation is $ A'(s) = -D_{A(s)}^* B(s)$ wherein $D$ denotes
  the maximal operator defined in Section \ref{secDN} and
  $D_{A(s)}^* B(s) =D^* B(s) + [ A(s) \lrc B(s)]$, which is in $L^2(M)$
  because $B(s)$ and $A(s)$ are both in $W_1$ and $B(s)_{norm} =0$.
  Since $\D(D) \supset W_1$ we may integrate by
  parts in the third line below.
   \begin{align}
(1/2) (d/dt)
 \| A_1(t)-& A_2(t) \|_2^2
=(A_1' -A_2', A_1- A_2)      \notag \\
&= (- D_{A_1}^* B_1 + D_{A_2}^* B_2, A_1 -A_2)      \notag \\
&= -(B_1, D_{A_1} (A_1 - A_2)) + ( B_2, D_{A_2}( A_1 -A_2)).    \label{U6}
\end{align}
But
\begin{align*}
D_{A_1}( A_1 -&A_2) = D(A_1- A_2) + [A_1 \wedge ( A_1-A_2) ]\\
&= B_1 -B_2 -(1/2) [ A_1\wedge A_1 ] + (1/2) [ A_2\wedge A_2 ]
+ [ A_1 \wedge ( A_1 - A_2) ]\\
& = B_1 - B_2 +(1/2) [ (A_1 -A_2) \wedge (A_1 - A_2) ].
\end{align*}
Defining $\alpha = (1/2) [(A_1-A_2)\wedge (A_1 - A_2)]$, we find, similarly,
that $ D_{A_2}(A_1 -A_2) = B_1 - B_2 - \alpha$. Hence
\begin{align*}
(1/2) (d/dt) \| A_1(t)- A_2(t) \|_2^2
&= -( B_1, B_1-B_2 + \alpha) + ( B_2, B_1-B_2 -\alpha)\\
&= - \| B_1 - B_2\|_2^2 - (B_1 + B_2, \alpha),
\end{align*}
which is \eref{U5}.

Next, consider the Dirichlet boundary condition \eref{U3D}. In this case
the heat equation is $A'(s) = -d_{A(s)}^* B(s)$, wherein $d$ is
the minimal covariant exterior derivative operator.
By \eref{U3D} we have $ (A_1(t) - A_2(t))_{tan} =0$. Since this difference is
also in $W_1$ the difference is in the domain of  $d$ by \eref{C19D}.
 We may therefore
integrate by parts, as in \eref{U6}, to find
$(1/2) (d/dt) \| A_1(t)- A_2(t) \|_2^2
 = -(B_1, d_{A_1} (A_1 - A_2)) + ( B_2, d_{A_2}( A_1 -A_2))$.
 The rest of the proof is the same as the Neumann case,
with $D$ replaced by $d$.
\end{proof}

\begin{proof}[Proof of Theorem \ref{thmunique}]
By \eref{U5} we have
\begin{align*}
(d/dt) \| A_1(t) - A_2(t) \|_2^2 &\le |(B_1 + B_2, ( A_1- A_2)\wedge (A_1 - A_2) )|\\
&\le ( \| B_1(t) \|_\infty + \|B_2(t) \|_\infty ) c \| A_1(t) - A_2(t) \|_2^2
\end{align*}
By  condition e) in  Definition \ref{defstrsol}
 we have, for some $b \in (0, T)$ and $a_5 <\infty$,
$$
 t^{3/4} \|B_j(t)\|_\infty \le a_5/2 \
                \text{for}\  0 < t \le b ,\  j =1,2.
 $$
Hence
\beq
(d/dt) \| A_1(t) - A_2(t) \|_2^2 \le a_5 t^{-3/4} \| A_1(t) - A_2(t) \|_2^2.      \label{U9}
\eeq
Since $\int_0^b t^{-3/4} dt < \infty$ and $\| A_1(0) - A_2(0) \|_2 = 0$,
Gronwall's lemma now shows that  $\| A_1(t) - A_2(t)\|_2 =0$
for $ 0 < t \le b$.
(For example \eref{U9} shows that
 $(d/dt) \{ e^{- 4a_5 t^{1/4}} \| A_1(t) - A_2(t) \|_2^2\} \le 0$.)
 Now if $[0,a]$ is a maximal interval of equality  and $a <T$ then, taking the origin
 now at $t =a$, condition d) in  Definition \ref{defstrsol} shows that
 $\| B_j(t)\|_\infty$ are both bounded on any finite interval $[a, b] \subset [a, T)$
 and therefore $ (t-a)^{3/4} \| B_j(t) \|_\infty$ is bounded on  $[a,b]$. The preceding step in the proof now shows that $A_1 = A_2$ on $[a, b]$ and
 therefore $a = T$.
\end{proof}

\begin{remark}\label{remU5}{\rm  It has already been pointed out in Remark \ref{remU1} that it is  the weak parabolicity of the Yang-Mills heat equation
 that is responsible for uniqueness under imposition of only two boundary conditions on the three component form $A(t)$.
       Although we have already   proven
 uniqueness of the parabolic equation \eref{ST11} under standard types of
 Dirichlet or Neumann boundary conditions it is potentially illuminating
 to see whether the previous proof of uniqueness for the weakly parabolic
 equation translates to the parabolic case and why it requires three boundary conditions on the three component form $C(t)$. It is indeed possible to carry out
 the preceding proof for the parabolic case, although it is a little more complicated,
 and does shed light on this comparison question for uniqueness.
 However we will not discuss it further.
}
\end{remark}

% LONG TIME EXISTENCE

\section{Long time existence}    \label{secLTE}

   Here we complete the proof of Theorems \ref{thm1N},
   \ref{thm1D} and \ref{thm1M}.

   We will need the growth estimate  \eref{M30} for strong solutions. But
  the proof of \eref{M30} given in Section \ref{secfe4} relies on
  existence of derivatives, e.g., $B'(t)$, which
  have not been proven to exist for a strong solution.
  We are therefore going to construct approximations of a given strong
   solution by  a sequence of smooth solutions,
  locally in time, using the parabolic equation \eref{ST11} and its
  partial gauge transforms
  $A_\epsilon$, described in Section  \ref{secST3}.

\subsection{Regularization of strong solutions}

                 \begin{lemma} \label{lemlocreg}
                 $($Local regularization.$)$
      Suppose that $A$ is a strong solution
 over $[0, T)$ for some $T \le \infty$. Let $ 0 < t <T$ and define
 $\beta = \sup_{0\le s \le t} \|A(s)\|_{W_1}$. Then there exists $\tau >0$, depending only on $\beta$, such that, for any interval $[a, b] \subset (0, t]$
  of length  $b-a <\tau$,
  there exists a sequence $A_n$ of smooth solutions over
 $[a,b]$ such that
 \begin{align}
 \sup_{ a \le s \le b} \Big\{ \|A_n(s)& - A(s)\|_{W_1}
                                         + \|A_n'(s) - A'(s)\|_{L^2}                   \notag \\
 &+ \| B_n(s) - B(s)\|_{W_1}
    +\| B_n(s) - B(s)\|_\infty \Big\} \rightarrow 0                         \label{LT5}
 \end{align}
 as $n\rightarrow \infty$.
 \end{lemma}

                        \begin{proof}
  The constant $\beta$ is finite because $A:[0, T) \rightarrow W_1$
  is continuous.
  By Theorem \ref{thmpara}  there exists $\tau >0$ such that, for any
  $t_0 \in [0,T)$, a solution $C(\cdot)$ to \eref{ST11}
   with initial value $A(t_0)$,
  exists over $[t_0, t_0 +\tau)$.
  Suppose then that $[a,b] \subset (0, t]$ and that $b < a +\tau$.
  Choose $t_0 \in (0,a)$ with $b <  t_0 +\tau$.
  Then $[a,b] \subset (t_0, t_0 +\tau)$ and the solution $C(\cdot)$  to \eref{ST11}
  over $[t_0, t_0 +\tau)$, with $C(t_0) = A(t_0)$,
   exists over $[t_0, b]$, at least. Define the
  usual gauge transforms $A_\epsilon$ of $C$ over $[t_0 +\epsilon, b]$
  as in \eref{sth1}. By Theorem \ref{thmSTE3} the smooth solutions
  $A_\epsilon$ converge as $\epsilon\downarrow 0$ to a strong solution
  on $[t_0, b]$ with initial data $A(t_0)$. Therefore, by the uniqueness
   theorem  of Section \ref{secST4},  the  solutions $A_\epsilon$
   converge to $A$ itself. The sense of convergence is specified in
   Theorem \ref{thmSTE3}  in \eref{sth2}, \eref{sth2a}, \eref{sth3} and \eref{sth4}.
   In particular, choosing $\epsilon = 1/n$, it follows from these
    that  \eref{LT5}  holds because $a -t_0 >0$.
   \end{proof}

 \begin{corollary} \label{corapstrong}  For any strong solution $A(\cdot)$
 on $[0,T)$

 \noindent
 a$)$ $\|B(\cdot)\|_2$ is non-increasing on $[0,T)$ and
 the Sobolev inequalities \eref{Sob1} and \eref{005} hold for $s \in (0, T)$,

 \noindent
 b$)$ the apriori estimates \eref{fe5}, \eref{fe11}, \eref{fe12}
          and \eref{fe80} hold, and

 \noindent
 c$)$ the inequalities \eref{M81'} -- \eref{M80'} hold.
  \end{corollary}

           \begin{proof}
 Assume that $A(\cdot)$ is a strong solution on $[0,T)$.

 For the proof of a), if $0 <s < T$ pick $t \in (s,T)$ and choose $\tau >0$ as in Lemma \ref{lemlocreg}. Choose an interval $[a,b] \subset (0, t]$ of length
 $b -a < \tau$ with
 $a < s < b$,  and choose a sequence $A_n$ of smooth solutions
  over $[a,b]$ as in the lemma.
 By \eref{Sob1} there holds, for the given $s$,
 \beq
 \| B_n(s)\|_6^2 \le \kappa^2 \Big( \| A_n'(s)\|_2^2
                      + \lambda_n(s) \| B_n(s) \|_2^2 \Big),        \label{LT6}
 \eeq
 where $\lambda_n(s)$ is a polynomial in $\|B_n(s)\|_3$ or $\|B_n(s)\|_2$.
 By Lemma \ref{lemlocreg}, $A_n'(s)$ converges to $A'(s)$ in $L^2(M)$, while
 $B_n(s)$ converges to $B(s)$ in $W_1(M)$ and therefore in $L^2(M)$,
 $L^3(M)$ and $L^6(M)$. Letting $n\rightarrow \infty $ in \eref{LT6} proves
 \eref{Sob1} for strong solutions.
 Now $\|B_n(\cdot)\|_2$ is non-increasing on $[a,b]$ by Theorem \ref{thmfe}
 and therefore $\| B(\cdot)\|_2$ is non-increasing also on this interval.
 Thus $\| B(\sigma)\|_2$
  is non-increasing on any  interval    $[a,b] \subset (0, t]$ of length
   less than $\tau$, and, being continuous at $ \sigma =0$,
    is therefore non-increasing on $[0,t]$ for any $t <T$.
   Now \eref{005} follows from \eref{Sob1} and the monotonicity of $\lambda_2(s)$ as in the original proof of \eref{005}.
   This proves the assertions of Part a).

 For the proof of b) note first that, unlike the Sobolev inequality just proven
 for fixed $s$, all four of the inequalities in Theorem \ref{thmfe} are global,
  in the
 sense that they involve integrals over large intervals. To use Lemma \ref{lemlocreg} it will be necessary to partition the large intervals into small intervals  of length less than $\tau$ and establish inequalities in each interval
 which can be added up with appropriate cancelation of boundary terms.
 We will illustrate the method by deriving the most complicated estimate,
  \eref{fe12}.
  Given a strong solution $A$ over $[0, T)$  and, given $t \in (0,T)$, pick
  $\tau$ as in Lemma \ref{lemlocreg}. Suppose that
  $[a,b] \subset (0, t]$ with $b-a <\tau$. Denote by $A_n$ a sequence of
  smooth solutions as prescribed in Lemma \ref{lemlocreg}.
  We may apply the inequality \eref{fa36} to $A_n$ over the interval $[a,b]$
  by taking the origin to be at $a$. Integrating \eref{fa36} over $[a,b]$ we find
  \beq
  e^{-\psi_n(s)} \|A_n'(s)\|_2^2|_a^b
             + \int_a^b e^{-\psi_n(s)} \|B_n'(s)\|_2^2 ds \le 0.  \label{LT9}
  \eeq
  Here
  $\psi_n(s) = \int_a^s \{ \lambda_M + 2(\kappa c)^2 \|B_n(\sigma)\|_3^2 \} d\sigma$ as in \eref{fa31}.
   Before letting $n \rightarrow \infty$ we need to eliminate
  $\|B_n'(s)\|_2^2$, which we have no control over (at the present time.)
  To this end multiply  \eref{LT9} by $\kappa^2$  and use
  \eref{006} in the integrand   to find
  \beq
  \kappa^2 e^{-\psi_n(s)} \| A_n'(s)\|_2^2 |_a^b
     +\int_a^b e^{-\psi_n(s)} \| A_n'(s)\|_6^2 ds
  \le \kappa^2 \int_a^b e^{-\psi_n(s)} \lambda_n(a) \|A_n'(s) \|_2^2 ds, \label{LT10}
  \eeq
  where   $\lambda_n(a)$ is a fourth degree polynomial in $\|B_n(a)\|_2$
  (see \eref{gfs2}).
      By Lemma \ref{lemlocreg} $A_n'(s)\rightarrow A'(s)$ in $L^2(M)$
      uniformly in $s$ over $[a,b]$. It now follows from Fatou's lemma
      that $\|A'(s)\|_6^2 \le \liminf_{n\rightarrow \infty} \|A_n'(s)\|_6^2$ and
      the same argument applies to the entire integral on the left
      of \eref{LT10}, considering that
      $B_n(\sigma)$, which appears in $\psi_n(s)$,
       converges in $W_1$ to $B(\sigma)$ and
      therefore in $L^3$ also. In fact $\psi_n(s) \rightarrow \psi_a^s$ and
      $\lambda_n(a)$ converges to the corresponding 4th order polynomial in
      $\|B(a)\|_2$. Thus by Part a) and the argument after \eref{gfs2}, we
      may conclude that $\lim_{n\rightarrow \infty}\lambda_n(a) \le \lambda_0$.
        We may now let $n \rightarrow \infty$ to arrive at
     \beq
     \kappa^2 e^{-\psi_a^s} \|A'(s)\|_2^2|_a^b
     +
     \int_a^b e^{-\psi_a^s} \|A'(s)\|_6^2 ds
     \le \kappa^2\lambda_0 \int_a^b e^{-\psi_a^s} \|A'(s)\|_2^2 ds.      \notag
     \eeq

       Now let  $0<\sigma <t$.  If $[a,b] \subset [\sigma, t]$, then,
  using $\psi_\sigma^a + \psi_a^s = \psi_\sigma^s$ for $a \le s$, we
   can multiply the last inequality by $e^{-\psi_\sigma^a}$ to deduce that
   \beq
     \kappa^2 e^{-\psi_\sigma^s} \|A'(s)\|_2^2|_a^b
     +
     \int_a^b e^{-\psi_\sigma^s} \|A'(s)\|_6^2 ds
     \le \kappa^2\lambda_0 \int_a^b e^{-\psi_\sigma^s}
                                 \|A'(s)\|_2^2 ds.                                 \label{LT12}
     \eeq
     Since the exponential factors no longer depend on $a$, \eref{LT12}
     allows for cancellation of the boundary terms thus:  Partition  the interval $[\sigma, t]$
    into small intervals, choosing $\sigma = a_0 < a_1< \cdots < a_n = t$
    with each interval of length less than $\tau$. Taking $a = a_{j-1}$ and
    $b =a_j$ in \eref{LT12} and summing from $j=1$ to $n$ we get
    cancellation of differences on the left and arrive at
    \beq
    \kappa^2\Big\{ e^{-\psi_\sigma^t} \|A'(t)\|_2^2 - \|A'(\sigma)\|_2^2\Big\}
    + \int_\sigma^t e^{-\psi_\sigma^s}  \|A'(s)\|_6^2 ds
    \le \kappa^2\lambda_0 \int_\sigma^t e^{-\psi_\sigma^s}  \|A'(s)\|_2^2 ds, \notag
    \eeq
     which, upon multiplying by $e^{\psi_\sigma^t}$, gives
     \begin{align*}
     \kappa^2 \|A'(t)\|_2^2 + \int_\sigma^t e^{\psi_s^t} \|A'(s)\|_6^2 ds
     \le  \kappa^2\Big\{e^{\psi_\sigma^t} \|A'(\sigma)\|_2^2 +\lambda_0 \int_\sigma^t e^{\psi_s^t}  \|A'(s)\|_2^2 ds   \Big\}.
     \end{align*}
     Combining this with \eref{005},
     which we now know holds for strong solutions  by Part a) of this lemma,
   yields an inequality like \eref{fe21} which gives \eref{fe12} by integration
   with respect to $\sigma$ over $(0,t)$, just as in the original proof of \eref{fe12}.

  For the proof of Part c) observe that, among the inequalities \eref{M81'} - \eref{M80'}, the only one relying on more
  smoothness than is available from the definition of  strong solutions
   is \eref{M80'}, because of its
  dependence on \eref{M73'}, which contains third spatial derivatives of $A$
  on the left side. But the integrated identity \eref{M75'} is clearly derivable from
  Lemma \ref{lemlocreg}
  by adding finitely many identities of the form
  $d^*A(\sigma)|_a^b = \int_a^b [A(s)\cdot A'(s)] ds$
    to arrive at
   $d^*A(t) - d^*A(r) = \int_r^t [A(s)\cdot A'(s)] ds$, and then letting
   $r\downarrow 0$. The rest of the proof is the same as the earlier derivation
   of \eref{M80'}.
  \end{proof}

 \begin{corollary} \label{growthstrong} For any strong solution $A(\cdot)$ over an interval $[0, T)$ the growth estimate \eref{M30} holds.
 \end{corollary}

             \begin{proof}
             The proof of the inequality \eref{M30} depends on the validity of
             the inequalities \eref{M81'} - \eref{M80'} for strong solutions.
             These have been proven for strong solutions in Corollary \ref{corapstrong},
             wherein the restriction that $A \in C^\infty((0,T))$ was removed. The proof given of
             Theorem \ref{thmM6} is now applicable to any strong solution.
             \end{proof}

\subsection{Dirichlet and Neumann boundary conditions}
 \bigskip
 \begin{proof}[ Proof of Theorems \ref{thm1N} and \ref{thm1D}]
     Suppose that  $A(\cdot)$ is a strong solution to \eref{ymh10}
  over $[0,T)$ satisfying either Neumann boundary conditions, \eref{N1}
  and \eref{N2} or Dirichlet boundary conditions, \eref{D1} and \eref{D2}.
  If $ T <\infty$ then by Theorem \ref{thmM6} there is a number $\beta < \infty$
  such that $\|A(t)\|_{H_1(M)} \le \beta$ for $0 \le t <T$.
        By Theorem \ref{thmSTE} there exists $\delta >0$ such that short time
  solutions exist over $[0, 2\delta)$ if $\| A_0\|_{H_1} \le \beta$.
  Apply this theorem     with $A_0 = A(T -\delta)$.
  Then we may conclude that there is a strong solution
   $\hat A(t)$ over $[ T- \delta, T +\delta)$ such that
    $\hat A(T- \delta) = A(T-\delta)$.
         By  uniqueness, Theorem \ref{thmunique},
    we see that $\hat A(t) = A(t)$ on $ [T - \delta, T)$.
     Hence $\hat A$ extends $A$  to the entire interval $[0, T+\delta)$.

        Since, for $a >0$, the condition  \eref{ymh11} shows that
         $\|B_{A(t)}\|_\infty$ is bounded on $[a,T)$,
         and, since
        $\| B_{\hat A(T-\delta + s)}\|_\infty s^{3/4} $ is bounded
         for $0<s < 2\delta$ by f) in Theorem \ref{thmSTE},
        it follows that, for the extension $A(\cdot)$ to $[0, T+\delta)$,
        one has $\sup_{a \le t <T+\delta} \| B_{A(t)}\|_\infty <\infty$.
        Therefore $A(\cdot)$ is a strong solution on $[0, T+\delta)$.
          Hence the maximal interval of existence of a strong solution is $[0, \infty)$.
     \end{proof}

\subsection{Marini boundary conditions}

      The following lemma will be used to deduce Theorem \ref{thm1M} from
  Theorem \ref{thm1N}.

\begin{lemma}\label{lemnormal}
Suppose that $A \in C^2(M; \L^1 \otimes \kf)$. Then there exists
a function $g\in C^2( M ; K)$ such that
\beq
(A^g)_{norm} =0
\eeq
\end{lemma}

          \begin{proof} For a point $P \in \p M$ let $x_P(s), 0\le s <\epsilon$
 be the geodesic in $M$ starting at $P$ and normal to $\p M$ at $P$.
Thus $x_P'(0) = -  \en$, where $\en$ is the outward drawn unit normal
at $P$.   We may choose $\epsilon >0$ so small that  the map
$ \p M \times [0,2\epsilon)\ni P,s \rightarrow x_P(s)$ is a diffeomorphism onto a collar neighborhood $U$ of $\p M$ in $M$. Choose a function $h \in C_c^\infty([0,2\epsilon))$ such that
$h(s) = s$ on $[0, \epsilon)$. Define
\beq
g(y) =\begin{cases} &e^{h(s) \< A, \en\>_P} \ \ \text{if}\ y = x_P(s) \in U\\
         &I_\V  \ \ \text{if}\ \ y \in M-U
         \end{cases}               \notag
  \eeq
  Then $g$ is $C^2$ in $U$ and, since $g(y) = I_\V \equiv e_K$ in a neighborhood of the inner boundary of $U$, it follows that $g \in C^2(M;K)$. Moreover
  $ dg(x_P(s))/ds |_{s=0} = h'(s)|_{s=0} \< A, \en\>_P = \< A, \en\>_P$.
  Therefore
  \begin{align*}
  (A^g)_{norm}(P) &= g(P)^{-1} A_{norm} g(P) + g(P)^{-1} \< dg(P), \en\>\\
  &= A_{norm}(P)  - dg(x_P(s))/ds|_{s=0}\\
  &= 0.
  \end{align*}
\end{proof}

\begin{remark}{\rm The preceding lemma has an imprecise analog for Dirichlet boundary conditions. Suppose that $A$ is in $C^\infty(M)$ and that $B_{tan}=0$.
 Then, given a point $P \in \p M$,
there is a  smooth function $g:M\rightarrow K$ such that
$(A^g)_{tan} =0$ in some
neighborhood of $P$ in $\p M$. Indeed, the connection form $A_{tan}$
 on $\p M$ has curvature  form $B_{tan}$, which is zero.
 So $A_{tan}$ is locally, on $\p M$, a pure gauge.
 That is, there exists a smooth function $\phi$ on a neighborhood
     of $P \in \p M$ such that $ A_{tan} = \phi^{-1} d\phi$ on this neighborhood.
     Extend $\phi$ smoothly to a neighborhood $U$ in $ M$ for which
     $P \in U\cap \p M \subset$ domain  $\phi$  and define
     $ g = \phi^{-1}$ there. It is straightforward to verify then that $(A^g)_{tan} =0$
  on $U \cap \p M$ as asserted.
             Moreover, choosing $\phi(P) = e_K$ and
   $U$ small, one can ensure that $\phi$ takes its values in a
   contractible neighborhood
    of $e_K$ in $K$ and therefore $g$ can be extended to all of $M$.

      For a  nontrivial bundle over $M$ the boundary conditions $B_{norm} =0$
      and $B_{tan}=0$ are both well defined,
      as opposed to $A_{norm} =0$ and $A_{tan}=0$. This has been observed
      and used by W. Gryc, \cite{Gry}, in his work extending the
       no-section theorem of
      Narasimhan and Ramadas, \cite{NR}, to manifolds with boundary.
    }
    \end{remark}

 \bigskip
  \begin{proof}[Proof of Theorem \ref{thm1M}]
  Suppose that $A_0 \in C^2(M)$. By Lemma \ref{lemnormal} there
  exists a function $g \in C^2(M; K)$ such that $ \hat A_0 \equiv A_0^g$
  has normal component zero. Clearly $\hat A_0 \in C^1(M) \subset W_1(M)$.
  By Theorem \ref{thm1N} there exists a unique strong solution
  $\hat A(\cdot)$   to \eref{ymh10} on $[0,\infty)$ such that
  $\hat A(s)_{norm}=0$ for $s \ge 0$  and $\hat B(s)_{norm}=0$ for $s > 0$.
    Define $A(s) = \hat A(s)^{g^{-1}}$ for $s\ge 0$. Since $g \in C^2(M;K)$,
  $A(s)$ is again a strong solution and
  $B(s)_{norm} = (\hat B(s)^{g^{-1}})_{norm} =0$.
    Of course $A(s)_{norm}$ need not be zero for $s\ge 0$.
  However the uniqueness portion of Theorem \ref{thm1N} applies, showing that $A(\cdot)$ is the unique strong solution with $A(0) = A_0$ and $B(s)_{norm} =0$
  for $s >0$.
  \end{proof}

  % References

\bibliographystyle{amsplain}
\bibliography{ymh}

\end{document}